\documentclass[11pt]{amsart}
\usepackage{graphicx}
\usepackage{amsmath}
\usepackage{amssymb}
\usepackage{amsfonts}
\usepackage{verbatim}
\usepackage{scalefnt}
\usepackage{multirow}
\usepackage{mathtools}
\usepackage{float}
\restylefloat{figure}
\usepackage[margin=1.3in]{geometry}

\usepackage{pstricks,tikz}

\usepackage{fancyhdr}
\usepackage{hyperref}

\DeclareMathOperator*{\Int}{{\rm Int}}
\DeclareMathOperator*{\End}{{\rm End}}

\begin{document}

\def\COMMENT#1{}

\newtheorem{theorem}{Theorem}[section]
\newtheorem{lemma}[theorem]{Lemma}
\newtheorem{proposition}[theorem]{Proposition}
\newtheorem{corollary}[theorem]{Corollary}
\newtheorem{conjecture}[theorem]{Conjecture}
\newtheorem{claim}[theorem]{Claim}
\newtheorem{definition}[theorem]{Definition}
\newtheorem{question}[theorem]{Question}

\numberwithin{equation}{section}

\def\eps{{\varepsilon}}
\newcommand{\cP}{\mathcal{P}}
\newcommand{\cT}{\mathcal{T}}
\newcommand{\cL}{\mathcal{L}}
\newcommand{\ex}{\mathbb{E}}
\newcommand{\eul}{e}
\newcommand{\pr}{\mathbb{P}}



\title{THE ROBUST COMPONENT STRUCTURE OF DENSE REGULAR GRAPHS AND APPLICATIONS}
\title[The robust component structure of dense regular graphs]{The robust component structure of dense regular graphs and applications}
\author{Daniela K\"uhn, Allan Lo, Deryk Osthus and Katherine Staden}
\thanks{The research leading to these results was partially supported by the  European Research Council
under the European Union's Seventh Framework Programme (FP/2007--2013) / ERC Grant
Agreement n. 258345 (D.~K\"uhn and A.~Lo) and 306349 (D.~Osthus).}

\begin{abstract}
In this paper, we study the large-scale structure of dense regular graphs.
This involves the notion of robust expansion,
a recent concept which has already been used successfully to settle several longstanding problems.
Roughly speaking, a graph is robustly expanding if it still expands after the deletion of a small fraction of its vertices and edges.
Our main result allows us to harness the useful consequences of robust expansion even if the graph itself is not a robust
expander. It states that every dense regular graph can be partitioned into `robust components', each of which is a robust expander or a bipartite robust expander.
We apply our result to obtain (amongst others) the following.
\begin{itemize}
\item[(i)] We prove that whenever $\eps >0$, every sufficiently large $3$-connected $D$-regular graph on $n$ vertices with $D \geq (1/4 + \eps)n$ is Hamiltonian.
This asymptotically confirms the only remaining case of a conjecture raised independently by Bollob\'as and H\"aggkvist in the 1970s.
\item[(ii)] We prove an asymptotically best possible result on the circumference of dense regular graphs of given connectivity.
The $2$-connected case of this was conjectured by Bondy and proved by Wei.
\end{itemize}
\end{abstract}

\date{\today}
\maketitle

\section{Introduction}

\vspace{0.1cm}

\subsection{The robust component structure of dense regular graphs}~\label{intro}
Our main result states that any dense regular graph $G$ is the vertex-disjoint union of a bounded number of `robust components'.
Each such component has a strong expansion property that is highly `resilient' and almost all edges of $G$ lie inside these robust components. 
In other words, the result implies that the large scale structure of dense regular graphs is remarkably simple.
This can be applied e.g.~to Hamiltonicity problems in dense regular graphs.
Note that the structural information obtained in this way is quite different from that given by Szemer\'edi's regularity lemma.

The crucial notion in our partition is that of robust expansion.
This is a structural property which has close connections to Hamiltonicity.
Given a graph $G$ on $n$ vertices, $S \subseteq V(G)$ and $0 < \nu \leq \tau < 1$, we define the \emph{$\nu$-robust neighbourhood} $RN_{\nu,G}(S)$ of $S$ to be the set of all those vertices of $G$ with at least $\nu n$ neighbours in $S$.
We say $G$ is a \emph{robust $(\nu,\tau)$-expander} if, for every $S \subseteq V(G)$ with $\tau n \leq |S| \leq (1-\tau)n$,
we have that $|RN_{\nu,G}(S)| \geq |S| + \nu n$.

There is an analogous notion of robust outexpansion for digraphs. This was first introduced in~\cite{kot}
and has been instrumental in proving several longstanding conjectures. For example,
K\"uhn and Osthus~\cite{KOKelly} recently settled a conjecture of Kelly from 1968 (for large tournaments)
by showing that every sufficiently large dense regular robust outexpander has a Hamilton decomposition.
Another example is the recent proof~\cite{sumner1,sumner2} of Sumner's universal tournament conjecture from 1971.

Our main result allows us to harness the useful consequences of robust expansion even if the graph itself is not a robust
expander. For this, we introduce the additional notion of `bipartite robust expanders'.
Let $G$ be a bipartite graph with vertex classes $A$ and $B$. Then clearly $G$ is not a robust expander.
However, we can obtain a bipartite analogue of robust expansion by only considering sets $S \subseteq A$ with $\tau|A| \leq |S| \leq (1-\tau)|A|$.
This notion extends in a natural way to graphs which are `close to bipartite'.

Roughly speaking, our main result (Theorem~\ref{structure}) implies the following.
\begin{itemize}
\item[($\dagger$)]
For all $ r \in \mathbb{N}$ and all $\varepsilon >0$, any sufficiently large $D$-regular graph on $n$ vertices with $D \ge (\frac{1}{r+1}+ \varepsilon) n$ has a vertex partition into at most $r$ robust expander components and bipartite robust expander components, so that the number of edges between these is $o(n^2)$.
\end{itemize}
We give a formal statement of this in Section~\ref{sec:struct}. 
In Section~\ref{almostregular} we obtain a generalisation to almost regular graphs.
(Here, $G$ is `almost regular' if $\Delta(G)-\delta(G) = o(n)$.)

In the special case of dense vertex-transitive graphs (which are always regular), Christofides, Hladk\'y and M\'ath\'e~\cite{chm}
introduced a partition into `iron connected components'.
(Iron connectivity is closely related to robust expansion.)
They applied this to resolve the dense case of a question of 
Lov\'asz~\cite{lovasz} on Hamilton paths (and cycles) in vertex-transitive graphs.
It would be very interesting to obtain a similar partition result for further classes of graphs.
In particular, it might be possible to generalise Theorem~\ref{structure} to sparser graphs.

In the current paper, we apply $(\dagger)$ to give an approximate solution to a longstanding conjecture on Hamilton cycles in regular graphs (Theorem~\ref{main}) as well as an asymptotically optimal result on the circumference of dense regular graphs of given connectivity (Theorem~\ref{tconnected}).
We are also confident that our robust partition result will have applications to other problems.

\subsection{An application to Hamilton cycles in regular graphs} \label{applications}

Consider the classical result of Dirac that every graph on $n \geq 3$ vertices with minimum degree at least $n/2$ contains a Hamilton cycle.
Suppose we wish to strengthen this by reducing the degree threshold at the expense of introducing some other condition(s).
The two extremal examples for Dirac's theorem (i.e. the disjoint union of two cliques and the almost balanced complete bipartite graph) make it natural to consider regular graphs with some connectivity property, see e.g.~the recent survey of Li~\cite{li} and handbook article of Bondy~\cite{bondyhbook}.

In particular, Szekeres (see~\cite{jackson}) asked for which $D$ every 2-connected $D$-regular graph $G$ on $n$ vertices is Hamiltonian. 
Jackson~\cite{jackson} showed that $D \geq n/3$ suffices.
This improved earlier results of Nash-Williams~\cite{nashwilliams}, Erd\H{o}s and Hobbs~\cite{erdoshobbs} and Bollob\'as and Hobbs~\cite{bollobashobbs}.
Hilbig~\cite{hilbig} improved the degree condition to $n/3-1$,
unless $G$ is the Petersen graph or another exceptional graph. As discussed later on in this section, this bound is best possible.

Bollob\'as~\cite{egt} as well as H\"aggkvist (see~\cite{jackson}) independently made the natural and far more general conjecture that any $t$-connected regular graph on $n$ vertices with degree at least $n/(t+1)$ is Hamiltonian.%
\COMMENT{Bollob\'as's conjecture was stronger, with $D \geq n/(t+1)-1$.}
However, the following counterexample (see Figure~\ref{fig:exactex}(i)), due to Jung~\cite{jung} and independently Jackson, Li and Zhu~\cite{jlz}, disproves this conjecture for $t>3$.

For $m$ divisible by four, construct $G$ as follows. Let $C_1, C_2$ be two disjoint copies of $K_{m+1}$ and let $A,B$ be two disjoint independent sets of orders $m,m-1$ respectively.
Add every edge between $A$ and $B$.
Add a set of $m/2$ independent edges from each of $C_1$ and $C_2$ to $A$ so that together these edges form a matching of size $m$.
Delete $m/4$ independent edges in each of $C_1, C_2$ so that $G$ is $m$-regular.
Then $G$ has $4m+1$ vertices and is $m/2$-connected.
However $G$ is not Hamiltonian since $G \setminus A$ has $|A|+1$ components (in other words, $G$ is not 1-tough).%
\COMMENT{Note that $G$ has a robust partition $C_1,C_2, A \cup B$ (so $(k,\ell) = (2,1)$ in the definition of robust partitions in Section~1).
Discussion of extremal examples for all values of $m$ after end of document}

\begin{center}
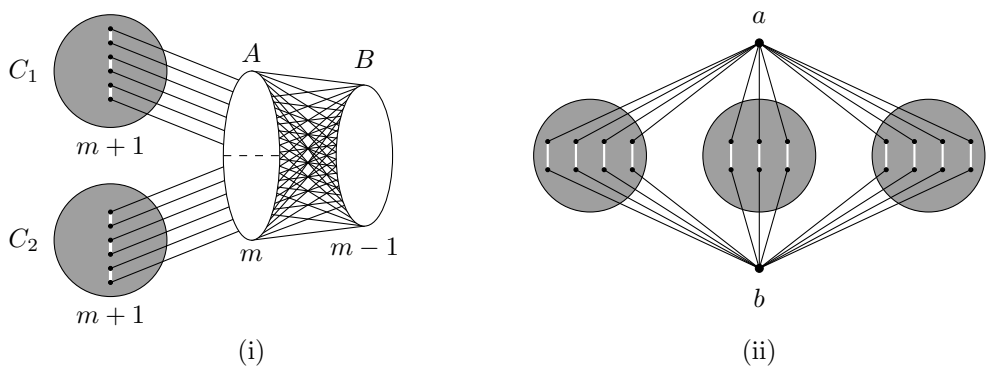
\begin{figure}[H]
\begin{tikzpicture}[scale=0.75]
\tikzstyle{every node}=[font=\small]

\begin{scope}
\foreach \x in {-1.5,-1,-0.5,0,0.5,1,1.5}
\foreach \y in {-1.25,-0.75,...,1.25}
{\draw (3,\x) -- (5,\y);}

\draw[fill=gray!75] (0.5,1.5) circle (1);
\draw[fill=gray!75] (0.5,-1.5) circle (1);

\draw (0.5,2.25) -- (3,1.25);
\draw (0.5,2) -- (3,1);
\draw (0.5,1.75) -- (3,0.75);
\draw (0.5,1.5) -- (3,0.5);
\draw (0.5,1.25) -- (3,0.25);
\draw (0.5,1) -- (3,0);

\draw (0.5,-2.25) -- (3,-1.25);
\draw (0.5,-2) -- (3,-1);
\draw (0.5,-1.75) -- (3,-0.75);
\draw (0.5,-1.5) -- (3,-0.5);
\draw (0.5,-1.25) -- (3,-0.25);
\draw (0.5,-1) -- (3,0);

\draw[very thick,color=white] (0.5,2.25) -- (0.5,2);
\draw[very thick,color=white] (0.5,1.75) -- (0.5,1.5);
\draw[very thick,color=white] (0.5,1.25) -- (0.5,1);

\draw[very thick,color=white] (0.5,-2.25) -- (0.5,-2);
\draw[very thick,color=white] (0.5,-1.75) -- (0.5,-1.5);
\draw[very thick,color=white] (0.5,-1.25) -- (0.5,-1);

\foreach \w in {2.25,2,1.75,1.5,1.25,1}
{
\draw[fill=black] (0.5,\w) circle (1pt);
\draw[fill=black] (0.5,-\w) circle (1pt);
}

\draw[fill=white] (3,0) ellipse (0.5 and 1.5);
\draw[fill=white] (5,0) ellipse (0.5 and 1.25);
\draw[style=dashed] (2.5,0) -- (3.5,0);

\node at (5,-1.1) [draw=none,
label=below:{$m-1$}] (){};

\node at (3,-1.3) [draw=none,
label=below:{$m$}] (){};

\node at (0.5,0.7) [draw=none,
label=below:{$m+1$}] (){};

\node at (0.5,-2.3) [draw=none,
label=below:{$m+1$}] (){};

\node at (-0.4,1.5) [draw=none,
label=left:{$C_1$}] (){};

\node at (-0.4,-1.5) [draw=none,
label=left:{$C_2$}] (){};

\node at (3,1.3) [draw=none,
label=above:{$A$}] (){};

\node at (5,1.2) [draw=none,
label=above:{$B$}] (){};

\node at (3,-2.9) [draw=none,
label=below:{(i)}] (){};
\end{scope}

\begin{scope}[xshift=12cm]
\tikzstyle{every node}=[font=\small]

\draw[fill=black] (0,2) circle (2pt);
\draw[fill=black] (0,-2) circle (2pt);

\node at (0,2) [draw=none,
label=above:{$a$}] (){};
\node at (0,-2) [draw=none,
label=below:{$b$}] (){};

\draw[fill=gray!75] (-3,0) circle (1);
\draw[fill=gray!75] (0,0) circle (1);
\draw[fill=gray!75] (3,0) circle (1);

\foreach \z in {-3,3}
\foreach \w in {-0.75,-0.25,0.25,0.75}
{
\draw (0,2) -- (\z+\w,0.25);
\draw (\z+\w,-0.25) -- (0,-2);
}

\foreach \z in {-3,3}
\foreach \w in {-0.75,-0.25,0.25,0.75}
{
\draw[thick,color=white] (\z+\w,0.25) -- (\z+\w,-0.25);
\draw[fill=black] (\z+\w,0.25) circle (1pt);
\draw[fill=black] (\z+\w,-0.25) circle (1pt);
}

\foreach \w in {-0.5,0,0.5}
{
\draw[thick,color=white] (\w,0.25) -- (\w,-0.25);
\draw (0,2) -- (\w,0.25);
\draw (\w,-0.25) -- (0,-2);
\draw[fill=black] (\w,0.25) circle (1pt);
\draw[fill=black] (\w,-0.25) circle (1pt);
}

\node at (0,-2.9) [draw=none,
label=below:{(ii)}] (){};
\end{scope}

\end{tikzpicture}
\caption{Extremal examples for Conjecture~\ref{exact}.}\label{fig:exactex}
\end{figure}
\end{center}


Jackson, Li and Zhu~\cite{jlz} believe that the conjecture of Bollob\'as and H\"aggkvist is true in the remaining open case when $t=3$.%
\COMMENT{original: $\forall k \geq 4$ $\forall$ 3-connected $k$-regular $G$ on $\leq 4k$ vertices, $G$ is Hamiltonian. $13$ is the least $n$ such that $\lceil n/4 \rceil \geq 4$}

\begin{conjecture}\label{exact}
Let $G$ be a $3$-connected $D$-regular graph on $n \geq 13$ vertices such that $D \geq n/4$.
Then $G$ contains a Hamilton cycle. 
\end{conjecture}

The $3$-regular graph obtained from the Petersen graph by replacing one vertex with a triangle shows that the conjecture does not hold for $n=12$.
The graph in Figure~\ref{fig:exactex}(i) is extremal and the bound on $D$ is tight.

As mentioned earlier, there exist non-Hamiltonian $2$-connected regular graphs on $n$ vertices with degree close to $n/3$ (see Figure~\ref{fig:exactex}(ii)).
Indeed, we can construct such a graph $G$ as follows.
Start with three disjoint cliques on $3m$ vertices each.
In the $i$th clique choose disjoint sets $A_i$ and $B_i$ with $|A_i|=|B_i|$ and $|A_1|=|A_3|=m$ and $|A_2|=m-1$. Remove a perfect matching between $A_i$ and $B_i$ for each $i$. Add two new vertices $a$ and $b$, where $a$ is connected to all vertices in the sets $A_i$ and $b$ is connected to all vertices in all the sets $B_i$.
Then $G$ is a $(3m-1)$-regular $2$-connected graph on $n=9m+2$ vertices.
However, $G$ is not Hamiltonian because $G \setminus \lbrace a,b \rbrace$ has three components.
Therefore none of the conditions -- degree, order or connectivity -- of Conjecture~\ref{exact} can be relaxed.

There have been several partial results in the direction of Conjecture~\ref{exact}.
Fan~\cite{fan} and Jung~\cite{jung} independently showed that every $3$-connected $D$-regular graph contains a cycle of length at least $3D$, or a Hamilton cycle.
Li and Zhu~\cite{lizhu} proved Conjecture~\ref{exact} in the case when $D \geq 7n/22$%
\COMMENT{there's a mistake in Li's survey where this result is quoted -- it has $3n/22$ which is of course smaller than $n/4$!} and Broersma, van den Heuvel, Jackson and Veldman~\cite{bhjv} proved it for $D \geq 2(n+7)/7$.
In~\cite{jlz} it is proved that, if $G$ satisfies the conditions of the conjecture, 
any longest cycle in $G$ is dominating provided that $n$ is not too small.
(Here, a subgraph $H$ of a graph $G$ is \emph{dominating} if $G \setminus V(H)$ is an independent set.) 
By considering robust partitions, we are able to prove an approximate version of the conjecture.

\begin{theorem} \label{main}
For all $\eps > 0$, there exists $n_0 \in \mathbb{N}$ such that every $3$-connected $D$-regular graph on $n \geq n_0$ vertices with $D \geq (1/4+\eps)n$ is Hamiltonian.
\end{theorem}

In fact, if $D$ is at least a little larger than $n/5$ but $G$ is not Hamiltonian we also determine the asymptotic structure of $G$ (see Theorem~\ref{stability}).
In a recent paper~\cite{hamconnex} we use this to prove the exact version of Conjecture~\ref{exact} for large~$n$.
Note that~\cite{hamconnex} does not supersede this paper but rather uses it as an essential tool.

There are also natural analogues of the above results and questions for directed graphs.
Here, a $D$-regular directed graph is such that every vertex has both in- and outdegree equal to $D$.
An \emph{oriented graph} is a digraph without $2$-cycles.

\begin{conjecture}\
\begin{itemize}
\item[(a)] For each $D>2$, every $D$-regular oriented graph $G$ on $n$ vertices with $D \geq (n-1)/4$ is Hamiltonian.
\item[(b)] Every strongly $2$-connected $D$-regular digraph on $n$ vertices with $D\ge n/3$ is Hamiltonian.
\item[(c)] For each $D>2$, every $D$-regular
strongly $2$-connected oriented graph $G$ on $n$ vertices with $D \ge n/6$ 
is Hamiltonian.
\end{itemize}
\end{conjecture}
(a) was conjectured by Jackson~\cite{jacksonconj}, (b) and (c) were raised in~\cite{KOsurvey},
which also contains a more detailed discussion of these conjectures.


\subsection{An application to the circumference of regular graphs}

More generally, we also consider the circumference of dense regular graphs of given connectivity.
Bondy~\cite{bondy} conjectured that, for $r \geq 3$, every sufficiently large $2$-connected $D$-regular graph $G$ on $n$ vertices with $D \geq n/r$ has circumference $c(G) \geq 2n/(r-1)$.
(Here the circumference $c(G)$ of $G$ is the length of the longest cycle in $G$.)
This was confirmed by Wei~\cite{wei}, who proved the conjecture for all $n$ and in fact showed that $c(G) \geq 2n/(r-1) + 2(r-3)/(r-1)$, which is best possible.
We are able to extend this (asymptotically) to $t$-connected dense regular graphs.
\begin{theorem} \label{tconnected}
Let $t,r \in \mathbb{N}$.
For all $\eps > 0$ there exists $n_0 \in \mathbb{N}$ such that the following holds.
Whenever $G$ is a $t$-connected $D$-regular graph on $n \geq n_0$ vertices where $D \geq (1/r+\eps)n$, the circumference of $G$ is at least $\min \lbrace t/(r-1) ,1- \eps \rbrace  n$.
\end{theorem}

This is asymptotically best possible. 
Indeed, in Proposition~\ref{bestposs} we show that, for every $t,r \in \mathbb{N}$, there are infinitely many $n$ such that there exists a graph $G$ on $n$ vertices which
is $((n-t)/(r-1)-1)$-regular%
    \COMMENT{DK: replaced $n/r$-regular by $((n-t)/(r-1)-1)$-regular}
and $t$-connected with $c(G) \leq tn/(r-1) + t$.
Moreover, as discussed above, the first extremal example in Figure~\ref{fig:exactex} shows that in general $\min \lbrace t/(r-1) ,1- \eps \rbrace  n$ cannot be replaced by $\min \lbrace t/(r-1) ,1 \rbrace  n$.

Theorem~\ref{tconnected} shows that the conjecture of Bollob\'as and H\"aggkvist is in fact close to being true after all -- any $t$-connected regular graph with degree slightly higher than $n/(t+1)$ contains an almost spanning cycle.

\subsection{An application to bipartite regular graphs}~\label{sec:bip}
One can consider similar questions about dense regular bipartite graphs.
H\"aggkvist~\cite{haggkvist} conjectured that every $2$-connected $D$-regular bipartite graph on $n$ vertices with $D \geq n/6$ is Hamiltonian.
If true, this result would be best possible.
Indeed, it was essentially verified by Jackson and Li~\cite{jacksonli} who proved it in the case when $D \geq (n+38)/6$.
Recently, Li~\cite{li} conjectured a bipartite analogue of Conjecture~\ref{exact}, i.e. that every $3$-connected $D$-regular bipartite graph on $n$ vertices with $D \geq n/8$ is Hamiltonian.

Restricting to bipartite graphs strengthens the structural information implied by our main result ($\dagger$) considerably.
So it seems likely that one can use our partition result to make progress towards these and other related conjectures.

One might ask if a bipartite analogue of the conjecture of Bollob\'as and H\"aggkvist
holds, i.e. whether every $t$-connected $D$-regular bipartite graph on $n$ vertices with $D \geq n/2(t+1)$ contains a Hamilton cycle.
However, as in the general case, it turns out that this is false for $t>3$.
Indeed, for each $t \geq 2$ and infinitely many $D \in \mathbb{N}$, Proposition~\ref{bipbestposs} guarantees a $D$-regular 
bipartite graph $G$ on $8D+2$ vertices that is $t$-connected and contains no Hamilton cycle.
(This observation generalises one from~\cite{li}, who considered the case when $t=3$.)

As in the general case, one may also consider the circumference of dense regular bipartite graphs.
Indeed, the argument for Theorem~\ref{tconnected} yields the following bipartite analogue.
Again, it is asymptotically best possible (see Proposition~\ref{bipbestposs}(i)).

\begin{theorem} \label{biptconnected}
Let $t,r \in \mathbb{N}$, where $r$ is even.
For all $\eps > 0$ there exists $n_0 \in \mathbb{N}$ such that the following holds.
Whenever $G$ is a $t$-connected $D$-regular bipartite graph on $n \geq n_0$ vertices where $D \geq (1/r + \eps)n$, the circumference of $G$
is at least
$
\min \lbrace 2tn/(r-2), n \rbrace  - \eps n.
$
\end{theorem}


\subsection{Organisation of the paper and sketch proof of Theorem~\ref{main}}

This paper is organised as follows.
In the remainder of this section we sketch the proof of Theorem~\ref{main}.
Section~\ref{notation} lists some notation which will be used throughout.
In Section~\ref{sec:struct} we state 
our robust partition result (Theorem~\ref{structure}) which formalises ($\dagger$).
We prove it in Section~\ref{sec:proof}, which also contains a sketch of the argument.
In Section~\ref{almostregular} we derive a version of Theorem~\ref{structure} for almost regular graphs.
In Section~\ref{sec:someresults} we show how to find suitable path systems covering the robust components obtained from Theorem~\ref{structure}.
These tools are then used in Section~\ref{sec:proofmain} to prove Theorem~\ref{main} and
used in Section~\ref{sec:prooflink} to prove Theorems~\ref{tconnected} and~\ref{biptconnected}.

In order to see how our partition result ($\dagger$) may be applied, we now briefly outline the argument used to prove Theorem~\ref{main}.

Let $\eps > 0$ and suppose that $G$ is a $3$-connected $D$-regular graph on $n$ vertices, where $D \geq (1/4+\eps)n$. 
Now ($\dagger$) gives us a robust partition $\mathcal{V}$ of $G$ containing exactly $k$ robust expander components and $\ell$ bipartite robust expander components,
where $k+\ell \le 3$.

However, Theorem~\ref{structure} actually gives the stronger bound that $k+2\ell \le 3$, so there are only five possible choices of $(k,\ell)$ (see Proposition~\ref{fewstructs}).
Assume for simplicity that $\mathcal{V}$ consists of three robust expander components $G_1, G_2, G_3$. So $(k,\ell)=(3,0)$.
The result of~\cite{kot} mentioned in Section~\ref{intro} implies that $G_i$ contains a Hamilton cycle for $i=1,2,3$.
In fact, it can be used to show (see Corollary~\ref{robexppaths}) that $G_i$ is Hamilton $p$-linked for each bounded $p$.
(Here a graph $G$ is \emph{Hamilton $p$-linked} if, whenever $x_1, y_1, \ldots, x_p, y_p$ are distinct vertices, there exist vertex-disjoint paths
$P_1, \ldots, P_p$ such that $P_j$ connects $x_j$ to $y_j$, and such that together the paths $P_1, \ldots, P_p$ cover all vertices of $G$.)
This means that the problem of finding a Hamilton cycle in $G$ can be reduced to finding only a suitable set of \emph{external edges}, where an edge is external if it has endpoints in different $G_i$.
We use the assumption of $3$-connectivity to find these external edges (in Section~\ref{sec:proofmain}).

The cases where $\ell>0$ are more difficult since a bipartite graph does not contain a Hamilton cycle if it is not balanced. 
So as well as suitable external edges, we need to find some `balancing edges' incident to the bipartite robust expander component. 
(Note that if $\ell>0$ we must have $\ell=1$ and $k\le 1$.)
Suppose for example that we have $k = \ell=1$ and that we have a bipartite robust expander component with vertex classes $A,B$ where $|A|=|B|+1$,
as well as a robust expander component $X$ and an edge $e$ joining $A$ to $X$ and an edge $f$ joining $B$ to $X$, where $e$ and $f$ are disjoint
(so $e$ and $f$ are external edges). 
Then one possible set of balancing edges consists e.g.~of two further external edges incident to $A$. 
Another example would be one edge inside $A$.
These balancing edges are guaranteed by our assumption that $G$ is regular.
We construct them in Section~\ref{sec:proofmain}.

\section{Notation}\label{notation}
For $A \subseteq V(G)$, complements are always taken within the entire graph $G$, so that $\overline{A} := V(G) \setminus A$.
Given $A \subseteq V(G)$, we write $N(A) := \bigcup_{a \in A}{N(a)}$.
For $x \in V(G)$ and $A \subseteq V(G)$ we write $d_A(x)$ for the number of edges $xy$ with $y \in A$.
For $A,B \subseteq V(G)$, we write $e(A,B)$ for the number of edges with exactly one endpoint in $A$ and one endpoint in~$B$. (Note that $A,B$ are not necessarily disjoint.)
Define $e'(A,B) := e(A,B) + e(A \cap B)$. So $e'(A,B) = \sum_{a \in A}d_B(a) = \sum_{b \in B}d_A(b)$ and if $A,B$ are disjoint then $e'(A,B)=e(A,B)$.
For a digraph $G$, we write $\delta^0(G)$ for the minimum of its minimum indegree and minimum outdegree.

For distinct $x,y \in V(G)$ and a path $P$ with endpoints $x$ and $y$, we sometimes write $P = xPy$ to emphasise this.
Given disjoint subsets $A,B$ of $V(G)$, we say that $P$ is an \emph{$AB$-path} if $P$ has one 
endpoint in $A$ and one endpoint in $B$.
We call a vertex-disjoint collection of paths a \emph{path system}.
We will often think of a path system $\mathcal{P}$ as a graph with edge set $\bigcup_{P \in \mathcal{P}}E(P)$, so that e.g. $V(\mathcal{P})$ is the union of the vertex sets of each path in $\mathcal{P}$.

Throughout we will omit floors and ceilings where the argument is unaffected. The constants in the hierarchies used to state our results are chosen from right to left.
For example, if we claim that a result holds whenever $0<1/n\ll a\ll b\ll c\le 1$ (where $n$ is the order of the graph or digraph), then 
there is a non-decreasing function $f:(0,1]\to (0,1]$ such that the result holds
for all $0<a,b,c\le 1$ and all $n\in \mathbb{N}$ with $b\le f(c)$, $a\le f(b)$ and $1/n\le f(a)$. 
Hierarchies with more constants are defined in a similar way. 


\section{Robust partitions of regular graphs} \label{sec:struct}

In this section we list the definitions which are required to state our main result.
For a graph $G$ on $n$ vertices, $0 < \nu < 1$ and $S \subseteq V(G)$, recall that the \emph{$\nu$-robust neighbourhood} $RN_{\nu, G}(S)$  of $S$ to be the set of all those vertices with at least $\nu n$ neighbours in $S$.
Also, recall that for $0 < \nu \leq \tau < 1$
we say that $G$ is a \emph{robust $(\nu, \tau)$-expander} if, for all sets $S$ of vertices satisfying $\tau n \leq |S| \leq (1-\tau)n$, we have that $|RN_{\nu, G}(S) | \geq |S| + \nu n$.%
\COMMENT{We already defined this in the introduction but I think it's slightly odd not to have the definition here, with all the others.}
For $S \subseteq X  \subseteq V(G)$ we write $RN_{\nu,X}(S) := RN_{\nu,G[X]}(S)$.

We now introduce the concept of `bipartite robust expansion'.
Let $0 < \nu \leq \tau < 1$. Suppose that $H$ is a (not necessarily bipartite) graph on $n$ vertices and that $A,B$ is a partition of $V(H)$.
We say that $H$ is a \emph{bipartite robust} $(\nu, \tau)$-\emph{expander with bipartition $A,B$} if every $S \subseteq A$ with $\tau |A| \leq |S| \leq (1-\tau)|A|$ satisfies $|RN_{\nu, H}(S) \cap B| \geq |S| + \nu n$.
Note that the order of $A$ and $B$ matters here.
We do not mention the bipartition if it is clear from the context.%
\COMMENT{we don't want a bipartite robust expander to necessarily be bipartite and balanced (because we want to decompose $G$ into (bipartite) robust expander components, and these will only be close to bipartite. Note that $RN_{\nu,H}(S)$ is the set of vertices with at least $\nu|H|$ neighbours in $S$, rather than $\nu |A|$. Otherwise there is confusion because the robust neighbourhood is defined differently depending on whether we are interested in robust expanders or bipartite robust expanders.}

Note that for $0 < \nu' \leq \nu \leq \tau \leq \tau' < 1$, any robust $(\nu,\tau)$-expander is also a robust $(\nu',\tau')$-expander (and the analogue holds in the bipartite case).

Given $0 < \rho < 1$, we say that $U \subseteq V(G)$ is a \emph{$\rho$-component} of a graph $G$ on $n$ vertices if $|U| \geq \sqrt{\rho}n$ and $e_G(U,\overline{U}) \leq \rho n^2$. 
Let $0 < \rho \leq \nu \leq 1$.
Let $G$ be a graph containing a $\rho$-component $U$ and let $S \subseteq U$.
We say that $S$ is \emph{$\nu$-expanding in $U$} if $|RN_{\nu,U}(S)| \geq |S| + \nu|U|$, and \emph{non-$\nu$-expanding} otherwise. 
So if $G[U]$ is a robust $(\nu,\tau)$-expander for some $\tau$, then all $S \subseteq U$ satisfying $\tau|U| \leq |S| \leq (1-\tau)|U|$ are $\nu$-expanding in $U$.

Recall that $\overline{U_1} = V(G) \setminus U_1$ and similarly for $U_2$.
Suppose that $G$ is a graph on $n$ vertices
and that $U \subseteq V(G)$.
We say that $G[U]$ is \emph{$\rho$-close to bipartite (with bipartition $U_1, U_2$)} if 
\begin{itemize}
\item[(C1)] $U$ is the union of two disjoint sets $U_1$ and $U_2$ with $|U_1|,|U_2| \geq \sqrt{\rho}n$;
\item[(C2)] $||U_1| - |U_2|| \leq \rho n$;
\item[(C3)] $e(U_1, \overline{U_2}) + e(U_2,\overline{U_1}) \leq \rho n^2$.
\end{itemize}

\noindent
Note that (C1) and (C3) together imply that $U$ is a $\rho$-component.
Suppose that $G$ is a graph on $n$ vertices
and that $U \subseteq V(G)$.
Let $0< \rho \leq \nu \leq \tau < 1$.
We say that $G[U]$ is a \emph{$(\rho,\nu,\tau)$-robust expander component of $G$} if
\begin{itemize}
\item[(E1)] $U$ is a $\rho$-component;
\item[(E2)] $G[U]$ is a robust $(\nu,\tau)$-expander.
\end{itemize}
We say that $G[U]$ is a \emph{bipartite $(\rho,\nu,\tau)$-robust expander component (with bipartition $A,B$) of $G$} if
\begin{itemize}
\item[(B1)] $G[U]$ is $\rho$-close to bipartite with bipartition $A,B$;
\item[(B2)] $G[U]$ is a bipartite robust $(\nu,\tau)$-expander with bipartition $A,B$.
\end{itemize}
We say that $U$ is a \emph{$(\rho,\nu,\tau)$-robust component} if it is either a $(\rho,\nu,\tau)$-robust expander component or a bipartite $(\rho,\nu,\tau)$-robust expander component.

Our main result states that any sufficiently dense regular graph has a partition into (bipartite) robust expander components.
Let $k,\ell,D \in \mathbb{N}$ and $0 < \rho \leq \nu \leq \tau < 1$.
Given a $D$-regular graph $G$ on $n$ vertices, we say that $\mathcal{V}$ is a \emph{robust partition of $G$ with parameters $\rho,\nu,\tau,k,\ell$} if the following conditions hold.
\begin{itemize}
\item[(D1)] $\mathcal{V} = \lbrace V_1, \ldots, V_k, W_1, \ldots, W_\ell \rbrace$ is a partition of $V(G)$;
\item[(D2)] for all $1 \leq i \leq k$, $G[V_i]$ is a $(\rho,\nu,\tau)$-robust expander component of $G$;
\item[(D3)] for all $1 \leq j \leq \ell$, there exists a partition $A_j,B_j$ of $W_j$ such that $G[W_j]$ is a bipartite $(\rho,\nu,\tau)$-robust expander component with respect to $A_j,B_j$;
\item[(D4)] for all $X,X' \in \mathcal{V}$ and all $x \in X$, we have $d_{X}(x) \geq d_{X'}(x)$. In particular, $d_X(x) \geq D/m$, where $m := k+\ell$;
\item[(D5)] for all $1 \leq j \leq \ell$ we have $d_{B_j}(u) \geq d_{A_j}(u)$ for all $u \in A_j$ and $d_{A_j}(v) \geq d_{B_j}(v)$ for all $v \in B_j$; in particular, $\delta(G[A_j,B_j]) \geq D/2m$;
\item[(D6)] $k + 2\ell \leq \left\lfloor (1+\rho^{1/3})n/D \right\rfloor$;
\item[(D7)] for all $X \in \mathcal{V}$, all but at most $\rho n$ vertices $x \in X$ satisfy $d_X(x) \geq D - \rho n$.%
\COMMENT{could remove (D7). can essentially derive it from the fact that we have a partition into $\rho$-components. It's used in the four cliques case.}
\end{itemize}

As we shall see, (D6) can be derived from (D1)--(D5) but it is useful to state it explicitly.
Our main result is the following theorem, which we prove in the next section.
As mentioned in the introduction, we can use Theorem~\ref{structure} to derive a version for almost regular graphs (see Section~\ref{almostregular}).

\begin{theorem} \label{structure}
For all $\alpha,\tau > 0$ and every non-decreasing function $f : (0,1) \rightarrow (0,1)$, there exists $n_0 \in \mathbb{N}$ such that the following holds.
For all $D$-regular graphs $G$ on $n \geq n_0$ vertices where $D \geq \alpha n$, there exist $\rho,\nu$ with $1/n_0 \leq \rho \leq \nu \leq \tau$; $\rho \leq f(\nu)$ and $1/n_0 \leq f(\rho)$, and $k,\ell \in \mathbb{N}$ such that $G$ has a robust partition $\mathcal{V}$ with parameters $\rho,\nu,\tau,k,\ell$.
\end{theorem}

When the degree of $G$ is large, (D6) implies that there are only a small number of possible choices for $k$ and $\ell$.

\begin{proposition} \label{fewstructs}
Let $n,D \in \mathbb{N}$ and suppose that $0 < 1/n \ll \rho \ll \nu \ll \tau \ll 1/r < 1$ and $\rho^{1/3} \leq \eps/2$. Let $G$ be a $D$-regular graph on $n$ vertices where $D \geq (1/r + \eps) n$ and let $\mathcal{V}$ be a robust partition of $G$ with parameters $\rho,\nu,\tau,k,\ell$.
Then $k+2\ell \leq r-1$ and so $\ell \leq \lfloor (r-1)/2 \rfloor$ and $k \leq r-1-2\ell$.
In particular,
\begin{itemize}
\item[(i)] if $r=4$ then $(k,\ell) \in \mathcal{S}$, where $\mathcal{S} := \lbrace (1,0), (2,0), (3,0), (0,1), (1,1) \rbrace$;
\item[(ii)] if $r = 5$ then $(k,\ell) \in \mathcal{S} \cup \lbrace (4,0), (2,1), (0,2) \rbrace$.
\end{itemize}
\end{proposition}

\begin{proof}
It suffices to show that $k+2\ell \leq r-1$.
By (D6) and our assumption that $\rho^{1/3} \leq \eps/2$ we have
$$
k + 2\ell \leq \left\lfloor \frac{1+\eps/2}{1/r + \eps} \right\rfloor = \left\lfloor \frac{r+r\eps/2}{1 + r\eps} \right\rfloor = r-1,
$$
as required.
\end{proof}

We will prove Theorem~\ref{main} separately for each of these cases in Proposition~\ref{fewstructs}(i).
Proposition~\ref{fewstructs} is the only point of the proof where we need the full strength of the degree condition $D \geq (1/4 + \eps)n$.
Furthermore, Proposition~\ref{fewstructs}(ii) implies that a $\lceil n/4 \rceil$-regular graph could have any of the structures specified by (i), as well as $(k,\ell) = (4,0),(2,1),(0,2)$.
Note also that Figure~\ref{fig:exactex}(i) has $(k,\ell) = (2,1)$ and Figure~\ref{fig:exactex}(ii) has $(k,\ell) = (3,0)$.


\section{The proof of Theorem~\ref{structure}} \label{sec:proof}

We begin by giving a brief sketch of the argument.

\subsection{Sketch proof of Theorem~\ref{structure}}

The basic proof strategy is to successively refine an appropriate partition of $G$.
So let $G$ be a $D$-regular graph on $n$ vertices, where $D$ is linear in $n$. Suppose that $G$ is not a (bipartite) robust expander.
Then $V(G)$ contains a set $S$ such that $N$ is not much larger than $S$, where $N := RN_{\nu,G}(S)$ for appropriate $\nu$.
Consider a minimal $S$ with this property. 
Since $G$ is regular, $N$ cannot be significantly smaller than $S$.
One can use this to show that there are very few edges between $S \cup N$ and $X := V(G) \setminus (S \cup N)$.
Moreover, one can show that $S$ and $N$ are either almost identical or almost disjoint.
In the former case, $G[S \cup N]$ is a robust expander and in the latter $G[S \cup N]$ is close to a bipartite robust expander.
So in both cases, $S \cup N$ is a (bipartite) robust expander component.
Similarly, if $X$ is non-empty, it is either a (bipartite) robust expander component or we can partition it further along the above lines.
In this way, we eventually arrive at the desired partition.

\subsection{Preliminary observations}

We will often use the following simple observation about $\rho$-components.

\begin{lemma} \label{comp}
Let $n,D \in \mathbb{N}$ and $\rho,\rho',\gamma > 0$ such that $\rho \leq \rho'$ and $\gamma \geq \rho+\rho'$. Let $G$ be a $D$-regular graph on $n$ vertices and let $U$ be a $\rho$-component of $G$. Then
\begin{itemize}
\item[(i)] $|U| \geq D - \sqrt{\rho}n$;
\item[(ii)] if $W,W'$ is a partition of $U$ and $W$ is a $\rho'$-component of $G$, then $e(W',\overline{W'}) \leq \gamma n^2$;
\item[(iii)] if $D \geq 2\sqrt{\rho'}n$, then $U$ is a $\rho'$-component of $G$.
\end{itemize}
Let $X \subseteq V(G)$ have bipartition $X_1, X_2$ such that $G[X]$ is $\rho$-close to bipartite with bipartition $X_1,X_2$. 
Then
\begin{itemize}
\item[(iv)] $|X_1|,|X_2|\geq D - 2\sqrt{\rho}n$;
\item[(v)] if $D \geq 3\sqrt{\rho'}n$, then $X$ is $\rho'$-close to bipartite.
\end{itemize}
\end{lemma} 

\begin{proof}
To prove (i), note that
$$
|U| D = \sum\limits_{x \in U}d_G(x) = 2e_G(U) + e_G(U,\overline{U}) \leq |U|^2 + \rho n^2.
$$
So
$
|U| \geq D - \rho n^2/|U| \geq D - \sqrt{\rho}n,
$
as required.
To see (ii), note that
$$
e(W',\overline{W'}) = e(W',W) + e(W',\overline{U}) \leq e(\overline{W},W) + e(U,\overline{U}) \leq (\rho+\rho')n^2 \leq \gamma n^2.
$$
To see (iii), note first that $e(U,\overline{U}) \leq \rho n^2 \leq \rho'n^2$. Furthermore, (i) implies that $|U| \geq D - \sqrt{\rho}n \geq \sqrt{\rho'}n$.

We now prove (iv).
Since $e'(X_1, \overline{X_2}) \leq 2e(X_1,\overline{X_2}) \leq 2\rho n^2$ and since
$G$ is $D$-regular, we have that
\begin{equation} \label{e(X1,X2)}
|X_1|D - 2\rho n^2 \leq e'(X_1,V(G)) - e'(X_1,\overline{X_2}) = e'(X_1,X_2) \leq |X_1||X_2|.
\end{equation}
So $|X_2| \geq D - 2\rho n^2/|X_1| \geq D - 2\sqrt{\rho} n$.
A similar argument shows that $|X_1| \geq D - 2\sqrt{\rho}n$.
Finally, (v) holds since (C2) and (C3) are immediate, and (C1) follows from (iv).
\end{proof}

The following lemma implies that, for any regular graph $G$ and any $S \subseteq V(G)$ that is not too small, the robust neighbourhood of $S$ cannot be significantly smaller than $S$ itself.

\begin{lemma} \label{regexp2}
Let $n,D \in \mathbb{N}$ and
suppose that $0 < 1/n \ll \rho \ll \nu \ll \tau \ll \alpha < 1$.
 Let $U$ be a $\rho$-component of a $D$-regular graph $G$ on $n$ vertices, where $D \geq \alpha n$.
Let $S \subseteq U$ satisfy $|S| \geq \tau|U|$.
Write $N := RN_{\nu, U}(S)$ and let $Y := S \setminus N$ and $W := V(G) \setminus (S \cup N)$. Then
\begin{itemize}
\item[(i)] $e(S,Y) \leq \nu n^2$ and $e(S,W) \leq 2 \nu n^2$;
\item[(ii)] $|N| \geq |S| - \sqrt{\nu} n$;
\item[(iii)] $|N| \geq D - \sqrt{\nu}n$.
\end{itemize}
\end{lemma}

\begin{proof}

To prove (i), note that 
$e(S,Y) = e_{G[U]}(S,Y) \leq |Y|\nu|U| \leq \nu n^2$.
Moreover, $e'(S,\overline{N} \cap U) =  \sum_{x \in \overline{N} \cap U}d_{S}(x) \leq \nu |U|^2 \leq \nu n^2$.
Since $U$ is a $\rho$-component of $G$, we have that $e(U,\overline{U}) \leq \rho n^2$.
Hence
\begin{equation} \label{Nbar}
e'(S,\overline{N}) = e'(S,\overline{N} \cap U) + e(S,\overline{U}) \leq (\nu + \rho)n^2 \leq 2\nu n^2.
\end{equation}
This proves (i) as $e(S,W) \leq e'(S,\overline{N})$.
We now prove (ii).
Certainly $e'(S,N) \leq \sum_{x \in N}{d(x)} = D |N|$.
Similarly
\begin{equation} \label{eSN}
e'(S,N) = D |S| - e'(S,\overline{N}) \stackrel{(\ref{Nbar})}{\geq} D |S| - 2 \nu n^2.
\end{equation}
Then $|N| \geq |S| - 2 \nu n^2/D \geq |S| - \sqrt{\nu}n$,
which proves (ii).
Finally, we prove (iii).
Lemma~\ref{comp}(i) implies that $|U| \geq D - \sqrt{\rho}n$, so
\begin{equation} \label{Ssize}
|S| \geq \tau|U| \geq \tau D/2.
\end{equation}
Moreover, (\ref{eSN}) implies that $|S||N| \geq e'(S,N) \geq D|S| - 2\nu n^2$
and hence
$$
|N| \geq D - \frac{2\nu}{|S|}n^2 \stackrel{(\ref{Ssize})}{\geq} D - \frac{4\nu}{\tau D}n^2 \geq D - \sqrt{\nu}n,
$$ 
as required.
\end{proof}

The next lemma gives some sufficient conditions for $G[U]$ to be close to bipartite when $G$ is a regular graph and $U \subseteq V(G)$.

\begin{lemma} \label{rhoclose1}
Let $n,D \in \mathbb{N}$ and suppose that $0 < 1/n \ll \gamma' \leq \gamma \ll \alpha < 1$ where $\gamma' \leq \gamma^{7/6}$.
Suppose that $G$ is a $D$-regular graph on $n$ vertices where $D \geq \alpha n$.
Let $Y,Z$ be disjoint subsets of $V(G)$ such that
\begin{itemize}
\item[(i)] $|Y| \geq \gamma n$;
\item[(ii)] $||Y|-|Z|| \leq \gamma n$;
\item[(iii)] $e(Y,\overline{Z}) \leq \gamma' n^2$.
\end{itemize}
Then $G[Y \cup Z]$ is ${\gamma}^{1/3}$-close to bipartite with bipartition $Y,Z$.
\end{lemma}

\begin{proof}
First note that (C2) certainly holds with $\gamma^{1/3}$ playing the role of $\rho$.
Since 
$e'(Y,\overline{Z}) \leq 2e(Y,\overline{Z}) \leq 2\gamma'n^2$ and
$G$ is $D$-regular, we have that
\begin{equation} \label{e(Y,Z)}
|Y|D - 2\gamma' n^2 \leq e'(Y,V(G)) - e'(Y,\overline{Z}) = e'(Y,Z) \leq |Y||Z|.
\end{equation}
So $|Z| \geq D - 2\gamma' n^2/|Y| \geq 2\gamma^{1/6} n$%
\COMMENT{$D - 2\gamma'n^2/|Y| \geq (\alpha - 2\gamma'/\gamma)n \geq (\alpha - 2\gamma^{1/6})n \geq 2\gamma^{1/6}n$.}
and $|Y| \geq |Z| - \gamma n \geq \gamma^{1/6}n$.
Thus (C1) holds.
We also have that
\begin{eqnarray*}
e(Z, \overline{Y}) &\leq& e'(Z,\overline{Y}) = |Z|D - e'(Y,Z) \stackrel{(\ref{e(Y,Z)})}{\leq} (|Z| - |Y|)D + 2\gamma' n^2\\
 \nonumber &\stackrel{\rm{(ii)}}{\leq}& D\gamma n + 2\gamma' n^2 \leq 3\gamma n^2.
\end{eqnarray*}
So $e(Y,\overline{Z}) + e(Z,\overline{Y}) \leq 4\gamma n^2 \leq \gamma^{1/3}n^2$ and therefore (C3) holds.
\end{proof}

We now show that if a regular graph $G$ contains a non-expanding set $S$ whose intersection with its robust neighbourhood is small, then $G$ contains an induced subgraph which is close to bipartite.

\begin{lemma} \label{rhoclose}
Let $n,D \in \mathbb{N}$ and suppose that $0 < 1/n \ll \rho \ll \nu \ll \tau \ll \alpha < 1$.
Suppose that $G$ is a $D$-regular graph on $n$ vertices where $D \geq \alpha n$.
Let $U \subseteq V(G)$ be a $\rho$-component of $G$.
Suppose that $S \subseteq U$ is non-$\nu$-expanding in $U$ and $|S| \geq \tau|U|$.
Let $N := RN_{\nu,U}(S)$, $Y := S \setminus N$ and $Z := N \setminus S$.
Then
\begin{itemize}
\item[(i)] $||Y|-|Z|| \leq \sqrt{\nu}n$;
\item[(ii)] if also $|Y| > \sqrt{\nu}n$, then $G[Y \cup Z]$ is ${\nu}^{1/6}$-close to bipartite with bipartition $Y,Z$.
\end{itemize}
\end{lemma}

\begin{proof}
Let $X := S \cap N$.
So $S = X \cup Y$ and $N = X \cup Z$.
Since $S$ is non-$\nu$-expanding in $U$, we have that $|N| < |S| + \nu |U|$.
By Lemma~\ref{regexp2}(ii) we have that
$$
|S| - \sqrt{\nu} n \leq |N| < |S| + \nu |U| \leq |S| + \sqrt{\nu} n,
$$
which proves (i).
To prove (ii), 
let $W := \overline{S \cup N} = \overline{X \cup Y \cup Z}$.
Note that Lemma~\ref{regexp2}(i) implies that
\begin{equation} \label{YZbar}
e(Y, \overline{Z}) = e(Y, S \cup W) \leq e(S,Y) + e(S,W) \leq 3\nu n^2.
\end{equation}
Set $\gamma' := 3\nu$ and $\gamma := \sqrt{\nu}$.
Then $\gamma' \leq \gamma^{5/6}$.
So we can apply Lemma~\ref{rhoclose1} to see that
$G[Y \cup Z]$ is ${\nu}^{1/6}$-close to bipartite with bipartition $Y,Z$.
\end{proof}

The next proposition formalises the fact that, if a graph $G$ contains a subset $U$ that is close to bipartite; we may add or remove any small set of vertices so that it is still close to bipartite (with slightly weaker parameters).
We omit the proof as it is straightforward to check that (C1)--(C3) are satisfied.

\begin{proposition} \label{symdiff}
Let $n,D \in \mathbb{N}$, $0 < 1/n \ll \rho_1,\rho_2 \ll \alpha < 1$ and let $\rho \geq \rho_1 + 2\rho_2$.
Suppose that $G$ is a $D$-regular graph on $n$ vertices where $D \geq \alpha n$ and let $U \subseteq V(G)$ be such that $G[U]$ is $\rho_1$-close to bipartite, with bipartition $A,B$.
Suppose that $A',B' \subseteq V(G)$ are disjoint and $|A \triangle A'| + |B \triangle B'| \leq \rho_2 n$.
Let $U' := A' \cup B'$.
Then $G[U']$ is $\rho$-close to bipartite with bipartition $A',B'$.%
\COMMENT{
We need to check that (C1)--(C3) hold.
Note that
$$
|A'| \geq |A| - \rho_2 n \geq D - (2\sqrt{\rho_1} + \rho_2)n \geq \sqrt{\rho} n,
$$
and similarly for $|B'|$.
So (C1) holds.
Also
\begin{align*}
||A'|-|B'|| &\leq |A' \triangle A| + ||A|-|B|| + |B \triangle B'| \leq (\rho_1 + \rho_2)n \leq \rho n,
\end{align*}
so (C2) holds.
Moreover,
\begin{align*}
e(A',\overline{B'}) + e(B',\overline{A'})
&\leq e(A,\overline{B}) + e(B,\overline{A}) + 2(|A' \triangle A| + |B' \triangle B|)n\\
&\leq  (\rho_1 + 2\rho_2)n^2 \leq \rho n^2.
\end{align*}
So (C3) holds, as required.
}
\end{proposition}

\subsection{Properties of non-expanding subsets}

In this subsection we prove that a $\rho$-component is either a robust expander component, a bipartite robust expander component, or the union of two $\rho'$-components (where $\rho \ll \rho'$).
This forms the core of the proof of Theorem~\ref{structure}.

For this, we first show that if $U$ is a $\rho$-component in a regular graph $G$ such that $G[U]$ is not a robust expander, then either $G[U]$ is close to bipartite, or $U$ can be decomposed into two $\rho'$-components.
To prove this, we consider a non-expanding set $S$ and its robust neighbourhood $N$. We use our previous results to show that either $S \cup N$ and its complement in $U$ are both $\rho'$-components or $G[U]$ is $\rho'$-close to bipartite.

\begin{lemma} \label{partialstruct}
Let $n,D \in \mathbb{N}$ and suppose that $0 < 1/n \ll \rho \ll \nu \ll \rho' \ll \tau \ll \alpha < 1$.
Let $U$ be a $\rho$-component of a $D$-regular graph $G$ on $n$ vertices where $D \geq \alpha n$.
Suppose that $G[U]$ is not a robust $(\nu,\tau)$-expander.
Then at least one of the following hold:
\begin{itemize}
\item[(i)] $U$ has a partition $U_1, U_2$ such that each of $U_1, U_2$ is a $\rho'$-component of $G$;
\item[(ii)] $G[U]$ is $\rho'$-close to bipartite.
\end{itemize}
\end{lemma}

\begin{proof}
Since%
\COMMENT{Actually, we don't need the corollaries which stated that any small robust component is automatically a (bipartite) robust expander.}
$G[U]$ is not a robust $(\nu,\tau)$-expander, there exists $S \subseteq U$ with 
\begin{equation} \label{S}
\tau |U| \leq |S| \leq (1-\tau)|U|
\end{equation}
and $|RN_{\nu,U}(S)| < |S| + \nu |U|$.
Let $N := RN_{\nu,U}(S)$, $X := S \cap N$, $Y := S \setminus N$, $Z := N \setminus S$ and $W := V(G) \setminus (S \cup N)$.
We consider two cases, depending on the size of $Y$.

\medskip
\noindent
\textbf{Case 1.}
\emph{$|Y| \leq \sqrt{\nu} n$.}

\medskip
\noindent
In this case, we will show that (i) holds.
Let $U_1 := S \cup N = S \cup Z$ so that $\overline{U_1} = W$.
Then Lemma~\ref{regexp2}(iii) implies that $|U_1| \geq |N| \geq D - \sqrt{\nu}n \geq \sqrt{\rho'}n$.
By Lemma~\ref{rhoclose}(i), we have
\begin{align} \label{Zsize}
|Z| &\leq |Y| + \sqrt{\nu} n \leq 2\sqrt{\nu} n\\
\label{Zsize2}&\leq \tau D/4 \leq \tau |U|/2,
\end{align}
where the last inequality holds since $|U| \geq D - \sqrt{\rho}n$ by Lemma~\ref{comp}(i).
Now Lemma~\ref{regexp2}(i) implies that
\begin{equation*}
e(U_1,\overline{U_1}) = e(S, W) + e(Z,W) \leq 2 \nu n^2 + |Z|n \stackrel{(\ref{Zsize})}{\leq} 3 \sqrt{\nu} n^2.
\end{equation*}
So $U_1$ is a $3\sqrt{\nu}$-component of $G$.
Moreover, (\ref{Zsize2}) and (\ref{S}) together imply that
$$
|U_1| = |S|+|Z| \leq \left(1-\tau/2\right)|U|.
$$
Let $U_2 := U \setminus U_1$.
Then
$|U_2| \geq \tau|U|/2 \geq \sqrt{\rho'}n$.
Since $U$ is a $\rho$-component, $U_1$ is a $3\sqrt{\nu}$-component, and $\rho+3\sqrt{\nu} \leq \rho'$, we can apply
Lemma~\ref{comp}(ii)
with $U_1,U_2,\rho,3\sqrt{\nu},\rho'$ playing the roles of $W,W',\rho,\rho',\gamma$ respectively
to see that $e(U_2,\overline{U_2}) \leq \rho'n^2$. 
Thus $U_2$ is a $\rho'$-component of $G$ and so (i) holds.

\medskip
\noindent
\textbf{Case 2.}
\emph{$|Y| > \sqrt{\nu} n$.}

\medskip
\noindent
Let $U_1 := Y \cup Z = S \triangle N$.
Lemma~\ref{rhoclose}(ii) implies that $G[U_1]$ is ${\nu}^{1/6}$-close to bipartite with bipartition $Y,Z$.
Therefore (C1) and (C3) imply that $U_1$ is a $\nu^{1/6}$-component.
Moreover, $|U_1| \geq 2(D-2\nu^{1/12}n) \geq 2(D - 2\sqrt{\rho'}n)$ by Lemma~\ref{comp}(iv). 
Let $U_2 := U \setminus U_1$.
Now Lemma~\ref{comp}(ii) with $U_1,U_2,\rho,\nu^{1/6},(\rho'/3)^2$ playing the roles of $W,W',\rho,\rho',\gamma$ implies that $e(U_2,\overline{U_2}) \leq (\rho'/3)^2n^2$.
If $|U_2| \geq \rho'n/3$ then $U_2$ is a $(\rho'/3)^2$-component.
So Lemma~\ref{comp}(i) implies that $|U_2| \geq D - \rho'n/3$ and thus $U_2$ is actually a $\rho'$-component of $G$. So (i) holds in this case.

Thus we may assume that $|U_2| < \rho'n/3$.
Let $Y' := Y \cup U_2$ and $Z' := Z$. Then $Y',Z'$ are disjoint subsets whose union is $U$. Note that
$|Y' \triangle Y| + |Z' \triangle Z| = |U_2| < \rho' n/3$.
Now Proposition~\ref{symdiff} with $U_1,U,Y,Z,Y',Z',\nu^{1/6},\rho'/3,\rho'$ playing%
\COMMENT{$\rho_1+2\rho_2 \leq \nu^{1/6}+2\rho'/3 \leq \rho'$}
 the roles of $U,U',A,B,A',B',\rho_1,\rho_2,\rho$ implies that $G[U]$ is $\rho'$-close to bipartite with bipartition $Y',Z'$.
So (ii) holds.
\end{proof}

The following lemma is a bipartite analogue of Lemma~\ref{partialstruct}.
It states that if $G$ is a regular graph and $U \subseteq V(G)$ such that $G[U]$ is close to bipartite and $G[U]$ is not a bipartite robust expander, then $U$ can be decomposed into two components.
The proof is similar to that of Lemma~\ref{partialstruct} -- we find the partition by considering a non-expanding set and its robust neighbourhood -- and can be found in~\cite{thesis}.

\begin{lemma} \label{bippartialstruct}
Let $n,D \in \mathbb{N}$ and suppose that $0 < 1/n \ll \rho \ll \nu \ll \rho' \ll \tau \ll \alpha < 1$.
Let $G$ be a $D$-regular graph on $n$ vertices where $D \geq \alpha n$.
Suppose that $U \subseteq V(G)$ is such that $G[U]$ is $\rho$-close to bipartite with bipartition $A,B$ and $G[U]$ is not a bipartite robust $(\nu,\tau)$-expander with bipartition $A,B$.%
\COMMENT{so $G$ is a $\rho$-component by definition.}
Then there is a partition $U_1, U_2$ of $U$ such that $U_1, U_2$ are $\rho'$-components.
\end{lemma}

\subsection{Adjusting partitions}

The results of this subsection will be needed to ensure (D4), (D5) and (D7) in the proof of Theorem~\ref{structure}.

The next two lemmas state that (bipartite) robust expanders are indeed robust, in the sense that the expansion property cannot be destroyed by adding or removing a small number of vertices.
We omit the proofs as they follow from the definitions in a straightforward way.

\begin{lemma} \label{expanderswallow}
Let $0 < \nu \ll \tau \ll 1$.
Suppose that $G$ is a graph and $U,U' \subseteq V(G)$ are such that $G[U]$ is a robust $(\nu,\tau)$-expander and $|U \triangle U'| \leq \nu|U|/2$.
Then $G[U']$ is a robust $(\nu/2,2\tau)$-expander.%
\COMMENT{To do: prove this (with better constants) for my thesis.}
\end{lemma}

\begin{lemma} \label{bipexpanderswallow}
Let $0 < \nu \ll \tau \ll 1$.
Suppose that $U \subseteq V(G)$ and that $G[U]$ is a bipartite robust $(\nu,\tau)$-expander with bipartition $A,B$.
Let $W,A',B' \subseteq V(G)$ be such that $|W| \leq \nu |A|/2$; $A'$ and $B'$ are disjoint; and $|A \triangle A'| + |B \triangle B'| \leq \nu|A|/2$.
Then 
\begin{itemize}
\item[(i)] $G[U \setminus W]$ is a bipartite robust $(\nu/2, 2\tau)$-expander with bipartition $A \setminus W, B \setminus W$;
\item[(ii)] $G[A' \cup B']$ is a bipartite robust $(\nu/2,2\tau)$-expander with bipartition $A',B'$.
\end{itemize}
\end{lemma}

We now extend Lemma~\ref{bipexpanderswallow} by showing that, after adding and removing a small number of vertices, a bipartite robust component is still a bipartite robust component, with slightly weaker parameters.
The proof may be found in~\cite{thesis}.

\begin{lemma} \label{BREadjust}
Let $0 < 1/n \ll \rho \leq \gamma \ll \nu \ll \tau \ll \alpha < 1$ and suppose that $G$ is a $D$-regular graph on $n$ vertices where $D \geq \alpha n$.
\begin{itemize}
\item[(i)] Suppose that $G[A \cup B]$ is a bipartite $(\rho,\nu,\tau)$-robust expander component of $G$ with bipartition $A,B$. Let $A',B' \subseteq V(G)$ be such that $|A \triangle A'| + |B \triangle B'| \leq \gamma n$.
Then $G[A' \cup B']$ is a bipartite $(3\gamma,\nu/2,2\tau)$-robust expander component of $G$ with bipartition $A',B'$.
\item[(ii)] Suppose that $G[U]$ is a bipartite $(\rho,\nu,\tau)$-robust expander component of $G$.
Let $U' \subseteq V(G)$ be such that $|U \triangle U'| \leq \gamma n$.
Then $G[U']$ is a bipartite $(3\gamma,\nu/2,2\tau)$-robust expander component of $G$.
\end{itemize}

\end{lemma}

In any $\rho$-component, almost all vertices have very few neighbours outside the component.
In particular, most vertices have more neighbours within their own component than in any other.
The following lemma allows us to move a small number of vertices in a partition into $\rho$-components so that this property holds for all vertices.

\begin{lemma} \label{shuffle}
Let $m,n,D \in \mathbb{N}$ and $0 < 1/n \ll \rho  \ll \alpha, 1/m \leq 1$. Let $G$ be a $D$-regular graph on $n$ vertices where $D \geq \alpha n$. Suppose that $\mathcal{U} := \lbrace U_1, \ldots, U_m \rbrace$ is a partition of $V(G)$ such that $U_i$ is a $\rho$-component for each $1 \leq i \leq m$.
Then $G$ has a vertex partition $\mathcal{V} := \lbrace V_1, \ldots, V_m \rbrace$ such that
\begin{itemize}
\item[(i)] $|U_i \triangle V_i| \leq \rho^{1/3}n$;
\item[(ii)] $V_i$ is a $\rho^{1/3}$-component for each $1 \leq i \leq m$;
\item[(iii)] if $x \in V_i$ then $d_{V_i}(x) \geq d_{V_j}(x)$ for all $1 \leq i,j \leq m$. In particular, $d_V(x) \geq D/m$ for all $x \in V$ and all $V \in \mathcal{V}$;
\item[(iv)] for all but at most $\rho^{1/3}n$ vertices $x \in V_i$ we have $d_{V_i}(x) \geq D-2\sqrt{\rho}n$.
\end{itemize}
\end{lemma}

\begin{proof}
First note that the second part of (iii) follows from the first.%
\COMMENT{Indeed, suppose that $d_{X}(x) < D/m$ for some $X \in \mathcal{V}$ and some $x \in X$.
Then $x$ has more than $(1-1/m)D = (m-1)D/m$ neighbours in $G \setminus X$.
Therefore, by an averaging argument, there is some $X' \in \mathcal{V} \setminus \lbrace X \rbrace$ such that $d_{X'}(x) \geq D/m > d_{X}(x)$, a contradiction.
}
For each $1 \leq i \leq m$, let $X_i$ be the collection of vertices $y \in U_i$ with $d_{\overline{U_i}}(y) \geq \sqrt{\rho}n$.
Since $U_i$ is a $\rho$-component, we have%
\COMMENT{
$
\rho n^2 \geq \sum\limits_{y \in U_i}d_{\overline{U_i}}(y) \geq \sum\limits_{x \in X_i}d_{\overline{U_i}}(x) \geq |X_i|\sqrt{\rho}n.
$}
$|X_i| \leq \sqrt{\rho}n$.
Let $W_i := U_i \setminus X_i$.
Then each $x \in W_i$ satisfies
\begin{equation} \label{dWi}
d_{W_i}(x) = d(x) - d_{\overline{U_i} \cup X_i}(x) \geq d(x) - \sqrt{\rho}n - |X_i| \geq d(x) - 2\sqrt{\rho}n.
\end{equation}

Let $X := \bigcup_{1 \leq i \leq m}X_i$.
Among all partitions $X_1',\ldots,X_m'$ of $X$, choose one such that $\sum_{1 \leq i \leq m}e(V_i,\overline{V_i})$ is minimal, where $V_i := W_i \cup X_i'$.
It is easy to see that $d_{V_i}(x) \geq d_{V_j}(x)$ for all $x \in X_i'$ and all $1 \leq i,j \leq m$.%
\COMMENT{
Suppose not; so there exist $1 \leq i,j \leq m$ and $x \in X_i'$ such that $d_{V_i}(x) < d_{V_j}(x)$. If we move $x$ from $X_i'$ to $X_j'$, then $e(V_i,\overline{V_i})+e(V_j,\overline{V_j})$ decreases by $2 ( d_{V_j}(x)-d_{V_i}(x) )$, and $e(V_k,\overline{V_k})$ is unchanged for all $k \neq i,j$.
So $\sum_{1 \leq i \leq m}e(V_i,\overline{V_i})$ decreases by at least two, contradicting minimality.}
So (iii) holds for all $x \in X_i'$ and $i \leq m$.
Moreover, if $x \in W_i$, then (\ref{dWi}) implies that $d_{V_i}(x) \geq d_{W_i}(x) \geq d(x)-2\sqrt{\rho}n \geq d(x)/2$.
So (iii) also holds for each vertex in $W_i$.
Furthermore,
$$
\sum\limits_{1 \leq i \leq m}e(V_i,\overline{V_i}) \leq \sum\limits_{1 \leq i \leq m}e(U_i,\overline{U_i}) \leq \rho m n^2 \leq \rho^{1/3}n^2
$$
and hence each $V_i$ is a $\rho^{1/3}$-component, so (ii) holds.

Note that $U_i \cap V_i \supseteq W_i$, so
\begin{equation} \label{UtriagV}
|U_i \triangle V_i| \leq \sum\limits_{1 \leq i \leq m}|X_i'| = |X| \leq m\sqrt{\rho}n \leq \rho^{1/3}n,
\end{equation}
which proves (i).
Finally, (\ref{dWi}) and the fact that $|V_i \setminus W_i| \leq |X| \leq \rho^{1/3}n$ by (\ref{UtriagV}) together imply (iv).
\end{proof}

The next lemma shows that, in a bipartite robust expander component, we can adjust the bipartition slightly so that any vertex has at least as many neighbours in the opposite class as within its own class.
The resulting graph will still be a bipartite robust expander component.
The proof is very similar to that of Lemma~\ref{shuffle} and may be found in~\cite{thesis}.

\begin{lemma} \label{bipshuffle}
Let $0 < 1/n \ll \rho \ll \nu \ll \tau \ll \alpha < 1$ and let $G$ be a $D$-regular graph on $n$ vertices where $D \geq \alpha n$.
Suppose that $U$ is a bipartite $(\rho,\nu,\tau)$-robust component of $G$.
Then there exists a bipartition $A,B$ of $U$ such that
\begin{itemize}
\item[(i)] $U$ is a bipartite $(3\sqrt{\rho},\nu/2,2\tau)$-robust component with partition $A,B$;
\item[(ii)] $d_{B}(u) \geq d_{A}(u)$ for all $u \in A$, and $d_{A}(v) \geq d_{B}(v)$ for all $v \in B$;
\end{itemize}
\end{lemma}

\subsection{Proof of the main result}

We are now ready to prove Theorem~\ref{structure} -- that every sufficiently dense regular graph has a robust partition.
The first part of the proof is an iteration of Lemmas~\ref{partialstruct} and~\ref{bippartialstruct} -- we begin with the trivial partition of $V(G)$ and successively refine it by applying Lemma~\ref{bippartialstruct} to those components which are close to bipartite and Lemma~\ref{partialstruct} to the others, until we obtain a partition into robust components.
We then use Lemma~\ref{shuffle} to adjust the partition slightly and Lemma~\ref{bipshuffle} to achieve an appropriate bipartition of the bipartite robust expander components.

\medskip
\noindent
\emph{Proof of Theorem~\ref{structure}.}
Let $t := 3\lceil 2/\alpha \rceil$.
Define further constants satisfying
$$
0 < 1/n_0 \ll \rho_1 \ll \nu_1 \ll \rho_2 \ll \nu_2 \ll \ldots \ll \rho_t \ll \nu_t \ll \tau' \ll \alpha,\tau
$$
so that $1/n_0 \leq f(\rho_1)$ and $3^{3/2}\rho_i^{1/6} \leq f(\nu_i/4)$  for all $1 \leq i \leq t$.
We first prove the following claim.

\medskip
\noindent
\textbf{Claim.}
\emph{There is some $1 \leq i < t$ and a partition $\mathcal{U}$ of $V(G)$ such that $U$ is a $(\rho_i,\nu_i,\tau')$-robust component for each $U \in \mathcal{U}$.}

\medskip
\noindent
To see this,
let $\mathcal{U}_1 := \lbrace V(G) \rbrace$.
Note that $V(G)$ is certainly a $\rho_1$-component of $G$, and $|\mathcal{U}_1| = 1$.
Suppose, for some $i$ with $1 \leq i < t$, we have inductively defined a partition $\mathcal{U}_i$ of $V(G)$ such that
$U$ is a $\rho_i$-component for each $U \in \mathcal{U}_i$ and $2|\mathcal{U}_i| + |\mathcal{W}_i| \geq i+1$, where $\mathcal{W}_i$ is the collection of all those $U \in \mathcal{U}_i$ which are $\rho_i$-close to bipartite.
If each $U \in \mathcal{U}_{i}$ is a $(\rho_{i},\nu_{i},\tau')$-robust component, then we are done by setting $\mathcal{U} := \mathcal{U}_i$.
Otherwise, we obtain $\mathcal{U}_{i+1}$ from $\mathcal{U}_i$ as follows.

There is some $U \in \mathcal{U}_i$ which is not a $(\rho_i,\nu_i,\tau')$-robust component.
If $U \in \mathcal{W}_i$, then apply Lemma~\ref{bippartialstruct} with $\rho_i,\nu_i,\rho_{i+1},\tau'$ playing the roles of $\rho,\nu,\rho',\tau$ to obtain a partition $U_1, U_2$ of $U$ such that $U_1,U_2$ are $\rho_{i+1}$-components.
Let $\mathcal{U}_{i+1} := (\mathcal{U}_i \setminus \lbrace U \rbrace) \cup \lbrace U_1, U_2 \rbrace$.
Lemma~\ref{comp}(v) implies that $\mathcal{W}_i \setminus \lbrace U \rbrace \subseteq \mathcal{W}_{i+1}$, 
where $\mathcal{W}_{i+1}$ is the collection of all those $U \in \mathcal{U}_{i+1}$ which are $\rho_{i+1}$-close to bipartite.
Thus $|\mathcal{U}_{i+1}| = |\mathcal{U}_i| + 1$ and $|\mathcal{W}_{i+1}| \geq |\mathcal{W}_i| - 1$.

So suppose next that
$U \in \mathcal{U}_i \setminus \mathcal{W}_i$.
Apply Lemma~\ref{partialstruct} with $\rho_i,\nu_i,\rho_{i+1},\tau'$ playing the roles of $\rho,\nu,\rho',\tau$.
If (i) holds, then $U$ has a partition $U_1,U_2$ such that $U_1,U_2$ are $\rho_{i+1}$-components.
As before, we let $\mathcal{U}_{i+1} := (\mathcal{U}_i \setminus \lbrace U \rbrace) \cup \lbrace U_1, U_2 \rbrace$.
So $|\mathcal{U}_{i+1}| = |\mathcal{U}_i| + 1$ and $|\mathcal{W}_{i+1}| \geq |\mathcal{W}_i|$.
Otherwise, Lemma~\ref{partialstruct}(ii) holds.
Then $U$ is $\rho_{i+1}$-close to bipartite.
We let $\mathcal{U}_{i+1} := \mathcal{U}_i$.
Then $|\mathcal{U}_{i+1}| = |\mathcal{U}_i|$ and $|\mathcal{W}_{i+1}| \geq |\mathcal{W}_i| + 1$.
 
Note that in each case we have $2|\mathcal{U}_{i+1}| + |\mathcal{W}_{i+1}| \geq i+2$.
Moreover, Lemma~\ref{comp}(iii) implies that each $W \in \mathcal{U}_i \setminus \lbrace U \rbrace$ is a $\rho_{i+1}$-component.
Therefore each $W \in \mathcal{U}_{i+1}$ is a $\rho_{i+1}$-component.

It remains to show that this process must stop before we define $\mathcal{U}_t$.
Suppose not, i.e.~suppose we have defined $\mathcal{U}_t$.
Since each $W \in \mathcal{U}_t$ is a $\rho_t$-component, Lemma~\ref{comp}(i) implies that $|W| \geq (\alpha - \sqrt{\rho_t})n$ for all $W \in \mathcal{U}_t$.
Moreover, $|\mathcal{U}_t| > t/3$ since $3|\mathcal{U}_t| \ge 2|\mathcal{U}_t| + |\mathcal{W}_t| \geq t+1$.
Altogether, this implies that
$$
|V(G)| \geq \frac{t}{3}(\alpha - \sqrt{\rho_t})n \geq \frac{2}{\alpha}(\alpha-\sqrt{\rho_t})n > n,
$$
a contradiction.
This completes the proof of the claim. 

\medskip
\noindent
Set $\rho' := \rho_i$, $\nu' := \nu_i$, $\rho := 3^{3/2}\rho'^{1/6}$ and $\nu := \nu'/4$.
So
\begin{equation} \label{fdiff}
\rho = 3^{3/2}\rho'^{1/6} \leq f(\nu'/4) = f(\nu)\ \ \mbox{ and }\ \ 1/n_0 \leq f(\rho_1) \leq f(\rho)
\end{equation}
and every $U \in \mathcal{U}$ is a $(\rho',\nu',\tau')$-robust component of $G$.
So there exist $k,\ell \in \mathbb{N}$ such that $\mathcal{U} = \lbrace U_1, \ldots, U_k, Z_1, \ldots, Z_\ell \rbrace$, where $U_i$ is a $(\rho',\nu',\tau')$-robust expander component for all $1 \leq i \leq k$, and $Z_j$ is a bipartite $(\rho',\nu',\tau')$-robust expander component for all $1 \leq j \leq \ell$.
Let $m := k + \ell$.
Note that for each $1 \leq i \leq k$, we have $|U_i| \geq D - \sqrt{\rho'}n$ (by Lemma~\ref{comp}(i) and since $U_i$ is a $\rho'$-component).
Moreover, for each $1 \leq j \leq \ell$, $|Z_j| \geq 2(D-2\sqrt{\rho'}n)$ by Lemma~\ref{comp}(iv).
Thus
$$
n = \sum\limits_{1 \leq i \leq k}|U_i| + \sum\limits_{1 \leq j \leq \ell}|Z_j| \geq (D-2\sqrt{\rho'}n)(k + 2\ell)
$$
and so
\begin{equation} \label{k+2l}
k + 2\ell \leq \left\lfloor \frac{n}{D-2\sqrt{\rho'}n} \right\rfloor \leq \left\lfloor (1+\rho'^{1/3})\frac{n}{D} \right\rfloor.
\end{equation}
In particular, $1/m \geq \alpha/2$.

To achieve (D4), we apply Lemma~\ref{shuffle} with $\rho'$ playing the role of $\rho$ to $\mathcal{U}$ to obtain a new partition $\mathcal{V} = \lbrace V_1, \ldots, V_k, W_1, \ldots, W_\ell \rbrace$  of $V(G)$ satisfying (i)--(iv), so in particular
\begin{equation} \label{UVdiffZW}
|U_i \triangle V_i|, |Z_j \triangle W_j| \leq \rho'^{1/3}n
\end{equation}
for all $1 \leq i \leq k$ and all $1 \leq j \leq \ell$.
We claim that $\mathcal{V}$ satisfies (D1)--(D7).

Now (D1) certainly holds, (D4) follows from Lemma~\ref{shuffle}(iii) and (D7) follows from Lemma~\ref{shuffle}(iv).%
\COMMENT{Lemma~\ref{shuffle}(iv) implies that, for all $X \in \mathcal{V}$, for all but at most $\rho'^{1/3}$ vertices $x \in X$, we have $d_X(x) \geq D - 2\sqrt{\rho'}n$ by Lemma~\ref{comp}(i). But $\rho'^{1/3} \leq \rho$ and $2\sqrt{\rho'} \leq \rho$, so (D7) holds.}
To prove (D2), note that $V_i$ is a $\rho'^{1/3}$-component by Lemma~\ref{shuffle}(ii) and $|V_i| \geq D/2 \geq \sqrt{\rho}n$. Thus $V_i$ is a $\rho$-component, i.e.~(E1) holds.
Now, by (\ref{UVdiffZW}) and Lemma~\ref{expanderswallow} with $\nu',\tau',U_i,V_i$ playing the roles of $\nu,\tau,U,U'$, we have that $G[V_i]$ is a robust $(\nu'/2,2\tau')$-expander and thus also a robust $(\nu,\tau)$-expander.
So (E2) holds, proving (D2).

To check (D3), recall that $G[Z_j]$ is a bipartite $(\rho',\nu',\tau')$-robust expander component. 
Then (\ref{UVdiffZW}) and Lemma~\ref{BREadjust}(ii) applied with $\rho',\rho'^{1/3},\nu',\tau',Z_j,W_j$ playing the roles of $\rho,\gamma,\nu,\tau,U,U'$ imply that $G[W_j]$ is a bipartite $(3\rho'^{1/3},\nu'/2,2\tau')$-robust expander component.
Now for each $1 \leq j \leq \ell$, apply Lemma~\ref{bipshuffle} to $W_j$ with $3\rho'^{1/3},\nu'/2,2\tau'$ playing the roles of $\rho,\nu,\tau$
to obtain a bipartition $A_j,B_j$ of $W_j$ satisfying (i) and (ii).
Lemma~\ref{bipshuffle}(i) implies that $G[W_j]$ is a bipartite $(\rho,\nu,\tau)$-robust expander component with bipartition $A_j,B_j$. So (D3) holds.
Lemma~\ref{bipshuffle}(ii) implies that (D5) holds.
Finally, (D6) follows from (\ref{k+2l}).
\hfill$\square$


\section{Extending Theorem~\ref{structure} to almost regular graphs}\label{almostregular}

In this section, we prove an extension of Theorem~\ref{structure} which states that every dense almost regular graph has a robust partition.
We first extend the definition of a robust partition to graphs which may not be regular.
Let $k,\ell, D \in \mathbb{N}$ and $0 < \rho \leq \nu \leq \tau < 1$.
Given a graph $G$ on $n$ vertices, we say that $\mathcal{V}$ is a \emph{robust partition of $G$ with parameters $\rho,\nu,\tau,k,\ell$} if (D1)--(D7) hold with $\delta(G)$ playing the role of $D$.
Note that, for $D$-regular graphs, this coincides with the definition given in Section~\ref{sec:struct}.

\begin{theorem} \label{structuregamma}
For all $\alpha,\tau > 0$ and every non-decreasing function $f : (0,1) \rightarrow (0,1)$, there exist $n_0 \in \mathbb{N}$ and $\gamma >0$ such that the following holds.
For all graphs $G$ on $n \geq n_0$ vertices with $\alpha n \leq \delta(G) \leq \Delta(G) \leq \delta(G)+\gamma n$, there exist $\rho,\nu$ with $1/n_0,\gamma \leq \rho \leq \nu \leq \tau$; $\rho \leq f(\nu)$ and $1/n_0 \leq f(\rho)$, and $k,\ell \in \mathbb{N}$ such that $G$ has a robust partition $\mathcal{V}$ with parameters $\rho,\nu,\tau,k,\ell$.
\end{theorem}

The proof proceeds by taking two copies of $G$ and adding a small number of edges between them to obtain a regular graph $G'$, whose degree is only slightly higher than~$\Delta(G)$.
We apply Theorem~\ref{structure} to obtain a robust partition $\mathcal{V}$ of $G'$.
The construction of $G'$ implies that every robust component in $\mathcal{V}$ lies entirely in one copy of $G$.
So there is a partition of $\mathcal{V}$ into two parts, one of which must be a robust partition of $G$.

In order to construct $G'$ from $G$, we need some preliminaries.
We say that a non-decreasing sequence $(d_i)_{1 \leq i \leq n}$ of positive integers is \emph{bipartite graphic} if there exists a bipartite graph $G$ with vertex classes $A$ and $B$ with $|A|=|B|=n$ such that the $i$th vertex of each of 
$A$ and $B$ has degree $d_i$.
The following theorem of Alon, Ben-Shimon and Krivelevich~\cite{abk} gives a sufficient condition for a sequence to be bipartite graphic.
Their result is stated differently to the statement below, but the two forms are equivalent, as observed in~\cite{cmn}.%
\COMMENT{This latter paper http://arxiv.org/pdf/1403.6307.pdf appeared on arxiv this week so I don't know if we want to quote it. We could just use the result from Alon et al but the version in the new paper is nicer.}

\begin{theorem}\label{bipgraphic}
Suppose that $(d_i)_{1 \leq i \leq n}$ is a non-decreasing sequence of positive integers. Then $(d_i)_{1 \leq i \leq n}$ is bipartite graphic if $nd_1 \geq (d_1+d_n)^2/4$.
\end{theorem}

We also need the following result (Lemma 3.8 from~\cite{KellyII}).

\begin{lemma}\label{denseexp}
Suppose that $0 < \nu \leq \tau \leq \eps < 1$ are such that $\eps \geq 2\nu/\tau$.
Let $G$ be a graph on $n$ vertices with minimum degree $\delta(G) \geq (1/2 + \eps)n$.
Then $G$ is a robust $(\nu,\tau)$-expander.
\end{lemma}

We are now able to deduce Theorem~\ref{structuregamma} from Theorem~\ref{structure}.

\medskip
\noindent
\emph{Proof that Theorem~\ref{structure} implies Theorem~\ref{structuregamma}.}
Define $f' : (0,1) \rightarrow (0,1)$ by $f'(x) := \min \lbrace f(x)/4, \alpha x/2 \rbrace$ and let $\tau' := \min \lbrace \tau,\alpha^2/20 \rbrace$.
Apply Theorem~\ref{structure} with $\alpha,\tau',f'$ playing the roles of $\alpha,\tau,f$ to obtain $n_0 \in \mathbb{N}$.
Let $\gamma := 1/4n_0$.
Let $G$ be a graph on $n \geq n_0$ vertices with $\alpha n \leq \delta(G) \leq \Delta(G) \leq \delta(G)+\gamma n$.
Let $D := \delta(G)$.
Order the vertices $v_1, \ldots, v_n$ of $G$ in order of increasing degree.

Obtain a graph $G''$ from $G$ as follows.
We let $W_1 := \lbrace w_1, \ldots, w_n \rbrace$ and $W_2 := \lbrace x_1, \ldots, x_n \rbrace$ be disjoint sets of vertices and let $G''$ have vertex set $W_1 \cup W_2$.
We add the edges $w_iw_j$ and $x_ix_j$ whenever $v_iv_j \in E(G)$.
Choose a constant $\beta$ such that $\gamma = \beta(1-\beta)$ and $\gamma \leq \beta \leq 2\gamma$.%
\COMMENT{can be done since $\gamma \leq 1/4$.}

Let $d_i := D+\beta n - d_{G}(v_{n+1-i})$.
Then
$(d_i)_{1 \leq i \leq n}$ is a non-decreasing sequence
and $(\beta-\gamma)n \leq d_1 \leq d_n \leq \beta n$.
Observe that
if $(d_i)_{1 \leq i \leq n}$ is bipartite graphic, then we can add edges to $G''$ between $W_1$ and $W_2$ to obtain a $(D +\beta n)$-regular graph $G'$.
Since
$
(d_1+d_n)^2/4 \leq \beta^2 n^2 = (\beta - \gamma)n^2 \leq nd_1
$,
Theorem~\ref{bipgraphic} implies that such a $G'$ exists.
Note that
\begin{equation}\label{sparseacross}
\Delta(G'[W_1,W_2])  = d_n \leq \beta n.
\end{equation}
 Theorem~\ref{structure} applied to $G'$ implies that there exist $\rho',\nu$ with $1/n_0 \leq \rho' \leq \nu \leq \tau'$; $\rho' \leq f'(\nu)$ and $1/n_0 \leq f'(\rho')$, and $k',\ell' \in \mathbb{N}$ such that $G'$ has a robust partition $\mathcal{V}$ with parameters $\rho',\nu,\tau',k',\ell'$.
Note that $\beta \leq 2\gamma = 1/2n_0 \leq \nu/2$.

\medskip\noindent
\textbf{Claim.}
\emph{
Let $U \in \mathcal{V}$ be arbitrary.
Then $U$ is contained entirely within one of $W_1,W_2$.}

\medskip
\noindent
To see this, let $U_i := U \cap W_i$ for $i=1,2$.
Assume, for a contradiction, that $U_1,U_2 \neq \emptyset$.
Then%
\COMMENT{Final inequality: since $\beta \leq \nu/2 \leq \tau'/2 \leq \alpha^2/40$ we have $\alpha^2/4-\beta \geq 9\alpha^2/40 \geq \alpha^2/5$.}
\begin{align}\label{Uilarge}
|U_i| &\geq \delta(G'[U_i]) \stackrel{{\rm (D4)},(\ref{sparseacross})}{\geq} \frac{D}{k'+\ell'}-\beta n \stackrel{{\rm (D6)}}{\geq} \frac{D}{2(1+\rho^{1/3})n/D} - \beta n \geq (\alpha^2/4 - \beta)n\\
\nonumber &\geq \alpha^2n/5.
\end{align}
In particular this implies that $\tau'|U| \leq |U_i| \leq (1-\tau')|U|$.
The fact that $\beta \leq \nu/2$ and (\ref{sparseacross}) imply that $RN_{\nu,U}(U_i) \subseteq U_i$.
Then $U$ cannot be a robust expander component.
So $U$ is a bipartite robust expander component, with bipartition $A,B$, say.
Let $A_i := A \cap U_i$ for $i=1,2$ and define $B_i$ analogously.
Similarly as in (\ref{Uilarge}), using (D5) instead of (D4), one can show that $|A_i|,|B_i| \geq (\alpha^2/8-\beta)n \geq \alpha^2n/10$.%
\COMMENT{Since $\beta \leq \alpha^2/40$.}
In particular, $\tau'|A| \leq |A_i| \leq (1-\tau')|A|$.
Without loss of generality, suppose that $|A_1|-|B_1| \geq |A_2|-|B_2|$.
Then (C2) implies that $|A_1|-|B_1| \geq -\rho' n$ and so $|RN_{\nu,U}(A_1) \cap B| \leq |B_1| \leq |A_1| + \rho' n < |A_1| + \nu |U|$, a contradiction.%
\COMMENT{By the definition of $f'$ we have $\rho' n \leq \nu\alpha n/2 < \nu |U|$.}
(Here we used the fact that $|U| > \alpha n/2$ and $\rho' \leq f'(\nu)$.)
This completes the proof of the claim.

\medskip
\noindent
So there is a partition $\mathcal{V}_1, \mathcal{V}_2$ of $\mathcal{V}$ such that $U \subseteq W_i$ for all $U \in \mathcal{V}_i$.
For $i=1,2$, let $k_i$ be the number of robust expander components and $\ell_i$ the number of bipartite robust expander components in $\mathcal{V}_i$.
Let $\rho := 4\rho'$.%
\COMMENT{Needed to ensure (E1),(B1) and (D7) since $|G'|=2n = 2|G|$.}
We claim that, for at least one of $i=1,2$, we have that $\mathcal{V}_i$ is a robust partition of $G$ with parameters $\rho,\nu,\tau',k_i,\ell_i$.
Suppose that, for both $i=1,2$, we have $k_i + 2\ell_i > \lfloor (1+\rho^{1/3})n/D \rfloor$.
Then
$$
k'+2\ell' \geq 2\left\lfloor \frac{(1+\rho^{1/3})n}{D} \right\rfloor + 2 > \left\lfloor \frac{2(1+\rho^{1/3})n}{D} \right\rfloor \geq \left\lfloor \frac{2(1+\rho'^{1/3})n}{D+\beta n} \right\rfloor,
$$ contradicting (D6) for $\mathcal{V}$.
So without loss of generality, we have that $\mathcal{V}_1$ satisfies (D6).
It is easy to check that the remaining properties (D1)--(D5) and (D7) are also satisfied by~$\mathcal{V}_1$.
Therefore $\mathcal{V}_1$ is a robust partition of $G$ with parameters $\rho,\nu,\tau',k_1,\ell_1$ and hence also with parameters $\rho,\nu,\tau,k_1,\ell_1$.
\hfill$\square$


\section{How to obtain a long cycle given a robust partition} \label{sec:someresults}

The main result of this section is Lemma~\ref{HES} which implies that, given a suitable set $\mathcal{P}$ of paths joining up the robust components of a robust partition, one can extend $\mathcal{P}$ into a Hamilton cycle.
Actually, in the proof of Theorem~\ref{tconnected} we will need to consider the more general notion of a weak robust subpartition, defined below.

\subsection{Definitions and the main statement}

Let $k,\ell \in \mathbb{N}$ and $0 < \rho \leq \nu \leq \tau \leq \eta < 1$.
Given a graph $G$ on $n$ vertices, we say that $\mathcal{U}$ is a \emph{weak robust subpartition in $G$ with parameters $\rho,\nu,\tau,\eta,k,\ell$} if the following conditions hold.%
\COMMENT{no need to assume regularity}
\begin{itemize}
\item[(D1$'$)] $\mathcal{U} = \lbrace U_1, \ldots, U_{k}, Z_1, \ldots, Z_{\ell} \rbrace$ is a collection of disjoint subsets of $V(G)$;
\item[(D2$'$)] for all $1 \leq i \leq k$, $G[U_i]$ is a $(\rho,\nu,\tau)$-robust expander component of $G$;
\item[(D3$'$)] for all $1 \leq j \leq \ell$, there exists a partition $A_j, B_j$ of $Z_j$ such that $G[Z_j]$ is a bipartite $(\rho,\nu,\tau)$-robust expander component with respect to $A_j,B_j$;
\item[(D4$'$)] $\delta(G[X]) \geq \eta n$ for all $X \in \mathcal{U}$;
\item[(D5$'$)] for all $1 \leq j \leq \ell$, we have $\delta(G[A_j,B_j]) \geq \eta n/2$.
\end{itemize}
A weak robust subpartition $\mathcal{U}$ is weaker than a robust partition in the sense that the graph is not necessarily regular and $\mathcal{U}$ need not involve the entire graph, and we can make small adjustments to the partition while still maintaining (D1$'$)--(D5$'$) with slightly worse parameters.%
\COMMENT{This is not possible in a robust partition since properties (D4) and (D5) would not be guaranteed. Need the extra strength of a robust partition (i.e. (D4) and (D5)) to find a tour. But only need a weak robust partition to extend the HES into a Hamilton cycle. In fact we only consider an induced subgraph of $G$ in proving Theorem~\ref{tconnected}, in which we do not have a robust partition.}
This is formalised by the following statement.

\begin{proposition} \label{WRSD-RD}
Let $k,\ell,D \in \mathbb{N}$ and suppose that $0 < 1/n \ll \rho \leq \nu \leq \tau \le \eta \le  \alpha^2/2 < 1$.
\begin{itemize}
\item[(i)] Suppose that $G$ is a $D$-regular graph on $n$ vertices where $D \geq \alpha n$. Let $\mathcal{V}$ be a robust partition of $G$ with parameters $\rho,\nu,\tau,k,\ell$. Then $\mathcal{V}$ is a weak robust subpartition in $G$ with parameters $\rho,\nu,\tau,\eta,k,\ell$.
\item[(ii)] Suppose that $H$ is a graph and $\mathcal{U}$ is a weak robust subpartition in $H$ with parameters $\rho,\nu,\tau,\eta,k,\ell$.
Let $\mathcal{U}' \subseteq \mathcal{U}$ be non-empty.
Then $\mathcal{U}'$ is a weak robust subpartition in $H$ with parameters $\rho,\nu,\tau,\eta,k',\ell'$ for some $k' \leq k$ and $\ell' \leq \ell$.
\end{itemize}
\end{proposition}

\begin{proof}
We only prove (i) since (ii) is clear. Note that properties (D1$'$)--(D3$'$) are immediate.
Note (D6) implies that
$
k+2\ell \leq \lfloor (1+\rho^{1/3})/\alpha \rfloor \leq 2/\alpha.
$
So $D/(k+\ell) \geq \alpha^2 n/2$.
Together with (D4) and (D5) this shows that
(D4$'$) and (D5$'$) hold.
This completes the proof.
\end{proof}

For a path system $\mathcal{P}$, we say that a vertex $x$ is an \emph{endpoint} of $\mathcal{P}$ if $x$ is an endpoint of some path in $\mathcal{P}$. 
Define the \emph{internal vertices} of $\mathcal{P}$ similarly.
If every endpoint of a path system $\mathcal{P}$ lies in some $U \subseteq V(G)$, we say that $\mathcal{P}$ is \emph{$U$-anchored}.
When $\mathcal{U}$ is a collection of disjoint subsets of $V(G)$, we say that $\mathcal{P}$ is \emph{$\mathcal{U}$-anchored} if it is $\bigcup_{U \in \mathcal{U}}U$-anchored.
Given a path $P$ in $G$, we say that $P'$ is an \emph{extension} of $P$ if $P'$ is a path which contains $P$ as a subpath.
An \emph{Euler tour} in a (multi)graph is a closed walk that visits every vertex and uses each edge exactly once.

Given a graph $G$ with $U \subseteq V(G)$ and a path system $\mathcal{P}$ in $G$, we write $\sideset{}{_\mathcal{P}}\End(U)$ and $\sideset{}{_\mathcal{P}}\Int(U)$ for, respectively, the number of endpoints/internal vertices of $\mathcal{P}$ which lie in $U$.
Given disjoint sets $A,B \subseteq V(G)$, we say that $\mathcal{P}$ is \emph{$(A,B)$-balanced} if
\begin{itemize}
\item $\sideset{}{_\mathcal{P}}\End(A) = \sideset{}{_\mathcal{P}}\End(B) > 0$; and
\item $|A| - \sideset{}{_\mathcal{P}}\Int(A) = |B| - \sideset{}{_\mathcal{P}}\Int(B)$.
\end{itemize}

Suppose that $G$ is a graph and $\mathcal{U}$ is a collection of disjoint subsets of $V(G)$.
Let $\mathcal{P}$ be a $\mathcal{U}$-anchored path system in $G$ (so all endpoints of the paths in $\mathcal{P}$ lie in $\bigcup_{U \in \mathcal{U}}U$).
We define the \emph{reduced multigraph $R_{\mathcal{U}}(\mathcal{P})$ of $\mathcal{P}$ with respect to $\mathcal{U}$} to be the multigraph with vertex set $\mathcal{U}$ in which we add a distinct edge between $U,U' \in \mathcal{U}$ whenever $\mathcal{P}$ contains a path with one endpoint in $U$ and one endpoint in $U'$.
So $R_{\mathcal{U}}(\mathcal{P})$ might contain loops.

Let $k,\ell \in \mathbb{N}$, let $0 < \rho \leq \nu \leq \tau \leq \eta < 1$ and let $0 < \gamma < 1$.
Suppose that $G$ is a graph on $n$ vertices with a weak robust subpartition $\mathcal{U} = \lbrace U_1, \ldots, U_k, Z_1, \ldots, Z_\ell \rbrace$ with parameters $\rho,\nu,\tau,\eta,k,\ell$, so that the bipartition of $Z_j$ specified by (D3$'$) is $A_j, B_j$.
We say that $\mathcal{P}$ is a \emph{$\mathcal{U}$-tour with parameter $\gamma$} if
\begin{itemize}
\item[(T1)] $\mathcal{P}$ is a $\mathcal{U}$-anchored path system;
\item[(T2)] $R_{\mathcal{U}}(\mathcal{P})$ has an Euler tour;
\item[(T3)] for all $U \in \mathcal{U}$ we have $|V(\mathcal{P}) \cap U| \leq \gamma n$;
\item[(T4)] for all $1 \leq j \leq \ell$, $\mathcal{P}$ is $(A_j, B_j)$-balanced.
\end{itemize}

We will often think of $R_{\mathcal{U}}(\mathcal{P})$ as a walk rather than a multigraph.
So in particular, we will often say that `$R_{\mathcal{U}}(\mathcal{P})$ is an Euler tour'.
The aim of this section is to prove the following lemma, stating that every graph with a weak robust subpartition $\mathcal{U}$ and a $\mathcal{U}$-tour contains a cycle which covers every vertex within the components of $\mathcal{U}$.

\begin{lemma} \label{HES}
Let $k,\ell,n \in \mathbb{N}$ and suppose that $0 < 1/n \ll \rho, \gamma \ll \nu \leq \tau \ll \eta < 1$.%
\COMMENT{it's easier in the proofs of Thms~\ref{main} and~\ref{tconnected} if we do not have a linear hierarchy.}
Suppose that $G$ is a graph on $n$ vertices and that $\mathcal{U}$ is a weak robust subpartition in $G$ with parameters $\rho,\nu,\tau,\eta,k,\ell$.
Suppose further that $G$ contains a $\mathcal{U}$-tour $\mathcal{P}$ with parameter $\gamma$.
Then there is a cycle in $G$ which contains $\mathcal{P}$ and every vertex in $\bigcup_{U \in \mathcal{U}}U$.%
\COMMENT{AL: added $\mathcal{P}$.}
\end{lemma}

Since by Proposition~\ref{WRSD-RD}(i) every robust partition is also a weak robust subpartition, Lemma~\ref{HES} immediately implies the following result which will be used in the proof of Theorem~\ref{main} while for the proof of Theorem~\ref{tconnected} we will need Lemma~\ref{HES} itself.

\begin{corollary} \label{HEScor}
Let $k,\ell, n,D \in \mathbb{N}$ and suppose that $0 < 1/n \ll \rho,\gamma \ll \nu \leq \tau \ll \alpha < 1$.
Suppose that $G$ is a $D$-regular graph on $n$ vertices where $D \geq \alpha n$, with a robust partition $\mathcal{V}$ with parameters $\rho,\nu,\tau,k,\ell$.
Suppose further that $G$ contains an $\mathcal{V}$-tour with parameter $\gamma$.
Then $G$ contains a Hamilton cycle.
\end{corollary}

The following corollary of Lemma~\ref{HES} will be used in~\cite{paper2}.

\begin{corollary} \label{biprobexpcor}
Let $n \in \mathbb{N}$ and suppose that $0 < 1/n \ll \gamma \ll \eta < 1/2$.
Suppose that $G$ is a graph and let $A,B,V_0$ be a partition of $V(G)$ with $|A| = |B|= n$.
Let $H$ be a spanning subgraph of $G[A,B]$ such that $\delta(H) \geq (1/2 + \eta )n$.
Suppose further that $G$ contains an $(A,B)$-balanced path system $\mathcal{P}$ with $|V(\mathcal{P}) \cap (A \cup B) |  \le \gamma n$ and
such that every vertex in $V_0$ lies in the interior of some path in~$\mathcal{P}$.
Then $G$ contains a Hamilton cycle $C$ with $E(\mathcal{P}) \subseteq E(C)$ and $E(C)\setminus E(\mathcal{P})  \subseteq E(H)$.
\end{corollary}

\begin{proof}
Let $ \nu, \tau$ be new constants such that $\gamma \ll \nu \ll \tau \ll \eta$. Let $\mathcal{P'}$ be the path system obtained
from $\mathcal{P}$ by iterating the following process: if $uvw$ is a subpath in $\mathcal{P}$ with $v \in V_0$, then we replace $uvw$ with an edge $uw$.
So $V(\mathcal{P'}) \subseteq A \cup B$. Let $H' := H \cup \bigcup_{P \in \mathcal{P'}} P$. Let $G''$ be the graph obtained from $G\setminus V_0$ by
deleting all edges in $G[A,B]-H$ and let $G':=G''\cup \bigcup_{P \in \mathcal{P'}} P$.
Note that $H'$ is $\gamma$-close to bipartite (when viewed as a subgraph of~$G'$)%
\COMMENT{KS: (C1) and (C2) obvious. The contraction of paths can only decrease the number of edges. $e_{H'}(A,\overline{B}) + e_{H'}(\overline{A},B) = e_{\mathcal{P}}(A,\overline{B}) + e_{\mathcal{P}}(\overline{A},B) \leq 2|V(\mathcal{P}) \cap (A \cup B)| \leq 2\gamma n \leq \gamma n^2$. So (C3) holds.}
 and $\delta(H') \geq (1/2 + \eta)n$.
We claim that $H'$ is a bipartite robust $(\nu, \tau)$-expander with bipartition $A,B$.
To see this, it suffices to show that $H$ has this property.
Consider any set $S \subseteq A$ with $\tau n \leq |S| \leq (1-\tau)n$.
Suppose first that $|S| \geq n/2$.
Then every vertex in $B$ has at least $\eta n \geq \nu n$ neighbours in $S$.
So $RN_{\nu,H}(S) = B$.	Thus we may assume that $|S| \leq n/2$.
Let $N := RN_{\nu,H}(S) \cap B$. Then
$$
(1/2+\eta)n|S| \leq e_{H}(S,N) + e_{H}(S,B \setminus N) \leq |S||N| + \nu n^2 \leq |S||N| + \nu n|S|/\tau 
$$
and so $|N| \geq (1/2 + \eta - \nu/\tau)n \geq (1+\eta)n/2 \geq |S| + \nu n$, as required.
Since $V(H') = V(G')$ it follows that $\mathcal{U} = \{ V(G') \}$ is a weak robust subpartition in $G'$ with parameters $\gamma,\nu,\tau,1/4,0,1$.%
    \COMMENT{don't get $1/2+\eta$ since $|G'|=2n$}
Moreover, $\mathcal{P'}$ is a $\mathcal{U}$-tour with parameter $\gamma$.
Lemma~\ref{HES} with $\gamma, 1/4$ playing the roles of $\rho,\eta$ implies that there is a Hamilton cycle $C'$ in $G'$ which contains $\mathcal{P'}$.
$C'$ corresponds to the required Hamilton cycle $C$ in $G$. 
\end{proof}

The remainder of this section is devoted to the proof of Lemma~\ref{HES}.

\subsection{Spanning path systems in robust expanders}

In this subsection, we prove Corollary~\ref{robexppaths}, which states that when $p$ is not too large, every robust expander $G$ is Hamilton $p$-linked, i.e.
given distinct vertices $y_1, y_1',\dots,y_p,y'_p$, there exist $p$ vertex-disjoint paths joining $y_i$ to $y_i'$ for all $i\le p$ such that together these paths cover all the vertices of $G$.
This, combined with a bipartite analogue in the next subsection, will be the main tool in proving Lemma~\ref{HES}: the $y_i$ and $y_i'$ will be suitable endpoints of the paths in the $\mathcal{U}$-tour $\mathcal{P}$.

We now define an analogue of robust expansion for digraphs. Let $0 < \nu \leq \tau < 1$. Given any digraph $G$ on $n$ vertices and $S \subseteq V(G)$, the \emph{$\nu$-robust outneighbourhood} $RN_{\nu,G}^+(S)$ of $S$ is the set of all those vertices of $G$ which have at least $\nu n$ inneighbours in $S$. $G$ is called a \emph{robust $(\nu,\tau)$-outexpander} if $|RN^+_{\nu,G}(S)| \geq |S| + \nu n$ for all $S \subseteq V(G)$ with $\tau n \leq |S| \leq (1-\tau)n$.

The next lemma is a directed analogue of Lemma~\ref{expanderswallow}.
Its proof follows immediately from the definition.

\begin{lemma} \label{outexpanderswallow}
Let $0 < \nu \ll \tau \ll 1$.
Suppose that $G$ is a digraph and $U \subseteq W \subseteq V(G)$ are such that $G[U]$ is a robust $(\nu,\tau)$-outexpander and $|U \setminus W| \leq \nu|U|/2$.
Then $G[W]$ is a robust $(\nu/2,2\tau)$-outexpander.%
\COMMENT{this had to be reincluded to prove Corollary~\ref{orderedham} -- that a robust outexpander is $k$-ordered Hamilton.}
\end{lemma}

The next lemma shows that the diameter of a robust outexpander is small.
Again, it follows immediately from the definition of robust outexpansion.%
\COMMENT{
To prove (i),
let $X_i$ be the set of vertices $v$ for which there is a directed walk from $x$ to $v$ in $G$ of length at most $i$.
So $X_0 = \lbrace x \rbrace$ and $X_1 = N^+(x) \cup \lbrace x \rbrace$.
So $|X_1| \geq \eta n$.
Note that $RN^+_{\nu,G}(X_i) \subseteq X_{i+1}$.
So $|X_{i+1}| \geq |RN^+_{\nu,G}(X_i)| \geq |X_i| + \nu n$.
So certainly for $i' := \lfloor 1/\nu \rfloor - 2$ we have that $|X_{i'}| \geq (1-\tau)n$.
But since $\delta^0(G) \geq \eta n \geq \tau n$ we have that $X_{i'+1} = V(G)$.
In particular, this implies that for any $y \neq x$ there is a path $P$ of length at most $1/\nu-1$ between $x$ and $y$ in $G$. 
Therefore $|V(P)| \leq 1/\nu$.
The proof of (ii) is practically identical.
}

\begin{lemma} \label{shortpath}
Let $n \in \mathbb{N}$ and $0 < 1/n \ll \nu \ll \tau \ll \eta \leq 1$.
Suppose that $G$ is a robust $(\nu,\tau)$-outexpander on $n$ vertices with $\delta^0(G) \geq \eta n$.
Then, given any distinct vertices $x,y \in V(G)$, there exists a path $P$ in $G$ from $x$ to $y$ such that $|V(P)| \leq 1/\nu$.
\end{lemma}

We will need the following result of K\"uhn, Osthus and Treglown \cite{kot}, which states that a robust outexpander whose minimum degree is not too small contains a (directed) Hamilton cycle.

\begin{theorem}[\cite{kot}] \label{kot} 
Let $n \in \mathbb{N}$ and suppose that $0 < 1/n \ll \nu \leq \tau \ll \eta < 1$. Let $G$ be a robust $(\nu,\tau)$-outexpander on $n$ vertices with $\delta^0(G) \geq \eta n$. Then $G$ contains a Hamilton cycle.
\end{theorem}

We say that a digraph $G$ is \emph{$p$-ordered Hamilton} if, given $x_1, \ldots, x_p \in V(G)$, $G$ contains a Hamilton cycle which traverses $x_1, \ldots, x_p$ in this order.

\begin{corollary} \label{orderedham}
Let $n,p \in \mathbb{N}$ and suppose that $0 < 1/n \ll \nu \ll \tau \ll \eta < 1$ and $p \leq \nu^3 n$. Let $G$ be a robust $(\nu,\tau)$-outexpander on $n$ vertices with $\delta^0(G) \geq \eta n$. 
Then $G$ is $p$-ordered Hamilton.
\end{corollary}

\begin{proof}
Let $x_1, \ldots, x_p \in V(G)$.
We claim that we can find a path $P$ in $G$ joining $x_1, x_p$ which traverses $x_1, \ldots, x_p$ in this order and such that $|V(P)| \leq \nu n/2$. To see this, suppose for some $i \leq p-1$ we have found a path $P_i$ joining $x_1,x_i$ with $|V(P_i)| \leq 2i/\nu$ which traverses $x_1, \ldots, x_i$ in this order and such that $x_{i+1}, \ldots, x_p$ do not lie in $P_i$.
Let $G_i := G \setminus ( ( V(P_i) \setminus \lbrace x_i \rbrace ) \cup \lbrace x_{i+2}, \ldots, x_p \rbrace)$.
Note that $n - |V(G_i)| \leq 2p/\nu \leq \nu n/2$.
So Lemma~\ref{outexpanderswallow} implies that $G_i$ is a robust $(\nu/2,2\tau)$-outexpander.
Apply Lemma~\ref{shortpath} with $G_i, x_i, x_{i+1}$ playing the roles of $G,x,y$ to obtain a path, which, when appended to $P_i$, gives a path $P_{i+1}$ joining $x_1, x_{i+1}$ which traverses $x_1, \ldots, x_{i+1}$ in this order such that $x_{i+2}, \ldots, x_p$ do not lie in $P_{i+1}$ and $|V(P_{i+1})| \leq |V(P_i)| + 2/\nu \leq 2(i+1)/\nu$.
Set $P := P_p$.
This proves the claim.

Let $G'$ be the graph obtained from $G \setminus V(P)$  by adding a new vertex $z$ such that $N^-_{G'}(z) := N^-_{G \setminus V(P)}(x_1)$ and $N^+_{G'}(z) := N^+_{G \setminus V(P)}(x_p)$.
Then $\delta^0(G')\geq\delta^0(G)-\nu n/2 \geq \eta |G'|/2$ and $G'$ is a robust $(\nu/2,2\tau)$-outexpander.
Therefore we can apply Theorem~\ref{kot} to find a directed Hamilton cycle in $G'$.
This corresponds to a Hamilton cycle in $G$ which traverses $x_1, \ldots, x_p$ in this order.
\end{proof}

The following corollary states that robust (out)expanders are Hamilton $p$-linked provided that $p$ is not too large.

\begin{corollary} \label{robexppaths}
Let $n,p \in \mathbb{N}$ and suppose that $0 < 1/n \ll \nu \ll \tau \ll \eta < 1$ and let $p \leq \nu^4n$. 
\begin{itemize}
\item[(i)] Let $G$ be a robust $(\nu,\tau)$-outexpander on $n$ vertices with $\delta^0(G) \geq \eta n$. 
Then $G$ is Hamilton $p$-linked.
\item[(ii)] Let $H$ be a robust $(\nu,\tau)$-expander on $n$ vertices with $\delta(H) \geq \eta n$. 
Then $H$ is Hamilton $p$-linked.
\end{itemize}
\end{corollary}

\begin{proof}
To prove (i), let $y_1, \ldots, y_p, y_1', \ldots, y_p' \in V(G)$.
Obtain $G^*$ from $G$ as follows. For each $1 \leq i \leq p$ (where indices are considered modulo $p$), replace the pair $y_{i+1},y_i'$ with a new vertex $z_i$ such that $N^+_{G^*}(z_i) := N^+_G(y_{i+1})$ and $N^-_{G^*}(z_i) := N^-_G(y_i')$.%
\COMMENT{strictly speaking $G^*$ should be obtained iteratively, successively replacing pairs with a new vertex whose neighbourhood lies in the `current' graph rather than $G$.}
Then it is easy to see that $G^*$ is a robust $(\nu/2,2\tau)$-outexpander.
Corollary~\ref{orderedham} implies that $G^*$ contains a Hamilton cycle which traverses $z_1, \ldots, z_p$ in this order.
This corresponds to a collection $P_1, \ldots, P_p$ of vertex-disjoint paths such that $P_i$ joins $y_i$ to $y_i'$ and all the $P_i$ together cover $V(G)$, proving (i).

To prove (ii), let $G$ be the digraph obtained from $H$ by replacing each edge $xy$ with directed edges $\overrightarrow{xy}$ and $\overrightarrow{yx}$.
Then $G$ is a robust $(\nu,\tau)$-outexpander with $\delta^0(G) \geq \eta n$.
Now (i) implies that $G$ is Hamilton $p$-linked.
For each $xy \in E(H)$, any path system in $G$ uses at most one of $\overrightarrow{xy}, \overrightarrow{yx}$.
So $H$ is Hamilton $p$-linked. 
\end{proof}


\subsection{Spanning path systems in bipartite robust expanders}

Given $p \in \mathbb{N}$ and a bipartite graph $G$ with vertex classes $A,B$, we say that $G$ is \emph{$(A,B)$-Hamilton $p$-linked} if, given any $Y := \lbrace y_1, y_1', y_2, y_2', \ldots, y_p, y_p' \rbrace \subseteq V(G)$ with $|Y \cap A| = |Y \cap B|=p$, we can find a set of vertex-disjoint paths joining $y_i$ to $y_i'$ in $G$ such that together these paths cover all the vertices of $G$.
Note that if $G$ is $(A,B)$-Hamilton $p$-linked then it is balanced.
In this subsection we show that, for $p$ not too large, $G$ is $(A,B)$-Hamilton $p$-linked when $G$ is a balanced bipartite robust expander.

Given a balanced bipartite graph $G$ with vertex classes $A,B$ which contains a perfect matching $M$, we denote by $G^*$ the \emph{$M$-auxiliary digraph of $G$} obtained from $G$ as follows.
Let $G^*$ have vertex set $B$. For each $v \in B$, we let $v'$ be the unique vertex of $A$ such that $vv' \in M$.
Then, for all $x,v \in B$, we let $\overrightarrow{vx} \in E(G^*)$ if and only if $x \in N_{G}(v') \setminus \lbrace v \rbrace$. 
Note that the order of $A$ and $B$ matters here.

\begin{lemma} \label{anythingREcandoBREcandobetter}
Let $n \in \mathbb{N}$ and $0 < 1/n \ll \nu \ll \tau \ll \eta < 1$.
Let $G$ be a balanced bipartite graph with vertex classes $A,B$ so that $|A|=|B|=n$ and $\delta(G) \geq \eta n$.
Suppose further that $G$ is a bipartite robust $(\nu,\tau)$-expander (with bipartition $A,B$).
Then
\begin{itemize}
\item[(i)] $G$ contains a perfect matching $M$;
\item[(ii)] the $M$-auxiliary digraph $G^*$ of $G$ is a robust $(\nu,\tau)$-outexpander with minimum degree at least $\eta n/2$.
\end{itemize}
\end{lemma}

\begin{proof}
Observe that (i) follows immediately from Hall's Theorem.%
\COMMENT{
By Hall's Theorem, it suffices to show that whenever $S$ is a proper subset of $A$, we have that $|N_{G}(S)| \geq |S|$.
Suppose first that $|S| \leq \tau n$.
Then $|N_{G}(S)| \geq \delta(G) \geq \eta n \geq \tau n \geq |S|$.
Suppose instead that $\tau n \leq |S| \leq (1-\tau)n$.
Then $|N_{G}(S)| \geq |RN_{\nu,G}(S)| \geq |S| + \nu n \geq |S|$.
Finally, suppose that $|S| \geq (1-\tau) n$.
Then $N_{G}(S) = B$ since $ \tau \leq \eta$ and $\delta(G) \geq \eta n$.
So certainly $|N_{G}(S)| \geq |S|$ in this case.
Therefore $G$ contains a perfect matching $M$.}
Write $M := \lbrace xx' : x \in B, x' \in A \rbrace$.
To prove (ii), note that
$\delta^0(G^*) \geq \delta(G) - 1 \geq \eta n/2$.
Consider any $S \subseteq B$ with $\tau n \leq |S| \leq (1-\tau)n$.
Let $S_A := \lbrace x' : x \in S \rbrace$ and note that
$RN^+_{\nu,G^*}(S) \supseteq RN_{\nu,G}(S_A)$.%
\COMMENT{
we really get the same parameter $\nu$ here since if $x \in RN_{\nu,G}(S_A)$ then $x$ has at least $\nu|G| = 2\nu n$ neighbours in $S_A$.
}
Thus
\begin{align*}
|RN^+_{\nu,G^*}(S)| &\geq |RN_{\nu,G}(S_A)| \geq |S_A| + \nu|V(G)| \geq |S| + \nu|V(G^*)|,
\end{align*}
and therefore $G^*$ is a robust $(\nu,\tau)$-expander, proving (ii).
\end{proof}

We now prove an analogue of Lemma~\ref{shortpath} for bipartite robust expanders.

\begin{lemma} \label{bipshortpath}
Let $n \in \mathbb{N}$ and $0 < 1/n \ll \nu \ll \tau \ll \eta < 1$.
Suppose that $G$ is a bipartite graph on $n$ vertices with vertex classes $A,B$, where $||A|-|B|| \leq \nu^2 n$.
Suppose further that $\delta(G) \geq \eta n$ and $G$ is a bipartite robust $(\nu,\tau)$-expander (with bipartition $A,B$).
Then, given any distinct vertices $x,y \in V(G)$ there exists a path $P$ between $x$ and $y$ in $G$ such that $|V(P)| \leq 4/\nu$.
\end{lemma}

\begin{proof}
Consider each $u \in \lbrace x,y \rbrace$.
If $u \in B$, let $u'$ be a neighbour of $u$ which lies in $A$.
If $u \in A$, let $u' := u$.
Make these choices so that $x', y'$ are distinct.
So $\lbrace x',y' \rbrace \subseteq A$.
Remove at most $||A|-|B|| \leq \nu|A|/2$ vertices from $A \cup B$ to obtain $A' \subseteq A$ and $B' \subseteq B$ such that $|A'|=|B'|$ and $\lbrace x',y' \rbrace \subseteq A'$.
Lemma~\ref{bipexpanderswallow}(i) implies that $G' := G[A',B']$ is a bipartite robust $(\nu/2,2\tau)$-expander and that $\delta(G') \geq \eta n'/2$
where $n' := |V(G')|$.

Let $X_i$ be the set of vertices $v \in A'$ of distance at most $2i$ to $x'$ in $G'$.
Now Lemma~\ref{anythingREcandoBREcandobetter}(i) implies that $G'$ contains a perfect matching $M$.
So for all $i \geq 0$ we have $X_{i+1} \supseteq \lbrace a \in A': ab \in M, b \in N_{G'}(X_{i}) \rbrace$.
Thus $|X_1| \geq \eta n'/2$ and whenever $i \geq 1$ and $|X_i| < (1-\tau)n$ then
$$
|X_{i+1}| \geq |N_{G'}(X_{i})| \geq |RN_{\nu/2,G'}(X_{i})| \geq |X_{i}| + \nu n'/2.
$$
So certainly for $i' := \lfloor 2/\nu \rfloor - 4$%
\COMMENT{need $\nu \ll \eta$ here} we have that $|X_{i'}| \geq (1-\tau)n'$.
But since $\delta(G') \geq \eta n'/2 \geq \tau n'$ we have that $X_{i'+1} = A'$.
In particular, this implies that there is a path of length at most $4/\nu-5$ between $x'$ and $y'$ in $G'$ and hence a path $P$ with $|V(P)| \leq 4/\nu$ between $x$ and $y$ in $G$.
\end{proof}

The following is a bipartite analogue of Corollary~\ref{robexppaths}.
To prove it, we iterate Lemma~\ref{bipshortpath} to find short paths between a small number of pairs of vertices.
Then the graph obtained by deleting these paths is still a bipartite robust expander.

\begin{lemma} \label{biprobexppaths}
Let $n,p \in \mathbb{N}$, $0 < 1/n \ll \nu \ll \tau \ll \eta \leq 1$ and $p \leq \nu^4 n$.
Suppose that $G$ is a bipartite graph vertex classes $A,B$, so that $|A| = |B| = n$.
Suppose further that $G$ is a bipartite robust $(\nu,\tau)$-expander with $\delta(G) \geq \eta n$.
Then $G$ is $(A,B)$-Hamilton $p$-linked.
\end{lemma}

\begin{proof}
Let $Y := \lbrace y_1, y_1', y_2, y_2', \ldots, y_p, y_p' \rbrace$ be a collection of distinct vertices in $G$ such that $|Y \cap A| = |Y \cap B|$.
For each $1 \leq i,j \leq p$, let $W_i := \lbrace y_i, y_i' \rbrace$ and let $W_{\geq j} := \bigcup_{j \leq i \leq p}W_i$.
Suppose, for some $0 \leq \ell \leq p-2$, we have already obtained vertex-disjoint paths $R_1, \ldots, R_\ell$, where for each $1 \leq i \leq \ell$,
$R_i$ has endpoints $y_i, y_i'$ and $|V(R_i)| \leq 8/\nu$.
We obtain $R_{\ell+1}$ as follows.
Let
$$
G_\ell := G \setminus (V(R_1) \cup \ldots \cup V(R_{\ell-1}) \cup W_{\geq \ell+1})
$$
and let $n_{\ell} := |V(G_{\ell})|$.
Note that
\begin{align}
\label{nell1} |V(G) \setminus V(G_\ell)| &= \sum\limits_{1 \leq i \leq \ell}|V(R_i)| + |W_{\geq \ell+1}| \leq 8p/\nu \leq \nu^2 n.
\end{align}
Let $A_{\ell} := A \cap V(G_{\ell})$ and define $B_{\ell}$ analogously.
Then Lemma~\ref{bipexpanderswallow}(i) implies that $G_{\ell}$ is a bipartite robust $(\nu/2,2\tau)$-expander with bipartition $A_\ell, B_\ell$, and $\delta(G_{\ell}) \geq \eta n_\ell/4$.
Moreover
$
||A_{\ell}| - |B_\ell|| \leq |V(G) \setminus V(G_\ell)| \leq \nu^2 n \le \nu^2 n_{\ell}.
$
Therefore we can apply Lemma~\ref{bipshortpath} with $A_\ell,B_\ell,\nu/2,2\tau,\eta/4$ playing the roles of $A,B,\nu,\tau,\eta$ to see that $G_\ell$ contains a path $R_{\ell+1}$ between $y_{\ell+1}$ and $y_{\ell+1}'$ such that $|V(R_{\ell+1})| \leq 8/\nu$.

Therefore we can obtain vertex-disjoint paths $R_1, \ldots, R_{p-1}$ in $G \setminus \lbrace y_p, y_p' \rbrace$ such that $|V(R_i)| \leq 8/\nu$  and $R_i$ joins $y_i,y_i'$ for all $1 \leq i \leq p-1$.
To obtain $R_p$, we now consider three cases depending on the classes in which $y_p,y_p'$ lie.
Let $V^* := \bigcup_{1 \leq i \leq p-1}V(R_i)$.

\medskip
\noindent
\textbf{Case 1.}
\emph{$y_p \in A$ and $y_p' \in B$.}

\medskip
\noindent
Using our assumption that $|Y \cap A| = |Y \cap B|$, it is easy to see that $|V^* \cap A| = |V^* \cap B|$.
Let $G' := G \setminus \left( V^* \cup \lbrace y_p,y_p' \rbrace \right)$.
Also let $A' := A \cap V(G')$ and define $B'$ analogously.
Then $|A'|=|B'| =: n'$.
As above,%
\COMMENT{
Indeed, if $R_i$ has endpoints in $A'$ and $B'$ then $|V(R_i) \cap A'| = |V(R_i) \cap B'|$. If $R_i$ has two endpoints in $A'$ then $|V(R_i) \cap A'| = |V(R_i) \cap B'| + 1$, but the number of $R_i$ with two endpoints in $A'$ equals the number of $R_i$ with two endpoints in $B'$, for which we have a similar result.
} 
$G'$ is a bipartite robust $(\nu/2,2\tau)$-expander with respect to $A',B'$, and $\delta(G') \ge \eta (n'+1)/2 $.%
	\COMMENT{We choose $\delta(G') \ge \eta (n'+1)/2$ on purpose, so $\delta(G^-) \geq \eta (n'+1)/2$.}
Therefore $G'$ contains a perfect matching $M'$ by Lemma~\ref{anythingREcandoBREcandobetter}(i).
Let $M'' := M' \cup \lbrace y_py_p' \rbrace$.
Then $M''$ is a perfect matching in the graph $G^-$ obtained from $G \setminus V^*$ by adding the edge $y_py_p'$ if necessary.
Note that $|G^-|=2(n'+1)$ and $\delta(G^-) \geq \eta (n'+1)/2$.
Let $G''$ be the $M''$-auxiliary digraph of $G^-$.
Then Lemma~\ref{anythingREcandoBREcandobetter}(ii) implies that $G''$ is a robust $(\nu/2,2\tau)$-outexpander with minimum degree at least $\eta (n'+1)/4$.
By Theorem~\ref{kot}, $G''$ contains a Hamilton cycle $C$.
Then $C$ corresponds to a Hamilton path $R_p$ in $G \setminus V^*$ which joins $y_p$ and $y_p'$. 
Thus $R_1, \ldots, R_p$ are vertex-disjoint from each other, join $y_i$ to $y_i'$, and together cover all the vertices of $G$.
So $G$ is $(A,B)$-Hamilton $p$-linked.

\medskip
\noindent
\textbf{Case 2.}
\emph{$y_p,y_p' \in A$.}

\medskip
\noindent
So it is easy to see that $|V^* \cap A| = |V^* \cap B| - 1$.
Choose a neighbour $z_p$ of $y_p'$ in $B$ which does not lie in $V^*$.
Now delete $y_p'$ from $G$ and proceed as above with $z_p$ playing the role of $y_p'$.

\medskip
\noindent
\textbf{Case 3.}
\emph{$y_p,y_p' \in B$.}

\medskip
\noindent
This is analogous to Case 2.
\end{proof}


\subsection{Proof of Lemma~\ref{HES}}

We are now ready to prove Lemma~\ref{HES}.
Given a robust subpartition $\mathcal{U}$ in $G$ and a $\mathcal{U}$-tour $\mathcal{P}$, we apply Corollary~\ref{robexppaths} within each robust expander component $U$ of $\mathcal{U}$, with the endpoints of $\mathcal{P}$ which lie in $U$ suitably ordered.
Similarly, we apply Lemma~\ref{biprobexppaths} within each bipartite robust expander component $Z$ of $\mathcal{U}$.
In this way, we obtain a set $\mathcal{R}$ of `joining paths'.
Then together the paths in $\mathcal{P} \cup \mathcal{R}$ form a cycle containing every vertex of $\bigcup_{U \in \mathcal{U}}U$.

\medskip
\noindent
\emph{Proof of Lemma~\ref{HES}.}
Note that if $\nu' \leq \nu$, then any (bipartite) robust $(\nu,\tau)$-expander is also a (bipartite) robust $(\nu',\tau)$-expander.
So without loss of generality, we may assume that $\nu \ll \tau$.
Write $\mathcal{U} := \lbrace U_1, \ldots,$ $U_k, Z_1, \ldots, Z_\ell \rbrace$ so that (D1$'$)--(D5$'$) are satisfied.
Let $\mathcal{P}$ be a $\mathcal{U}$-tour with parameter $\gamma$,
let $q := |\mathcal{P}|$ and
$R := R_{\mathcal{U}}(\mathcal{P})$.
So for each path $P \in \mathcal{P}$ there is a unique edge $e_P$ in $R$.
Without loss of generality, $e_{P_1} \ldots e_{P_q}$ is the Euler tour guaranteed by (T2).
This corresponds to an ordering $P_1, \ldots, P_q$ of the paths in $\mathcal{P}$.
 Direct the edges of $R$ so that $e_{P_1} \ldots e_{P_q}$ is a directed tour.
Direct the edges of (the paths in) $\mathcal{P}$ correspondingly, so that for all $1 \leq s \leq q$, if $e_{P_s}$ has startpoint $U$ and endpoint $W$, then $P_s$ is a directed path from some vertex $x_s^- \in U$ to some vertex $x_s^+ \in W$.
We thus obtain an ordering $x_1^+, x_2^-, x_2^+, \ldots, x_q^-, x_q^+, x_1^-$ of the endpoints of $\mathcal{P}$.
Note that for each $1 \leq i \leq q$, $x_i^+, x_{i+1}^-$ lie in the same $X \in \mathcal{U}$, where the indices are considered modulo $q$.

Fix some $U \in \mathcal{U}$.
Let $p := \sideset{}{_\mathcal{P}}\End(U)/2$. 
Thus $p \in \mathbb{N}$.
Then there exists a subsequence $i_1, \ldots, i_{p}$ of $1, \ldots, q$ such that
$$
K := (x_{i_1}^+, x_{i_1+1}^-, x_{i_2}^+, x_{i_2+1}^-, \ldots, x_{i_{p}}^+, x_{i_{p}+1}^-)
$$
is the subsequence of ordered endpoints of $\mathcal{P}$ which lie in $U$ (where $x_{q+1}^- := x_1^-$).
Let $I$ be the (unordered) collection of internal vertices of $\mathcal{P}$ which lie in $U$.
Let $U' := U \setminus I$.
Note that each element of $K$ lies in $U'$.
Now (D4$'$) implies that $\delta(G[U]) \geq \eta n$.
Furthermore, (T3) implies that $\End_{\mathcal{P}}(U) + \Int_{\mathcal{P}}(U) \leq \gamma n$. So 
\begin{equation} \label{nj}
|U'| \stackrel{{\rm (T3)}}{\geq} |U| - \gamma n \geq (\eta - \gamma)n \geq \eta n/2
\end{equation}
and hence
\begin{align} \label{pj}
p &= \sideset{}{_\mathcal{P}}\End(U)/2 \leq \gamma n/2 \leq \gamma |U'|/\eta \leq \sqrt{\gamma} |U'|;\\
\label{Ij}
\mbox{ and }~~ |I| &= \sideset{}{_\mathcal{P}}\Int(U) \leq 2 \gamma n \leq 4\gamma |U'|/\eta  \leq \nu|U'|/5. 
\end{align}

Suppose first that $U = U_i$ for some $1 \leq i \leq k$.
Then $U$ is a $(\rho,\nu,\tau)$-robust expander component.
By (\ref{Ij}), we may apply Lemma~\ref{expanderswallow} with $U,U \setminus I$ playing the roles of $U,U'$ to see that $U'$ is a robust $(\nu/2,2\tau)$-expander and $\delta(G[U']) \geq \eta n/2$. 
By (\ref{pj}) and Corollary~\ref{robexppaths}, $G[U']$ is Hamilton $p$-linked.
So there is an ordered collection $\mathcal{R}_{U}$ of $p$ vertex-disjoint paths in $G[U']$ spanning $U'$ such that the $j$th path in $\mathcal{R}_{U}$ joins $x_{i_j}^+$ and $x_{i_j+1}^-$.

Suppose instead that $U = Z_i$ for some $1 \leq i \leq \ell$.
Then there exists a bipartition $A,B$ of $U$ such that $U$ is a bipartite $(\rho,\nu,\tau)$-robust expander component with bipartition $A,B$.
Let $A' := A \setminus I$ and $B' := B \setminus I$.
So $A',B'$ is a bipartition of $U'$.
Recall from (T4) that $\mathcal{P}$ is $(A,B)$-balanced. Thus $|A'| = |B'|$ and $\sideset{}{_\mathcal{P}}\End(A') = \sideset{}{_\mathcal{P}}\End(B')>0$.

Let $n' := |A'|$.
Note that (D5$'$) implies that $\delta(G[A,B]) \geq \eta n/2$.
By (B1) and (C2) we have that $||A|-|B|| \leq \rho n$ and hence (\ref{nj}) implies that $|A| \geq 2|U'|/5$.
Now (\ref{Ij}) implies that $|I| \leq \nu|U'|/5 \leq \nu|A|/2$.
So we may apply Lemma~\ref{bipexpanderswallow}(i) to see that $G[U']$ is a bipartite robust $(\nu/2,2\tau)$-expander with bipartition $A',B'$, and that $\delta(G[A',B']) \geq \eta n'/4$.
By (\ref{pj}) and Lemma~\ref{biprobexppaths}, $H$ is $(A',B')$-Hamilton $p$-linked.
So there is an ordered collection $\mathcal{R}_{U}$ of $p$ vertex-disjoint paths in $H$ spanning $U'$ such that the $j$th path in $\mathcal{R}_{U}$ joins $x_{i_j}^+$ and $x_{i_j+1}^-$.

Proceed in this way for each $U \in \mathcal{U}$ and let $\mathcal{R} := \bigcup_{U \in \mathcal{U}}\mathcal{R}_U$. Then 
for each $1 \leq i \leq q$, there exists exactly one path $R_i$ in $\mathcal{R}$ which joins $x_i^+$ and $x_{i+1}^-$ (with indices modulo $q$). 
Let
$$
C := x_1^-P_1x_1^+R_1x_2^-P_2x_2^+ \ldots x_p^-P_px_p^+R_px_1^-.
$$
Then $C$ is a cycle in $G$ which covers $\bigcup_{U \in \mathcal{U}}U$.
\hfill$\square$


\section{The proof of Theorem~\ref{main}} \label{sec:proofmain}

Our aim is to prove Theorem~\ref{main}, i.e. that every sufficiently large 3-connected $D$-regular graph $G$ on $n$ vertices with $D \geq (1/4+\eps)n$ contains a Hamilton cycle.
By Theorem~\ref{structure} and Proposition~\ref{fewstructs}(i), $G$ has a robust partition $\mathcal{V}$ such that $(k,\ell)$ takes one of five values.
By Corollary~\ref{HEScor}, to find a Hamilton cycle it suffices to find a $\mathcal{V}$-tour.
We achieve this for each case. In the first subsection we consider the case $\ell=0$ (so $1 \leq k \leq 3$), i.e. when $G$ is a union of robust expander components. Then in Subsection~\ref{sec:balance} we prove some lemmas which are useful for the case when $\ell \geq 1$.
Finally in Subsections~\ref{sec:pathsys} and~\ref{sec:pathsys2} we consider the cases $(k,\ell) = (0,1), (1,1)$ respectively.

\subsection{Finding $\mathcal{V}$-tours in a 3-connected graph with at most three robust expander components} \label{sec:robcomp}

The main result of this section guarantees a $\mathcal{V}$-tour in a 3-connected graph $G$ which has a robust partition $\mathcal{V}$ into at most three robust expander components.

\begin{lemma} \label{easy}
Let $D, n \in \mathbb{N}$, let $0 < 1/n \ll \rho \ll \nu \ll \tau \ll \alpha < 1$ and let $D \geq \alpha n$.
Suppose that $G$ is a $D$-regular $3$-connected graph on $n$ vertices and that $\mathcal{V}$ is a robust partition of $G$ with parameters $\rho,\nu,\tau,k,0$ where $k \leq 3$.
Then $G$ contains a $\mathcal{V}$-tour with parameter $4/n$.
\end{lemma}

We will use the following proposition which is an immediate consequence of Menger's Theorem.

\begin{proposition} \label{menger}
Let $k \in \mathbb{N}$ and let $G$ be a $k$-connected graph.
Suppose that $A$ is a subset of $G$ with $|A|,|\overline{A}| \geq k$.
Then there is a matching of size $k$ between $A$ and $\overline{A}$.
\end{proposition}

Lemma~\ref{easy} is an immediate corollary of the following lemma.
To see this, note that (T4) is vacuous here.%
\COMMENT{which we state separately because it is needed for the four cliques case. If we can find a matching of size at least six between any pair of cliques $X,Y$, then we use this lemma to extend this matching into an Euler tour. We do applying this lemma to the partition of $V(G)$ obtained by merging $X$ and $Y$.}

\begin{lemma}\label{3conn}
Let $G$ be a $3$-connected graph and let $\mathcal{V}$ be a partition of $V(G)$ into at most three parts, where $|V| \geq 3$ for each $V \in \mathcal{V}$.
Then $G$ contains a path system $\mathcal{P}$ such that
\begin{itemize}
\item[(i)] $e(\mathcal{P}) \leq 4$ and $\mathcal{P} \subseteq \bigcup_{V \in \mathcal{V}}G[V,\overline{V}]$;
\item[(ii)] $R_{\mathcal{V}}(\mathcal{P})$ is an Euler tour;
\item[(iii)] for each $V \in \mathcal{V}$, if $c_i$ is the number of vertices in $V$ with degree $i$ in $\mathcal{P}$
(for $i=1,2$), then $c_1+2c_2 \in \lbrace 2,4 \rbrace$ and $c_2 \leq 1$.
\end{itemize}
\end{lemma}

\begin{proof}
Suppose first that $|\mathcal{V}| = 1$.
Let $\mathcal{P}$ consist of a single arbitrary edge.
So (i) and (iii) are clear.
Then $R_{\mathcal{V}}(\mathcal{P})$ is a loop, so (ii) holds.

Suppose instead that $|\mathcal{V}|=2$ and write $\mathcal{V} := \lbrace V,W \rbrace$.
Then Proposition~\ref{menger} implies that $G$ contains a matching $\mathcal{P}$ of size two between $V$ and $W$.
So (i) holds.
In this case, $R_{\mathcal{V}}(\mathcal{P})$ consists of exactly two $VW$-edges, so (ii) holds.
Moreover, for each $V \in \mathcal{V}$ we have $(c_1,c_2)=(2,0)$, implying (iii).

Suppose finally that $|\mathcal{V}| = 3$ and write $\mathcal{V} := \lbrace V_1, V_2, V_3 \rbrace$. 
We write $M_{ij}$ for a matching between $V_i$ and $V_j$.
Given a path system $\mathcal{P}$ in $G$, we write $c_i^j$ for the number of vertices in $V_j$ with degree $i$ in $\mathcal{P}$.
Proposition~\ref{menger} implies that there is a matching of size three between $V_1$ and $V_2 \cup V_3$.
Without loss of generality, choose $M_{12}$ such that $|M_{12}| = 2$.
By Proposition~\ref{menger} there is a matching of size three between $V_3$ and $V_1 \cup V_2$.
Therefore there exist vertex-disjoint $M_{13}, M_{23}$ such that $|M_{13}| + |M_{23}| = 3$.
Throughout the remainder of the proof, we will let $u_1,v_1,w_1,x_1$ be distinct vertices in $V_1$ and we will label vertices in other classes similarly.

\medskip
\noindent
\textbf{Case 1.}
\emph{$|M_{13}| = 3$.}

\medskip
\noindent
If $M_{13}$ contains two edges $e,e'$ that are vertex-disjoint from $M_{12}$, then we let $\mathcal{P}$ have edge-set $\lbrace e,e' \rbrace \cup M_{12}$.
So (i) holds.
Note that $R_{\mathcal{V}}(\mathcal{P})$ consists of precisely two $V_1V_2$-edges and two $V_1V_3$-edges.
Therefore (ii) holds.
Moreover, $(c_1^1,c_2^1)=(4,0)$ and $(c_1^j,c_2^j)=(2,0)$ for $j=2,3$, implying (iii).
  
Otherwise, $M_{13}$ contains exactly two edges that share endpoints with edges in $M_{12}$.
Without loss of generality, let $M_{13} := \lbrace u_1u_3, v_1v_3, w_1w_3 \rbrace$ and $M_{12} := \lbrace u_1u_2, v_1v_2 \rbrace$.
In this case, let $\mathcal{P} := \lbrace u_1u_2, v_2v_1v_3, w_1w_3 \rbrace$.
(i) is immediate, and $R_{\mathcal{V}}(\mathcal{P}) \cong C_3$ so (ii) holds.
Moreover, $(c_1^1,c_2^1)=(2,1)$ and $(c_1^j,c_2^j)=(2,0)$ for $j=2,3$, implying (iii).

\medskip
\noindent
\textbf{Case 2.}
\emph{Without loss of generality, $|M_{13}| = 2$ and $|M_{23}| = 1$.}

\medskip
\noindent
Let $v_2v_3$ be the edge in $M_{23}$.
Since $|M_{12}| = |M_{13}| = 2$ we can pick edges $w_1w_2 \in M_{12}$ and $x_1x_3 \in M_{13}$
so that $w_2 \neq v_2$ and $x_1 \neq w_1$.
But $M_{13}$ and $M_{23}$ are vertex-disjoint, so $x_3 \neq v_3$.
In this case, we let $\mathcal{P} := \lbrace w_1w_2, v_2v_3, x_3x_1 \rbrace$.
(i) is immediate, and $R_{\mathcal{V}}(\mathcal{P}) \cong C_3$ so (ii) holds.
Moreover, $(c_1^j,c_2^j)=(2,0)$ for all $V \in \mathcal{V}$, implying (iii).
This completes the proof of the case $|\mathcal{V}|=3$.
\end{proof}


\subsection{Finding an $(A,B)$-balanced path system in a bipartite robust expander}\label{sec:balance}

In Section~\ref{sec:someresults} we showed that, given a robust partition $\mathcal{V}$, `the balancing property' (T4) was sufficient to extend a $\mathcal{V}$-tour into a Hamilton cycle.
In this section we prove some lemmas which will be useful in finding a path system which satisfies (T4).
We begin by observing the following crucial fact.

\begin{proposition} \label{fact2}
Let $G$ be a $D$-regular graph with vertex partition $A$, $B$, $V$.
Then
\begin{itemize}
\item[(i)]
$
2(e(A) - e(B)) + e(A,V) - e(B,V) = (|A| - |B|)D.
$
\end{itemize}
In particular,
\begin{itemize}
\item[(ii)]
$
2e(A) + e(A,V) \geq (|A|-|B|)D;
$
\item[(iii)] if $V = \emptyset$ then 
$
2(e(A) - e(B)) = (|A| - |B|)D.
$
\end{itemize}
\end{proposition}

\begin{proof}
It suffices to prove (i) since (ii) and (iii) are then immediate.
We have that
$$
\sum\limits_{x \in A} d_{B}(x) = e(A, B) = \sum\limits_{y \in B} d_{A}(y).
$$
Moreover, by counting degrees,
\begin{align*}
2e(A) + e(A,V) &= \sum\limits_{x \in A}\left(D - d_{B}(x) \right) = D|A| - \sum\limits_{x \in A}d_{B}(x),
\end{align*}
and similarly for $B$.
So $2e(A) - 2e(B) + e(A,V) - e(B,V) = D(|A| - |B|)$, as desired.
\end{proof}

The following proposition follows immediately from Vizing's Theorem on edge-colourings.

\begin{proposition} \label{largematching}
Let $H$ be a graph with $\Delta(H) \leq \Delta$. Then $H$ contains a matching of size $\left\lceil e(H)/(\Delta+1) \right\rceil$.
\end{proposition}

Given a graph $G$, a collection $\mathcal{U}$ of disjoint subsets of $V(G)$ and a $\mathcal{U}$-anchored path system $\mathcal{P}$ in $G$, we say that a path system $\mathcal{P}'$ is a \emph{$\mathcal{U}$-extension of $\mathcal{P}$} if
\begin{itemize}
\item every edge which lies in a path of $\mathcal{P}'$ but not a path of $\mathcal{P}$ lies in $\bigcup_{U \in \mathcal{U}}G[U]$;
\item for every $P \in \mathcal{P}$ there is a unique $P' \in \mathcal{P}'$ such that $P \subseteq P'$;
\item for every $P' \in \mathcal{P}'$ there is at most one $P \in \mathcal{P}$ such that $P \subseteq P'$.
\end{itemize}
If $U \subseteq V(G)$ we will write \emph{$U$-extension} for $\lbrace U \rbrace$-extension.
The next lemma shows that a $\mathcal{U}$-extension $\mathcal{P}'$ of $\mathcal{P}$ `behaves similarly' to $\mathcal{P}$ in the reduced multigraph $R_{\mathcal{U}}$, and also that $R_{\mathcal{U}}$ is not affected by considering a slightly different partition.

\begin{lemma} \label{xext}
Let $\mathcal{U}$ be a collection of disjoint vertex-subsets of a graph $G$ and let $\mathcal{P}$ be a $\mathcal{U}$-anchored path system in $G$. 
\begin{itemize}
\item[(i)] Suppose that $\mathcal{P}'$ is a $\mathcal{U}$-extension of $\mathcal{P}$. Then $\mathcal{P}'$ is a $\mathcal{U}$-anchored path system.
\item[(ii)] Suppose that $\mathcal{P}'$ is an $\mathcal{X}$-extension of $\mathcal{P}$ for some $\mathcal{X} \subseteq \mathcal{U}$.
Then $\mathcal{P}'$ is a $\mathcal{U}$-extension of $\mathcal{P}$.
\item[(iii)] Suppose that $\mathcal{P}'$ is a $\mathcal{U}$-extension of $\mathcal{P}$.
Then $R_{\mathcal{U}}(\mathcal{P}')$ is an Euler tour if and only if $R_{\mathcal{U}}(\mathcal{P})$ is an Euler tour.
\item[(iv)] Suppose that $\mathcal{U} := \lbrace U_1, \ldots, U_t \rbrace$, $\mathcal{X} := \lbrace X_1, \ldots, X_t \rbrace$, $X_i \subseteq U_i$ for all $1 \leq i \leq t$, and $\mathcal{P}$ is $\mathcal{X}$-anchored.
Then $R_{\mathcal{X}}(\mathcal{P}) \cong R_{\mathcal{U}}(\mathcal{P})$.%
\COMMENT{needed for $k$-linkage}
\end{itemize}
\end{lemma}

\begin{proof}
Note that (i), (ii) and (iv) are immediate.
To prove (iii), let $\mathcal{R}$ be the subset of $\mathcal{P}'$ such that every $R \in \mathcal{R}$ contains some $P_R \in \mathcal{P}$. 
So $|\mathcal{R}| = |\mathcal{P}|$.
Observe that $P_R$ has endpoints in $U,U' \in \mathcal{U}$ if and only if $R$ has endpoints in $U,U'$.
So $R_{\mathcal{U}}(\mathcal{R}) \cong R_{\mathcal{U}}(\mathcal{P})$.
Let $\mathcal{Q} := \mathcal{P}' \setminus \mathcal{R}$.
Then every edge in a path in $\mathcal{Q}$ lies in $\bigcup_{U \in \mathcal{U}}G[U]$.
So $R_{\mathcal{U}}(\mathcal{Q})$ consists entirely of loops.
Therefore $R_{\mathcal{U}}(\mathcal{P}') = R_{\mathcal{U}}(\mathcal{R}) \cup R_{\mathcal{U}}(\mathcal{Q})$ is an Euler tour if and only if $R_{\mathcal{U}}(\mathcal{R})$ is, i.e. if and only if $R_{\mathcal{U}}(\mathcal{P})$ is.
This proves (iii).
\end{proof}

Suppose that $A,B \subseteq V(G)$ are disjoint.
The following lemma gives a sufficient condition which ensures that a path system $\mathcal{P}$ can be extended into an $(A,B)$-balanced path system which does not cover too much of $A \cup B$.
Whenever we wish to find a balanced path system we need then only find a collection of paths which satisfy this condition.

\begin{lemma} \label{pathcover}
Let $n \in \mathbb{N}$ and $0 < 1/n \ll \rho < 1$ and suppose that $G$ is a graph on $n$ vertices. Let $U \subseteq V(G)$ have bipartition $A,B$ where $||A|-|B|| \leq \rho n$ and $\delta(G[A,B]) > 9\rho n$.
Let $\mathcal{P}$ be a path system in $G$ such that $|V(\mathcal{P}) \cap U| \leq \rho n$,
\begin{equation} \label{bal}
2e_{\mathcal{P}}(A) - 2e_{\mathcal{P}}(B) + e_{\mathcal{P}}(A,\overline{U}) - e_{\mathcal{P}}(B,\overline{U}) = 2(|A|-|B|)
\end{equation}
and $\mathcal{P}$ has at least one endpoint in $U$.%
\COMMENT{so $\mathcal{P}$ is certainly non-empty}
Then $G$ contains a path system $\mathcal{P}'$ such that
\begin{itemize}
\item[($\alpha$)] $\mathcal{P}'$ is a $U$-extension of $\mathcal{P}$;
\item[($\beta$)] $\mathcal{P}'$ is $(A,B)$-balanced;
\item[($\gamma$)] $|V(\mathcal{P}') \cap U| \leq 9\rho n$.
\end{itemize}
\end{lemma}

\begin{proof}
Without loss of generality, suppose that $|A| \geq |B|$.
Let $A_0 \subseteq A$ and $B_0 \subseteq B$ be minimal such that $V(\mathcal{P}) \cap U \subseteq A_0 \cup B_0$ and
\begin{equation} \label{samediff'}
|A_0| - |B_0| = |A|-|B|.
\end{equation}
Note that
\begin{equation} \label{A0B0}
|A_0| + |B_0| = |A|-|B| + 2|B_0| \leq ||A|-|B|| + 2|V(\mathcal{P}) \cap U| \leq 3\rho n.
\end{equation}
For each $u \in A_0$, find a set $N_u$ of $2-d_{\mathcal{P}}(u)$ neighbours of $u$ in $B \setminus B_0$.
For each $v \in B_0$, find a set $N_v$ of $2-d_{\mathcal{P}}(v)$ neighbours of $v$ in $A \setminus A_0$.
Choose these sets to be disjoint and such that $(N_u \cup N_v) \cap V(\mathcal{P}) = \emptyset$.
This is possible since for each $u \in A$ and $v \in B$ we have $d_B(u), d_A(v) > 3(|A_0|+|B_0|)$.
Obtain $\mathcal{P}'$ from $\mathcal{P}$ by adding the edges $xx'$ to (the paths in) $\mathcal{P}$
for each $x \in A_0 \cup B_0$ and for each $x' \in N_x$.
It is clear that $\mathcal{P}'$ is a $U$-extension of $\mathcal{P}$, so ($\alpha$) holds.

Note that the set of internal vertices of $\mathcal{P}'$ which lie in $U$ is precisely $A_0 \cup B_0$.
Then $\sideset{}{_{\mathcal{P}'}}\Int(A) - \sideset{}{_{\mathcal{P}'}}\Int(B) = |A_0| - |B_0| = |A|-|B|$ by (\ref{samediff'}).
So to show ($\beta$), it is enough to check that $\sideset{}{_{\mathcal{P}'}}\End(A) = \sideset{}{_{\mathcal{P}'}}\End(B)$ and that this value is non-zero.
Since
\begin{align*}
\sum_{u \in A} d_{\mathcal{P}'}(u) &= 2e_{\mathcal{P}'}(A) + e_{\mathcal{P}'}(A,B) + e_{\mathcal{P}'}(A,\overline{U}) = 2e_{\mathcal{P}}(A) + e_{\mathcal{P}'}(A,B) + e_{\mathcal{P}}(A,\overline{U}),
\end{align*}
and similarly for $B$, 
we have that
\begin{align} \label{degdiff}
\sum\limits_{u \in A}d_{\mathcal{P}'}(u)-\sum\limits_{v \in B}d_{\mathcal{P}'}(v) &= 2e_{\mathcal{P}}(A) - 2e_{\mathcal{P}}(B) + e_{\mathcal{P}}(A,\overline{U}) - e_{\mathcal{P}}(B,\overline{U})\\
\nonumber &\hspace{-0.18cm}\stackrel{(\ref{bal})}{=} 2(|A|-|B|).
\end{align}
By construction, $\sum_{u \in A}d_{\mathcal{P}'}(u) = \sideset{}{_{\mathcal{P}'}}\End(A)  + 2|A_0|$,
and similarly for $B$.
So, by (\ref{degdiff}),
\begin{equation} \label{sameendpts}
\sideset{}{_{\mathcal{P}'}}\End(A) - \sideset{}{_{\mathcal{P}'}}\End(B) = 2(|A|-|B|) - 2(|A_0|-|B_0|) \stackrel{(\ref{samediff'})}{=} 0.
\end{equation}
Recall that $\mathcal{P}$ has at least one endpoint $x$ lying in $U$.
Then $|N_x| = 1$ and the vertex in $N_x$ is an endpoint of a path in $\mathcal{P}'$.
So $\sideset{}{_{\mathcal{P}'}}\End(A) = \sideset{}{_{\mathcal{P}'}}\End(B)$ is non-zero, proving ($\beta$).

Finally, note that every vertex in $V(\mathcal{P}') \cap U$ which does not lie in $A_0 \cup B_0$ is a neighbour of some $x \in A_0 \cup B_0$ in $\mathcal{P}'$.
So
(\ref{A0B0}) implies that
$$
|V(\mathcal{P}') \cap U| \leq |A_0 \cup B_0| + |N_{\mathcal{P}'}(A_0 \cup B_0)| \leq 3(|A_0|+|B_0|)  \leq 9\rho n,
$$
proving ($\gamma$).
\end{proof}

The next lemma is essentially an iteration of Lemma~\ref{pathcover}.
We will use it to successively extend a path system into one that is $(A,B)$-balanced for all appropriate $A,B$.

\begin{lemma} \label{balextend}
Let $n,k,\ell \in \mathbb{N}$ and $0 < 1/n \ll \rho \ll \nu \ll \tau \ll \eta < 1$.%
\COMMENT{$\nu$ and $\tau$ are superfluous and only needed to define the WRSP} Let $G$ be a graph on $n$ vertices and suppose that $\mathcal{U} := \lbrace U_1, \ldots, U_k, W_1, \ldots, W_\ell \rbrace$ is a weak robust subpartition in $G$ with parameters $\rho,\nu,\tau,\eta,k,\ell$.
For each $1 \leq j \leq \ell$, let $A_j,B_j$ be the bipartition of $W_j$ specified by \emph{(D3$'$)}.
Let $\mathcal{P}$ be a $\mathcal{U}$-anchored path system such that for each $1 \leq j \leq \ell$, 
\begin{equation} \label{bal2}
2e_{\mathcal{P}}(A_j) - 2e_{\mathcal{P}}(B_j) + e_{\mathcal{P}}(A_j,\overline{W_j}) - e_{\mathcal{P}}(B_j,\overline{W_j}) = 2(|A_j|-|B_j|).
\end{equation}
Suppose further that $|V(\mathcal{P}) \cap U| \leq \rho n$ for all $U \in \mathcal{U}$, and that $R_{\mathcal{U}}(\mathcal{P})$ is a non-empty Euler tour.
Then $G$ contains a $\mathcal{U}$-extension $\mathcal{P}'$ of $\mathcal{P}$ that is a $\mathcal{U}$-tour with parameter $9\rho$.
\end{lemma}

\begin{proof}
Let $\mathcal{P}_0 := \mathcal{P}$.
Suppose that for some $0 \leq i < \ell$, we have already defined a path system $\mathcal{P}_i$ such that
\begin{itemize}
\item[($\alpha_i$)] $\mathcal{P}_i$ is a $\lbrace W_1, \ldots, W_i \rbrace$-extension of $\mathcal{P}$;
\item[($\beta_i$)] for all $1 \leq j \leq i$, $\mathcal{P}_i$ is $(A_j,B_j)$-balanced;
\item[($\gamma_i$)] for all $1 \leq j \leq i$, $|V(\mathcal{P}_i) \cap W_j| \leq 9\rho n$.
\end{itemize}
Now we obtain $\mathcal{P}_{i+1}$ from $\mathcal{P}_i$ as follows.
Note that (D3$'$) implies that (B1) and (C2) hold and hence that $||A_{i+1}|-|B_{i+1}|| \leq \rho n$.
Moreover, by (D5$'$) we have that $\delta(G[A_j,B_j]) \geq \eta n/2 > 9\rho n$.
Also ($\alpha_i$) implies that $|V(\mathcal{P}_i) \cap W_{i+1}| = |V(\mathcal{P}) \cap W_{i+1}| \leq \rho n$ and that (\ref{bal2}) still holds with $i+1$ and $\mathcal{P}_i$ playing the roles of $j$ and $\mathcal{P}$.
Finally, $R_{\mathcal{U}}(\mathcal{P})$ is a non-empty Euler tour, so $\mathcal{P}$ contains at least one endpoint in $W_{i+1}$.
Thus $\mathcal{P}_i$ contains at least one endpoint in $W_{i+1}$ by ($\alpha_i$).
Therefore we can apply Lemma~\ref{pathcover} with $W_{i+1},A_{i+1},B_{i+1},\mathcal{P}_i,\rho$ playing the roles of $U,A,B,\mathcal{P},\rho$. 
We thus obtain a path system $\mathcal{P}_{i+1}$ satisfying Lemma~\ref{pathcover}($\alpha$)--($\gamma$).
Now ($\alpha$) and ($\alpha_{i}$) imply that ($\alpha_{i+1}$) holds.
We obtain $(\beta_{i+1})$ and $(\gamma_{i+1})$ in a similar way.

Therefore we can obtain $\mathcal{P}' := \mathcal{P}_{\ell}$ that satisfies ($\alpha_{\ell}$)--($\gamma_\ell$).
Now ($\alpha_{\ell}$) and Lemma~\ref{xext}(ii) imply that $\mathcal{P}'$ is a $\mathcal{U}$-extension of $\mathcal{P}$.
It remains to show that (T1)--(T4) hold for $\mathcal{P}'$ with $9\rho$ playing the role of $\gamma$.
Indeed, (T1) follows from Lemma~\ref{xext}(i) and the fact that $\mathcal{P}'$ is a $\mathcal{U}$-extension of $\mathcal{P}$.
Since $\mathcal{R}_{\mathcal{U}}(\mathcal{P})$ is an Euler tour, Lemma~\ref{xext}(iii) implies that $\mathcal{R}_{\mathcal{U}}(\mathcal{P}')$ is an Euler tour, and hence (T2) holds.
We have $|V(\mathcal{P}') \cap W_j| \leq 9\rho n$ for all $1 \leq j \leq \ell$ by ($\gamma_\ell$).
Moreover, by ($\alpha_\ell$) we have that $|V(\mathcal{P}') \cap U_j| = |V(\mathcal{P}) \cap U_j| \leq \rho n$ for all $1 \leq j \leq k$.
So (T3) holds.
Finally, (T4) is immediate from ($\beta_\ell$).
\end{proof}


\subsection{Finding a $\mathcal{V}$-tour in a regular bipartite robust expander} \label{sec:pathsys}

We now consider the case when $G$ has a robust partition with $(k,\ell)=(0,1)$, i.e.~$G$ is a regular bipartite robust expander.
By Corollary~\ref{HEScor}, in order to find a Hamilton cycle in $G$ it suffices to find a $\lbrace V(G) \rbrace$-tour with an appropriate parameter.
This is guaranteed by the following lemma.

\begin{lemma} \label{(0,1)}
Let $D, n \in \mathbb{N}$ and let $0 < 1/n \ll \rho \ll \nu \ll \tau \ll \alpha < 1$.
Let $G$ be a $D$-regular graph on $n$ vertices where $D \geq \alpha n$.
Suppose that $G$ has a robust partition $\mathcal{V}$ with parameters $\rho,\nu,\tau,0,1$.
Then $G$ contains a $\mathcal{V}$-tour with parameter~$18\rho$.
\end{lemma}

\begin{proof}
Note (D3) implies that there exists a bipartition $A,B$ of $V(G)$ such that $G$ is a bipartite $(\rho,\nu,\tau)$-expander with bipartition $A,B$.
By (D5) we have that $\delta(G[A,B]) \geq D/2$.
Therefore
\begin{equation} \label{degA}
\Delta(G[A]),\Delta(G[B]) \leq D/2.
\end{equation}
Moreover, (B1) (which follows from (D3)) implies that $G$ is $\rho$-close to bipartite with bipartition $A,B$. So (C2) holds, i.e.
\begin{equation} \label{A-B}
||A|-|B|| \leq \rho n.
\end{equation}
Suppose first that $|A|=|B|$.%
\COMMENT{we consider this separately because we want a non-empty Euler tour.}
Then let $\mathcal{P}$ consist of exactly one $AB$-edge. Note that $R_{\mathcal{V}}(\mathcal{P})$ is a loop and that $\mathcal{P}$ is $(A,B)$-balanced.
All of (T1)--(T4) hold.

Let us now assume that $|A| > |B|$ (the case where $|B| > |A|$ is similar).
Proposition~\ref{fact2}(iii) implies that
\begin{equation} \label{eqedgesA}
e(A) \geq e(A)-e(B) = (|A|-|B|)D/2.
\end{equation}
Proposition~\ref{largematching} implies that $G[A]$ contains a matching of size
\begin{eqnarray*}
\left\lceil \frac{e(A)}{\Delta(A) + 1} \right\rceil &\stackrel{(\ref{degA}),(\ref{eqedgesA})}{\geq}& \left\lceil \frac{(|A|-|B|)D/2}{D/2 + 1} \right\rceil = |A|-|B| - \left\lfloor \frac{|A|-|B|}{D/2+1}\right\rfloor\\
&\stackrel{(\ref{A-B})}{\geq}& |A| - |B| - \left\lfloor 2\rho/\alpha \right\rfloor = |A|-|B|.
\end{eqnarray*}
So we can choose a matching $M$ of size $|A|-|B|$ in $G[A]$.

Now Proposition~\ref{WRSD-RD}(i) implies that $\mathcal{V}$ is a weak robust subpartition in $G$ with parameters $\rho,\nu,\tau,\alpha^2/2,0,1$. Certainly $M$ is $\mathcal{V}$-anchored and
$$
2e_M(A) - 2e_M(B) + e_M(A,\overline{V(G)}) - e_M(B,\overline{V(G)}) = 2e_M(A) = 2(|A|-|B|).
$$
We also have that $|V(M)| = 2(|A|-|B|) \leq 2\rho n$.
Moreover, $M$ is non-empty since $|A|-|B| > 0$.
Thus $R_{\mathcal{V}}(M)$ is a non-empty collection of loops and hence a non-empty Euler tour.
Therefore we can apply Lemma~\ref{balextend} with $\mathcal{V},0,1,V(G)$,$A,B,M,2\rho,\alpha^2/2$ playing the roles of $\mathcal{U},k,\ell,W_j,A_j,B_j,\mathcal{P},\rho,\eta$ to obtain a path system $\mathcal{P}$ which is a $\mathcal{V}$-tour with parameter $18\rho$.
\end{proof}

\subsection{Finding a $\mathcal{V}$-tour when there is exactly one component of each type} \label{sec:pathsys2}

We would like to find a Hamilton cycle when $G$ is the union of a robust expander component $V$ and a bipartite robust expander component $W$.
By Corollary~\ref{HEScor}, it is sufficient to find a $\mathcal{V}$-tour for this robust partition $\mathcal{V}$.
This is guaranteed by the following lemma.

\begin{lemma} \label{11path}
Let $n,D \in \mathbb{N}$, $0 < 1/n \ll \rho \ll \nu \ll \tau \ll \alpha < 1$ and let $D \geq \alpha n$.
Suppose that $G$ is a $3$-connected $D$-regular graph on $n$ vertices and that $\mathcal{V}$ is a robust partition of $G$ with parameters $\rho,\nu,\tau,1,1$.
Then $G$ contains a $\mathcal{V}$-tour with parameter $36\rho$.
\end{lemma}

Let $V,W$ be as above and let $A,B$ be a bipartition of $W$ such that $W$ is a bipartite robust expander with respect to $A,B$.
Suppose that $|A| \geq |B|$.
To prove Lemma~\ref{11path}, our aim is to find a path system $\mathcal{P}$ to which we can apply Lemma~\ref{balextend} and hence obtain a $\mathcal{V}$-tour.
Roughly speaking, $\mathcal{P}$ will consist of the union of two matchings, $M_A$ in $G[A]$ and $M_{A,V}$ in $G[A,V]$ which together have the right size to `balance' $W$.

\medskip
\noindent
\emph{Proof of Lemma~\ref{11path}.}
Let $\mathcal{V} := \lbrace V,W \rbrace$, where $V$ is a $(\rho,\nu,\tau)$-robust expander component and $W$ has bipartition $A,B$ so that $W$ is a bipartite $(\rho,\nu,\tau)$-robust expander component with respect to $A,B$. 
So (B1) and (C2) imply that
\begin{equation} \label{ABdiff}
||A|-|B|| \leq \rho n.
\end{equation}
Moreover, (D4) implies that $\delta(G[V]),\delta(G[W]) \geq D/2$ and therefore
\begin{equation} \label{deltaVW}
D/2 \geq \Delta(G[W,V]) \geq \Delta(G[A,V]).
\end{equation}
By (D5) we have
\begin{equation} \label{deltaA}
\Delta(G[A]) \leq D/2.
\end{equation}

\medskip
\noindent
\textbf{Claim 1.} 
\emph{
It suffices to find a path system $\mathcal{P}$ in $G$ such that the following hold:
\begin{itemize}
\item[(i)] $2e_{\mathcal{P}}(A) - 2e_{\mathcal{P}}(B) + e_{\mathcal{P}}(A,V) - e_{\mathcal{P}}(B,V) = 2(|A|-|B|)$;
\item[(ii)] $e(\mathcal{P}) \leq 2\rho n$;
\item[(iii)] $\mathcal{P}$ has at least one $VW$-path.
\end{itemize}}
 
\medskip
\noindent
To see this, note that Proposition~\ref{WRSD-RD}(i) implies that $\mathcal{V}$ is a weak robust subpartition in $G$ with parameters $\rho,\nu,\tau,\alpha^2/2,1,1$.
Clearly, $\mathcal{P}$ is a $\mathcal{V}$-anchored path system.
Observe that (D5) implies that $\delta(G[A,B]) \geq D/4$.
Let $p$ be the number of $VW$-paths in $\mathcal{P}$.
Then $R_{\mathcal{V}}(\mathcal{P})$ is an Euler tour if and only if $p$ is positive and even.
By (iii) we have $p > 0$.
Now (i) implies that 
$$
e_{\mathcal{P}}(W,V) = e_{\mathcal{P}}(A,V)+e_{\mathcal{P}}(B,V) = 2(|A|-|B|) - 2e_{\mathcal{P}}(A) + 2e_{\mathcal{P}}(B) + 2e_{\mathcal{P}}(B,V)
$$
is even.
Note that any $P \in \mathcal{P}$ contains an odd number of $VW$-edges if $P$ is a $VW$-path, and an even number otherwise.
Therefore $p$ is even and so $R_{\mathcal{V}}(\mathcal{P})$ is a non-empty Euler tour.
Finally, for each $X \in \mathcal{V}$ we have $|V(\mathcal{P}) \cap X| \leq 2e(\mathcal{P}) \leq 4\rho n$ by (ii).
Therefore we can apply Lemma~\ref{balextend} with $\mathcal{V},1,1,W,A,B,\mathcal{P},4\rho,\alpha^2/2$ playing the roles of $\mathcal{U},k,\ell,W_j,A_j,B_j,\mathcal{P},\rho,\eta$ to find a $\mathcal{V}$-extension $\mathcal{P}'$ of $\mathcal{P}$ that is a $\mathcal{V}$-tour with parameter $36\rho$, proving the claim.

\medskip
\noindent
So it remains to find a path system $\mathcal{P}$ as in Claim~1.
Suppose first that $|A|=|B|$. 
Since $G$ is 3-connected,
Proposition~\ref{menger} implies that $G[V,W]$ contains a matching of size three. 
We only consider the case when $G[A,V]$ contains a matching $M_{A,V}$ of size two.
(The case when this holds for $G[B,V]$ is similar.)
Now Proposition~\ref{fact2}(i) implies that
$$
2e(B) + e(B,V) = 2e(A) + e(A,V) \geq 2.
$$
If $e(B) \geq 1$,
let $\mathcal{P} := M_{A,V} \cup \lbrace e \rbrace$, where $e$ is an edge in $G[B]$.
Otherwise, $e(B) = 0$ and hence $e(B,V) \geq 2$.
In this case we let $\mathcal{P}$ consist of two vertex-disjoint edges $e \in G[A,V]$ and $e' \in G[B,V]$.
In both cases, (i)--(iii) clearly hold for $\mathcal{P}$ and we are done.

So let us assume that $|A| > |B|$.
(The case when $|B| > |A|$ is similar.)
Proposition~\ref{fact2}(ii) implies that
\begin{equation} \label{eqedgesB'}
2e(A) + e(A,V) \geq (|A|-|B|)D.
\end{equation}

Suppose first that $e(A) < D/5$.
Then (\ref{eqedgesB'}) implies that $e(A,V) \geq (|A|-|B|)D - 2D/5$.
Now Proposition~\ref{largematching} implies that $G[A,V]$ contains a matching of size at least
\begin{eqnarray} \label{D/5}
\left\lceil \frac{e(A,V)}{\Delta(G[A,V])+1} \right\rceil &\stackrel{(\ref{deltaVW})}{\geq}& \left\lceil \frac{(|A|-|B|)D - 2D/5}{D/2+1} \right\rceil\\ 
\nonumber &=& 2(|A|-|B|) - \left\lfloor \frac{2(|A|-|B|)+2D/5}{D/2+1} \right\rfloor\\
\nonumber &\stackrel{(\ref{ABdiff})}{\geq}& 2(|A|-|B|) - \left\lfloor \frac{D/2}{D/2+1} \right\rfloor 
= 2(|A|-|B|).
\end{eqnarray}
Let $\mathcal{P}$ be a matching of size $2(|A|-|B|)$ in $G[A,V]$.
Then $\mathcal{P}$ satisfies (i)--(iii) (indeed, (ii) follows from (\ref{ABdiff})).

Therefore we can assume that $e(A) \geq D/5$.
Let
\begin{equation} \label{l}
\ell := \min \left\lbrace \left\lceil \frac{e(A)}{D/2+1} \right\rceil, |A|-|B| \right\rbrace.
\end{equation}
Note that $\ell \geq 1$.
Clearly $G[A]$ contains a matching of size $\ell$ by Proposition~\ref{largematching} and (\ref{deltaA}).
We now consider two cases, depending on the value of $\ell$.

\medskip
\noindent
\textbf{Case 1.}
\emph{$\ell = |A|-|B|$.}

\medskip
\noindent
Let $M$ be a matching of size $\ell$ in $G[A]$.
Since $G$ is 3-connected, Proposition~\ref{menger} implies that $G[V,W]$ contains a matching of size three.
Suppose first that $G[A,V]$ contains a matching $M_{A,V}$ of size two.
Write $V(M_{A,V}) \cap A := \lbrace u,u' \rbrace$.
If $uu'$ is an edge in $M$, delete it to obtain $M'$.
Otherwise delete an arbitrary edge from $M$ to obtain $M'$.
Let $\mathcal{P} := M' \cup M_{A,V}$.
Then $\mathcal{P}$ is a path system satisfying (i).
Also (ii) follows from (\ref{ABdiff}).
Moreover, $u$ lies in a $VW$-path in $\mathcal{P}$, so (iii) holds.

So suppose that $G[A,V]$ does not contain a matching of size two.
Then $G[B,V]$ contains a matching $M_{B,V}$ of size two.
Moreover, there is at most one vertex in $A \cup V$ such that every edge in $G[A,V]$ is incident to this vertex. Therefore (\ref{deltaVW}) implies that $e(A,V) \leq \Delta(G[A,V]) \leq D/2$.
So
$$
e(A) - |M| \stackrel{(\ref{eqedgesB'})}{\geq} (|A|-|B|)D/2 - D/4 - |M| \geq D/4-1 > 0,
$$ 
where the penultimate inequality follows from the fact that $|M| = |A|-|B|>0$.
So we can find an edge $e$ in $G[A]$ that is not contained in $M$.
Let $\mathcal{P} := M_{B,V} \cup M \cup \lbrace e \rbrace$.
Then $\mathcal{P}$ is a path system satisfying (i)--(iii).
This completes the proof of Case 1.

\medskip
\noindent
\textbf{Case 2.}
\emph{$\ell < |A|-|B|$ and so $\ell = \lceil e(A)/(D/2+1) \rceil$.}

\medskip
\noindent
\textbf{Claim 2.}
\emph{Suppose that $G[A]$ contains no matching of size $\ell +1$.
Then $G[A]$ contains a matching $M^-$ of size $\ell-1$ and a path $P := xyz$ which is vertex-disjoint from $M^-$.} 

\medskip
\noindent
To see this,
suppose first that $\Delta(G[A]) \leq D/8-1$.
Then Proposition~\ref{largematching} implies that $G[A]$ contains a matching of size
\begin{equation} \label{D/8}
\left\lceil \frac{e(A)}{D/8} \right\rceil = \left\lceil \frac{e(A)}{D/3} + \frac{5e(A)}{D} \right\rceil \geq \left\lceil \frac{e(A)}{D/3} + 1 \right\rceil \geq \ell +1,
\end{equation}
a contradiction. 
So $\Delta(G[A]) > D/8-1 > 2\ell$ by (\ref{ABdiff}) and (\ref{l}).
Recall that $G[A]$ contains a matching $M$ of size $\ell$.
Since $M$ must be maximal, there is some $y \in V(M)$ such that $d_A(y) > 2\ell$.
Let $x \in A$ be a neighbour of $y$ such that $x \notin V(M)$.
Let $z$ be the neighbour of $y$ in $M$. 
Let $M^- := M \setminus \lbrace yz \rbrace$ and $P := xyz$.
This completes the proof of Claim~2.

\medskip
\noindent
Proposition~\ref{largematching} implies that $G[A,V]$ contains a matching of size
\begin{eqnarray*}
\left\lceil \frac{e(A,V)}{\Delta(G[A,V])+1} \right\rceil &\stackrel{(\ref{deltaVW})}{\geq}&
\left\lceil \frac{e(A,V)}{D/2+1} \right\rceil + 2\left\lceil \frac{e(A)}{D/2+1} \right\rceil - 2\ell \\
&\geq& \left\lceil \frac{2e(A) + e(A,V)}{D/2+1} \right\rceil - 2\ell\\
&\stackrel{(\ref{eqedgesB'})}{\geq}& \left\lceil \frac{(|A|-|B|)D}{D/2+1} \right\rceil - 2\ell \geq 2(|A|-|B| - \ell),
\end{eqnarray*}
where the final inequality follows in a similar way to (\ref{D/5}). So we can choose a matching $M_{A,V}$ in $G[A,V]$ of size $2(|A|-|B|-\ell) > 0$.

Let $E$ be any collection of $\ell$ edges in $G[A]$ and let $H := E \cup M_{A,V}$. Then
\begin{equation} \label{Hi}
2e_{H}(A) - 2e_{H}(B) + e_{H}(A,V) - e_{H}(B,V)
= 2|E| + |M_{A,V}| = 2(|A|-|B|).
\end{equation}
Moreover,
\begin{equation} \label{Hii}
e(H) = |M_{A,V}| + |E| = 2(|A|-|B|)-\ell \stackrel{(\ref{ABdiff})}{\leq} 2\rho n.
\end{equation}
Suppose that $G[A]$ contains a matching $M$ of size $\ell +1$.
Then $\mathcal{P}^+ := M \cup M_{A,V}$ is a path system.
If $\mathcal{P}^+$ contains a $VW$-path then obtain $\mathcal{P}$ from $\mathcal{P}^+$ by deleting an arbitrary edge of $M$.
Otherwise there is an edge $e$ in $M$ which is incident to some edge in $M_{A,V}$.
Let $\mathcal{P} := \mathcal{P}^+ \setminus \lbrace e \rbrace$.
Then at least one endpoint of $e$ is an endpoint of a $VW$-path in $\mathcal{P}$.
In both cases, (iii) holds.
Also (i) and (ii) hold by (\ref{Hi}) and (\ref{Hii}).

Therefore we may assume that $G[A]$ contains no matching of size $\ell+1$.
Let $M^-,P=xyz$ be as guaranteed by Claim~2.
Then $M_1 := M^- \cup \lbrace xy \rbrace$ and $M_2 := M^- \cup \lbrace yz \rbrace$ are both matchings of size $\ell$ in $G[A]$.
For $i=1,2$, let $\mathcal{P}_i := M_i \cup M_{A,V}$.
These are both path systems.
Now (\ref{Hi}) and (\ref{Hii}) imply that both of $\mathcal{P}_1$ and $\mathcal{P}_2$ satisfy (i) and (ii).
If, for some $i=1,2$, $\mathcal{P}_i$ also satisfies (iii) then we are done by setting $\mathcal{P} := \mathcal{P}_i$, so suppose not.
Then for each $i=1,2$ there exists $M_i' \subseteq M_i$ such that $V(M_i') = V(M_{A,V}) \cap A$.
In particular, this implies that $M_1',M_2' \subseteq M^-$.
Pick any edge $e \in M_1'$ and let $\mathcal{P} := P \cup (M^- \setminus \lbrace e \rbrace) \cup M_{A,V}$.
Then both endpoints of $e$ are endpoints of a $VW$-path in $\mathcal{P}$, so (\ref{Hi}) and (\ref{Hii}) imply that $\mathcal{P}$ satisfies (i)--(iii).
\hfill$\square$

\subsection{The proof of Theorem~\ref{main}}

As already indicated at the beginning of the section, Theorem~\ref{main} now follows easily.
Indeed, recall that we have a robust partition $\mathcal{V}$
with only five possible values of $(k,\ell)$.
But Lemmas~\ref{easy},~\ref{(0,1)} and~\ref{11path} guarantee a $\mathcal{V}$-tour in each of these cases.
Now Corollary~\ref{HEScor} implies that $G$ contains a Hamilton cycle.

Actually, we even prove the following stronger stability result of which Theorem~\ref{main} is an immediate consequence: 
if the degree of $G$ is close to $n/4$ and $G$ is not Hamiltonian, then $G$ is either close to the union of four cliques, or two complete bipartite graphs, or the first extremal example discussed in Subsection~\ref{applications}.

\begin{theorem} \label{stability}
For every $\eps,\tau > 0$ with $2\tau^{1/3} \leq \eps$ and every non-decreasing function $g : (0,1) \rightarrow (0,1)$, there exists $n_0 \in \mathbb{N}$ such that the following holds.
For all $3$-connected $D$-regular graphs $G$ on $n \geq n_0$ vertices where $D \geq (1/5 + \eps) n$, at least one of the following holds:
\begin{itemize}
\item[(i)] $G$ has a Hamilton cycle;
\item[(ii)] $D < (1/4 + \eps)n$ and there exist $\rho,\nu$ with $1/n_0 \leq \rho \leq \nu \leq \tau$; $1/n_0 \leq g(\rho)$; $\rho \leq g(\nu)$, and $(k,\ell) \in \lbrace (4,0), (2,1), (0,2) \rbrace$ such that $G$ has a robust partition $\mathcal{V}$ with parameters $\rho,\nu,\tau,k,\ell$.
\end{itemize}
\end{theorem}

\begin{proof}
Let $\alpha := 1/5 + \eps$.
Choose a non-decreasing function $f: (0,1) \rightarrow (0,1)$ with $f(x) \leq \min \lbrace x, g(x) \rbrace$ for all $x \in (0,1)$ 
such that the requirements of
Proposition~\ref{fewstructs} (applied with $r := 5$), Corollary~\ref{HEScor} and Lemmas~\ref{easy},~\ref{(0,1)} and~\ref{11path}
(each applied with $\tau'$ playing the role of $\tau$) 
are satisfied whenever $n,\rho,\gamma,\nu,\tau'$ satisfy
\begin{align} \label{hierarchy}
1/n &\leq f(\rho), f(\gamma);\ \ \rho \leq f(\nu),\eps^3/8;\ \ \gamma \leq f(\nu);\ \ \nu \leq f(\tau');\ \ \tau' \leq f(\eps),f(1/5),\tau
\end{align}
(and so $\tau' \leq f(\alpha)$).
Choose $\tau',\tau''$ such that $0 < \tau' \leq f(\eps),f(1/5),\tau$ and let $\tau'' := f(\tau')$.
Apply Theorem~\ref{structure} with $f/36,\alpha,\tau''$ playing the roles of $f,\alpha,\tau$ to obtain an integer $n_0$.%
\COMMENT{$(f/36)(x) = f(x)/36$. Use this function so that $\gamma := 36\rho \leq f(\nu)$.}
Let $G$ be a 3-connected $D$-regular graph on $n \geq n_0$ vertices where $D \geq \alpha n$.
Theorem~\ref{structure} now guarantees
$\rho,\nu,k,\ell$ with $1/n_0 \leq \rho \leq \nu \leq \tau''$, $1/n_0 \leq f(\rho)$ and $36\rho \leq f(\nu)$ such that $G$ has a robust partition $\mathcal{V}$ with parameters $\rho,\nu,\tau'',k,\ell$ (and thus also a robust partition with parameters $\rho,\nu,\tau',k,\ell$).

Let $\gamma:=36\rho$. Note that $n,\rho,\gamma,\nu,\tau'$ satisfy (\ref{hierarchy}).
So we can apply Proposition~\ref{fewstructs}(ii) with $\tau',5$ playing the roles of $\tau,r$ to see that $(k,\ell)$ is equal to (a) $(k,0)$ for $1 \leq k \leq 3$; (b) $(0,1)$; (c) $(1,1)$; or (d) $(4,0), (2,1), (0,2)$.
Apply Lemmas~\ref{easy},~\ref{(0,1)} and~\ref{11path} (with $\tau'$ playing the role of $\tau$) in the cases (a), (b), (c) respectively to obtain a $\mathcal{V}$-tour of $G$
with parameter $36\rho=\gamma$.%
\COMMENT{(Note that a $\mathcal{V}$-tour with parameter $\gamma$ is a $\mathcal{V}$-tour with parameter $\gamma'$ whenever $\gamma \leq \gamma'$.)}
Then Corollary~\ref{HEScor} (with $\tau'$ playing the role of $\tau$) implies that $G$ contains a Hamilton cycle so we are in case (i).
If instead (d) holds, Proposition~\ref{fewstructs}(i) implies that $D < (1/4 + \eps)n$.
Since $f \leq g$ and $\mathcal{V}$ is a robust partition with parameters $\rho,\nu,\tau,k,\ell$ (as $\tau' \leq \tau$) we are in case (ii).
\end{proof}

\medskip
\noindent
\emph{Proof of Theorem~\ref{main}.}
Let $\eps >0$.
Choose a positive constant $\tau$ such that $2\tau^{1/3} \leq \eps$.
Apply Theorem~\ref{stability} (with $g(x)=x$, say) to obtain an integer $n_0$.
Let $G$ be a $3$-connected $D$-regular graph on $n \geq n_0$ vertices with $D \geq (1/4 + \eps)n$.
Then Theorem~\ref{stability} implies that $G$ has a Hamilton cycle.
\hfill$\square$
\medskip


\section{The proofs of Theorems~\ref{tconnected} and~\ref{biptconnected}} \label{sec:prooflink}

We first show that Theorem~\ref{tconnected} is asymptotically best possible.%
\COMMENT{The next two propositions have changed: firstly, previously the $D$ in the propositions was smaller than in the theorems, 
so formally this didn't show that the bounds were best possible. Secondly, both are now cleaner.}

\begin{proposition}\label{bestposs}
Let $t,r \in \mathbb{N}$ be such that $r\ge 2$.%
    \COMMENT{DK: added $r\ge 2$}
Then there are infinitely many $n \in \mathbb{N}$ for which there exists a $t$-connected $D$-regular graph $G$ on $n$ vertices with 
$D:=(n-t)/(r-1) -1$ and circumference $c(G) \leq tn/(r-1)+t$.
\end{proposition}

One can easily modify the construction to obtain a $t$-connected $D$-regular graph $G$ with the same bound on $c(G)$ for smaller values of $D$ (e.g.~$D=n/r$).

\begin{proof}
We may suppose that $t \leq r-1$.%
\COMMENT{or $tn/(r-1)+t \geq n$.} 
Pick any $k \in \mathbb{N}$ with $k\ge 2t$.%
   \COMMENT{DK: $k\ge 2t$ ensures that $D\ge 2k\ge 4t\ge t+2$ and so $G[U_i]$ will certainly be $t$-connected (with room to spare)}
Let $$
n:=(r-1)(2k(r-1)+1)+t \ \ \mbox{and} \ \ D:=\frac{n-t}{r-1} -1=2k(r-1).
$$
Construct a graph $G$ on $n$ vertices as follows. 
Let $X,U_1, \ldots, U_{r-1}$ be a partition of $V(G)$, where $|X|=t$ and the $|U_i|=D+1$.
Add all edges within the $U_i$. So $G[U_i]$ is $D$-regular.
Let $M_i$ be a matching in $G[U_i]$ with $|V(M_i)| = tD/(r-1)$.
Note that $M_i$ exists since $tD/(r-1)$ is even, and at most $D$ since $t \leq r-1$.
Add exactly one edge from each $y \in V(M_i)$ to $X$ so that each $x \in X$ receives exactly $D/(r-1)$ edges from $V(M_i)$.
Remove $M_i$ from $G$.

Therefore $G$ is $t$-connected (with vertex cut-set $X$) and $D$-regular.
But any cycle in $G$ traverses at most $t$ of the $U_i$, so
$$
c(G) \leq  t|U_i| + |X| \leq tn/(r-1) +|X| = tn/(r-1)+t,
$$
as required.
\end{proof}

The first part of the following proposition shows that the bound on the circumference in Theorem~\ref{biptconnected} is close to best possible.
The second part of the proposition is a bipartite analogue of the extremal example in Figure~\ref{fig:exactex}(i).
The proofs may be found in~\cite{thesis}.

\begin{proposition}~\label{bipbestposs}
\begin{itemize}
\item[(i)] Let $t,r \in \mathbb{N}$ be such that $r \geq 4$ is even and $t \geq 2$.%
\COMMENT{$t =2,3$ doesn't tell us anything because the degree is too small, but the picture shows $t=3$.}
Then there are infinitely many $n \in \mathbb{N}$ for which there exists a $t$-connected $D$-regular bipartite graph $G$ on $n$
vertices with $D:=(n-2)/(r-2)$ and circumference $c(G) \leq 2tn/(r-2) + t$;
\item[(ii)] For every  $t\in \mathbb{N}$ with $t \geq 2$, there are infinitely many $D\in \mathbb{N}$ such that
there exists a bipartite graph on $8D+2$ vertices which is $D$-regular and $t$-connected  but does not contain a Hamilton cycle.
\end{itemize}
\end{proposition}
One can easily modify the construction to obtain a $t$-connected $D$-regular graph $G$ with the same bound on $c(G)$ for smaller values of $D$.

The proof of Theorem~\ref{tconnected} uses robust partitions as the main tool (Theorem~\ref{structure}).
We show that, in a $t$-connected graph $G$ with a robust partition, we can find a cycle that contains every vertex in the $t$ largest robust components of $G$ (or at least almost all the vertices in the case of bipartite robust components).
When $G$ has degree slightly larger than $n/r$, its robust partition contains at most $r-1$ components.
So the $t$ largest components together contain at least $tn/(r-1)$ vertices, as required.

We let $C_1$ denote a loop and $C_2$ a double edge.
The following result shows that, given any $t$-connected graph $G$ and any collection $\mathcal{U}$ of $t$ disjoint subsets of $V(G)$, we can find a path system $\mathcal{P}$ such that $R_{\mathcal{U}}(\mathcal{P}) \cong C_t$.

\begin{proposition} \label{compconn}
Let $t \in \mathbb{N}$, let $G$ be a $t$-connected graph and let $\mathcal{U} := \lbrace U_1, \ldots, U_t \rbrace$ be a collection of disjoint vertex-subsets of $G$ with $|U_i| \geq 2t$ for each $1 \leq i \leq t$.
Then there exists a $\mathcal{U}$-anchored path system $\mathcal{P}$ in $G$ such that $R_{\mathcal{U}}(\mathcal{P}) \cong C_t$.
\end{proposition}

\begin{proof}
For each $i$, let $\mathcal{U}_i := \lbrace U_1, \ldots, U_i \rbrace$.
Let $P$ be a non-trivial path in $G$ with both endpoints in $U_1$ and let $\mathcal{P}_1 := \lbrace P \rbrace$. Thus $R_{\mathcal{U}_1}(\mathcal{P}) \cong C_1$.
Now suppose, for some $i <t$, we have obtained a $\mathcal{U}_i$-anchored path system $\mathcal{P}_i$ in $G$ such that $R_{\mathcal{U}_i}(\mathcal{P}_i) \cong C_i$.
Without loss of generality, we may assume that this cycle is $U_1 U_2 \ldots U_i$.
So $\mathcal{P}_i$ consists of $i$ paths $P_1, \ldots, P_i$ where $P_j$ has endpoints $x_j \in U_j, y_{j+1} \in U_{j+1}$ (with indices modulo $i$).

Suppose that there is some path $P_j \in \mathcal{P}_i$ with $|V(P_j) \cap U_{i+1}| \geq 2$.
Let $u, v \in V(P_j) \cap U_{i+1}$ be distinct such that $u$ is closer than $v$ to $x_j$ on $P_j$.
Let $\mathcal{P}_{i+1}$ be the path system obtained from $\mathcal{P}_i$ be replacing $P_j$ with the paths $x_jP_j u, vP_jy_{j+1}$.

So we may assume that $|V(\mathcal{P}_i) \cap U_{i+1}| = \sum_{1 \leq j \leq i}|V(P_j) \cap U_{i+1}|  \leq i$.
Let $U_{i+1}' := U_{i+1} \setminus V(\mathcal{P}_i)$.
Note that $|U_{i+1}'| \geq 2t-i > t$.
By Menger's Theorem, there exists a path system $\mathcal{R}$ consisting of $i+1$ paths which join $V(\mathcal{P}_i)$ to $U_{i+1}'$ and have no internal vertices in $V(\mathcal{P}_i)$.
By the pigeonhole principle, there exist $j \leq i$ and distinct paths $xRy, x'R'y' \in \mathcal{R}$ such that $x,x' \in V(P_j)$.
Without loss of generality, $x$ is closer to $x_j$ on $P_j$ than $x'$.
Obtain $\mathcal{P}_{i+1}$ from $\mathcal{P}_i$ by replacing $P_j$ with $x_jP_jxRy$, $y'R'x'P_jy_{j+1}$.

In both cases, $\mathcal{P}_{i+1}$ is a $\mathcal{U}_{i+1}$-anchored path system, and
$$
R_{\mathcal{U}_{i+1}}(\mathcal{P}_{i+1}) = U_1 \ldots U_{j} U_{i+1} U_{j+1} \ldots U_i
$$
is a cycle with vertex set $\mathcal{U}_{i+1}$. The path system $\mathcal{P}_t$ obtained in this way is as required in the proposition.
\end{proof}

Now we show that, if $R_{\mathcal{U}}(\mathcal{P})$ is an Euler tour, we can discard suitable subpaths of each $P \in \mathcal{P}$ to ensure that $|V(\mathcal{P}) \cap U|$ is small for each $U \in \mathcal{U}$.

\begin{proposition} \label{clean}
Let $\mathcal{U}$ be a collection of disjoint non-empty vertex-subsets of a graph $G$ and let $\mathcal{P}$ be a $\mathcal{U}$-anchored path system in $G$ containing $t$ paths such that $R_{\mathcal{U}}(\mathcal{P})$ is an Euler tour.
Then there exists a $\mathcal{U}$-anchored path system $\mathcal{P}'$ in $G$ such that $R_{\mathcal{U}}(\mathcal{P}')$ is an Euler tour, and for each $U \in \mathcal{U}$ we have that $|V(\mathcal{P}') \cap U| \leq 2t$. 
\end{proposition}

\begin{proof}
Let  $s:=|\mathcal{U}| $. Clearly, the proposition holds if $s=1$. So we may assume that $s\ge 2$ and
that no $P \in \mathcal{P}$ has both endpoints in the same $X \in \mathcal{U}$ (otherwise we could remove $P$ from $\mathcal{P}$). 
Fix a path $P \in \mathcal{P}$ with endpoints $u \in U, v \in V$ where $U,V \in \mathcal{U}$ are distinct.
We will define a sequence of path systems $\mathcal{R}_\ell$ with $E(\mathcal{R}_\ell) \subseteq E(P)$ as follows.
Let $\mathcal{R}_0 := \lbrace P \rbrace$.
Suppose, for some $0 \leq \ell < s$, we have already defined a path system $\mathcal{R}_\ell$ such that
\begin{itemize}
\item[$(\alpha_\ell)$] $\mathcal{R}_\ell$ is $\mathcal{U}$-anchored;
\item[$(\beta_\ell)$] if $\ell \geq 1$ then $E(\mathcal{R}_\ell) \subseteq E(\mathcal{R}_{\ell-1})$;
\item[$(\gamma_\ell)$] $E(R_{\mathcal{U}}(\mathcal{R}_\ell))$ forms a walk from $U$ to $V$;%
	\COMMENT{We can actually write walk instead of path for this new proof.
Before we have "$R_{\mathcal{U}}(\mathcal{R}_\ell)$ consists of a path between $U$ and $V$ and possibly some isolated vertices;"
}
\item[$(\delta_\ell)$] for at least $\ell$ of the $X$ in $\mathcal{U}$, $|X \cap V(\mathcal{R}_{\ell})| \le 2$.
\end{itemize}

Now we obtain $\mathcal{R}_{\ell+1}$ from $\mathcal{R}_\ell$ as follows.
We are done if there are at least $\ell+1$ sets $X$ in $\mathcal{U}$ such that $|X \cap V(\mathcal{R}_{\ell})| \le 2$, so suppose not.
Let $W \in \mathcal{U}$ be such that $|W \cap V(\mathcal{R}_{\ell})| \ge 3$.
By ($\gamma_\ell$), there exists an integer $p \ge 1$ such that $R_{\mathcal{U}}(\mathcal{R}_\ell)$ equals the walk $U_1 U_2 \ldots U_{p+1}$ from $U_1 := U$ to $U_{p+1} := V$.
So $\mathcal{R}_\ell$ consists of $p$ paths $R_1, \ldots, R_p$ such that $R_j$ has endpoints $x_j \in U_j$ and $y_{j+1} \in U_{j+1}$.
Choose $j \leq j'$ such that $W \cap V(R_j)  \ne \emptyset  \ne W \cap V(R_{j'})$ and $j'-j$ is maximal with this property.
Let $w \in W$ be the vertex on $R_j$ which is closest to $x_j$
and let $w' \in W$ be the vertex on $R_{j'}$ which is closest to $y_{j'+1}$.%
	\COMMENT{Since $|W \cap V(\mathcal{R}_{\ell})| \ge 3$, $w \ne w'$.}
Let $\mathcal{R}_{\ell+1}:=\lbrace R_1, \ldots, R_{j-1}, x_jR_jw, w'R_{j'}y_{j'+1}, R_{j'+1}, \ldots, R_p \rbrace$.%
	\COMMENT{Note that $w \ne x_j$ unless $j = 1$ or else we should have chosen $j-1$ instead of $j$. 
	Similarly, $w \ne x_{j'+1}$ unless $j' = p$.
	So at most one of $x_jR_jw$ and $w'R_{j'}y_{j'+1}$ is a trivial path as $U \ne V$. But it doesn't matter if $\mathcal{R}_{\ell+1}$ contains trivial paths.}
Certainly $\mathcal{R}_{\ell+1}$ satisfies $(\beta_{\ell+1})$ and $(\delta_{\ell+1})$ from the construction.
$(\alpha_{\ell+1})$ follows from $(\alpha_\ell)$.
Since $w,w'$ lie in the same set in $\mathcal{U}$, $(\gamma_{\ell+1})$ holds by $(\gamma_\ell)$.

Therefore we can obtain $\mathcal{P}_P := \mathcal{R}_s$ that satisfies $(\alpha_s)$--$(\delta_s)$.
We can obtain $\mathcal{P}_P$ independently for each $P \in \mathcal{P}$.
Since the $P$ are vertex-disjoint and $(\beta_s)$ implies that $E(\mathcal{P}_P) \subseteq E(P)$, it follows that
$\mathcal{P}' := \bigcup_{P \in \mathcal{P}}\mathcal{P}_P$ is a path system.
Moreover $\mathcal{P}'$ is certainly $\mathcal{U}$-anchored by $(\alpha_s)$.
We write $R := R_{\mathcal{U}}(\mathcal{P})$ and $R' := R_{\mathcal{U}}(\mathcal{P}')$.
Note ($\gamma_s)$ implies that one can obtain $R'$ from $R$ by replacing each edge $UV$ of $R$ with a walk joining $U,V$.
Since $R$ is an Euler tour we therefore have that $R'$ is an Euler tour.
Moreover, $(\delta_s)$ implies that for each $X \in \mathcal{U}$ we have 
$
|V(\mathcal{P}') \cap X| = \sum_{P \in \mathcal{P} } |V(\mathcal{P}_P) \cap X|  \leq 2t
$
as required.
\end{proof}

In the following proposition, 
we show that, given a weak robust subpartition $\mathcal{U}$ in a $t$-connected
graph $G$, we can adjust $\mathcal{U}$ slightly so that $G$ contains a path system $\mathcal{P}$ which is a $\mathcal{U}$-tour.
For this, we simply apply Propositions~\ref{compconn} and~\ref{clean} to obtain a suitable $\mathcal{U}$-anchored path system
and remove a small number of vertices from each bipartite robust component.

\begin{proposition} \label{messabout}
Let $t,n \in \mathbb{N}$ and let $0 < 1/n \ll \rho \ll \nu \ll \tau \ll \eta,1/t \le 1$.
Suppose that $G$ is a regular $t$-connected  
graph on $n$ vertices. Let $\mathcal{U}$ be a weak robust subpartition in $G$ with parameters $\rho,\nu,\tau,\eta,k,\ell$ where $k+\ell \leq t$.
Then
\begin{itemize}
\item[(i)] $G$ has a weak robust subpartition $\mathcal{X}$ with parameters $6\rho,\nu/2,2\tau,\eta/2,k,\ell$;
\item[(ii)] $|\bigcup_{X \in \mathcal{X}}X| \geq |\bigcup_{U \in \mathcal{U}}U| - 2\rho \ell n$;
\item[(iii)] $G$ contains an $\mathcal{X}$-tour with parameter $54\rho$.
\end{itemize}
\end{proposition}

\begin{proof}
Write $\mathcal{U} = \lbrace U_1, \ldots, U_k, Z_1, \ldots, Z_\ell \rbrace$ satisfying (D1$'$)--(D5$'$).
Apply Proposition~\ref{compconn} to $\mathcal{U}$ with $t' := k+\ell$ playing the role of $t$ to obtain a $\mathcal{U}$-anchored path system $\mathcal{P}^*$ such that $R_{\mathcal{U}}(\mathcal{P}^*) \cong C_{t'}$.
Since $\mathcal{P}^*$ contains at most $t$ paths, we may apply Proposition~\ref{clean} to $\mathcal{P}^*$ to obtain a $\mathcal{U}$-anchored path system $\mathcal{P}$ such that $R_{\mathcal{U}}(\mathcal{P})$ is an Euler tour and $|V(\mathcal{P}) \cap U| \leq 2t$ for all $U \in \mathcal{U}$.

Consider any $1 \leq j \leq \ell$. Let $A_j,B_j$ be the bipartition of $Z_j$ guaranteed by (D3$'$).
So $Z_j$ is a bipartite $(\rho,\nu,\tau)$-robust expander component with respect to $A_j,B_j$.
Moreover,
$$
2e_{\mathcal{P}}(A_j) + e_{\mathcal{P}}(A_j,\overline{Z_j}) \leq \sum_{x \in Z_j}d_{\mathcal{P}}(Z_j) \leq 2|V(\mathcal{P}) \cap Z_j| \leq 4t.
$$
A similar inequality holds for $B_j$.
Now $||A_j|-|B_j|| \leq \rho n$ by (D3$'$), (B1) and (C2).
Therefore we can remove at most $\rho n + 4t \leq 2\rho n$ vertices from $Z_j \setminus V(\mathcal{P})$ to obtain $A_j' \subseteq A_j$, $B_j' \subseteq B_j$ and $Z_j' := A_j' \cup B_j'$ such that
\begin{equation} \label{baleq}
2e_{\mathcal{P}}(A_j') - 2e_{\mathcal{P}}(B_j') + e_{\mathcal{P}}(A_j',\overline{Z_j'}) - e_{\mathcal{P}}(B_j',\overline{Z_j'}) = 2(|A_j'|-|B_j'|).
\end{equation}
To see this, it suffices to check that $e_{\mathcal{P}}(A_j',\overline{Z_j'}) - e_{\mathcal{P}}(B_j',\overline{Z_j'})$ (and thus the left-hand side of (\ref{baleq})) is even.
To verify the latter note that, modulo two, $e_{\mathcal{P}}(Z_j',\overline{Z_j'}) \equiv \sideset{}{_{\mathcal{P}}}\End(Z_j') = d_{R_{\mathcal{U}}(\mathcal{P})}(Z_j')$.
So
$$
e_{\mathcal{P}}(A_j',\overline{Z_j'}) - e_{\mathcal{P}}(B_j',\overline{Z_j'}) = e_{\mathcal{P}}(Z_j',\overline{Z_j'}) - 2e_{\mathcal{P}}(B_j',\overline{Z_j'}) \equiv 
d_{R_{\mathcal{U}}(\mathcal{P})}(Z_j') - 2e_{\mathcal{P}}(B_j',\overline{Z_j'}) \equiv 0
$$
where the final congruence follows since $R_{\mathcal{U}}(\mathcal{P})$ is an Euler tour.
Therefore (\ref{baleq}) can be satisfied.

Let $\mathcal{X} := \lbrace U_1, \ldots, U_k, Z_1', \ldots, Z_\ell' \rbrace$.
Clearly (ii) holds.
To prove (i),
first note that for each $1 \leq j \leq \ell$, we have $|A_j' \triangle A_j| + |B_j' \triangle B_j| = |Z_j \setminus Z_j'| \leq 2\rho n$.
Then Lemma~\ref{BREadjust}(i)%
\COMMENT{need $G$ to be regular for this. If we write $D$ for the degree of $G$ then $D \geq \eta n$ by (D4$'$). }
implies that $G[Z_j']$ is a bipartite $(6\rho,\nu/2,2\tau)$-robust expander component of $G$ with bipartition $A_j', B_j'$.
So (D3$'$) holds.
The remaining properties (D1$'$), (D2$'$), (D4$'$) and (D5$'$) are clear.

Finally, by (\ref{baleq}) and the properties of $\mathcal{P}$ stated above, we can apply Lemma~\ref{balextend} with $\mathcal{X},\mathcal{P},6\rho,\nu/2,2\tau,\eta/2,k,\ell$ playing the roles of $\mathcal{U},\mathcal{P},\rho,\nu,\tau,\eta,k,\ell$ to obtain an $\mathcal{X}$-extension $\mathcal{P}'$ of $\mathcal{P}$ in $G$ that is an $\mathcal{X}$-tour with parameter $54\rho$.
This proves~(iii).
\end{proof}

We are now able to prove Theorem~\ref{tconnected}.

\medskip
\noindent
\emph{Proof of Theorem~\ref{tconnected}.}
Let $\alpha := 1/r+\eps$ and $\eta := 1/2r^2 \leq \alpha^2/2$.
Choose a non-decreasing function $f : (0,1) \rightarrow (0,1)$ with $f(x) \leq x$ for all $x \in (0,1)$ such that the requirements of
Propositions~\ref{fewstructs},~\ref{WRSD-RD} and~\ref{messabout} as well as Lemma~\ref{HES}
are satisfied whenever $n,\rho,\gamma,\nu,\tau$ satisfy the following:
\begin{equation} \label{hierarchy2}
1/n \leq f(\rho);\ \ \rho \leq f(\nu),\eps^3/8;\ \ \nu \leq f(\tau);\ \ \tau \leq f(\eta),f(1/t),f(1/r);
\end{equation}
as well as $1/n\le f(\gamma)$ and $\gamma \leq f(\nu)$.
Choose $\tau,\tau'$ so that
\begin{equation} \label{tau}
0 < \tau' \leq \tau \leq \frac{1}{2r^2}, \frac{\eps}{2t}, \frac{\eps^3}{8}, \frac{f(1/t)}{2}, \frac{f(\eta/2)}{2},f(1/r)\ \ \mbox{ and }\tau' \leq f(\tau).
\end{equation}
Choose a non-decreasing function $f' : (0,1) \rightarrow (0,1)$ such that $54f'(x) \leq f(x/2)$ for all $x \in (0,1)$.
Apply Theorem~\ref{structure} with $f',\alpha,\tau'$ playing the roles of $f,\alpha,\tau$ to obtain an integer $n_0$.
Let $G$ be a $t$-connected $D$-regular graph on $n \geq n_0$ vertices where $D \geq \alpha n$.
Theorem~\ref{structure} now guarantees
$\rho,\nu,k',\ell'$ with
\begin{equation} \label{f'}
1/n_0 \leq \rho \leq \nu \leq \tau',\ \ 1/n_0 \leq f'(\rho)\ \ \mbox{ and }\ \ \rho \leq f'(\nu)
\end{equation} 
such that $G$ has a robust partition $\mathcal{V}$ with parameters $\rho,\nu,\tau',k',\ell'$
(and thus also with parameters $\rho,\nu,\tau,k',\ell'$).
Note that (\ref{tau}) and (\ref{f'}) together imply that (\ref{hierarchy2}) holds.
Moreover,
\begin{equation} \label{rho}
2\rho \leq 1/r^2\ \ \mbox{ and }\ \ 2\rho t \leq \eps.
\end{equation}

\medskip
\noindent
\textbf{Claim.}
\emph{
There are integers $k,\ell$ with $k+\ell \leq t$ such that
$G$ has a weak robust subpartition $\mathcal{U}$ with parameters $\rho,\nu,\tau,\eta,k,\ell$ where
\begin{equation} \label{largesize}
\sum_{U \in \mathcal{U}}|U| \geq \min\left\lbrace\frac{t}{r-1} + \frac{\ell}{r^2}, 1\right\rbrace n.
\end{equation}
}

\medskip
\noindent
To see this, recall that $\mathcal{V}$ is a robust partition in $G$ with parameters $\rho,\nu,\tau,k',\ell'$.
Let $m := k' + \ell'$.
Suppose first that $m \leq t$.
Since by Proposition~\ref{WRSD-RD}(i), $\mathcal{V}$ is a weak robust subpartition in $G$ with parameters $\rho,\nu,\tau,\eta,k',\ell'$ we can take $\mathcal{U} := \mathcal{V}$
(and so $k = k'$ and $\ell = \ell'$).
Therefore we may assume that $t \leq m-1$.
Order the members of $\mathcal{V}$ as $X_1, \ldots, X_{m}$ so that $|X_1| \geq \ldots \geq |X_{m}|$.
Let $\mathcal{U} := \lbrace X_1, \ldots, X_t \rbrace$.
Now by Proposition~\ref{WRSD-RD}(i) and~(ii) there exist integers $k,\ell$ so that $k+\ell = t$ and $\mathcal{U}$ is a weak robust subpartition in $G$ with parameters $\rho,\nu,\tau,\eta,k,\ell$.

By averaging, we have that $\sum_{U \in \mathcal{U}}|U| \geq tn/m$.
Note also that $m + \ell \leq m + \ell' = k' + 2\ell' \leq r-1$ where the last inequality follows from Proposition~\ref{fewstructs}.
Therefore
$$
\sum_{U \in \mathcal{U}}|U| \geq \frac{tn}{m} \geq \frac{tn}{r-1-\ell} = \frac{tn}{r-1} \left( 1 + \frac{\ell}{r-1-\ell} \right) \geq \frac{tn}{r-1} + \frac{\ell n}{r^2},
$$
proving the claim.

\medskip
\noindent
Apply Proposition~\ref{messabout} to $G,\mathcal{U}$ to obtain a weak robust subpartition $\mathcal{X}$ with parameters $6\rho,\nu/2,2\tau,\eta/2,k,\ell$ in $G$ and an $\mathcal{X}$-anchored path system $\mathcal{P}$ such that $\sum_{X \in \mathcal{X}}|X| \geq \sum_{U \in \mathcal{U}}|U| - 2\rho \ell n$ and $\mathcal{P}$ is an $\mathcal{X}$-tour with parameter $\gamma := 54\rho$.
Now (\ref{tau}) and (\ref{f'}) imply that
\begin{align} \label{hierarchyhave}
1/n \leq f(6\rho),f(\gamma);\ \ 6\rho,\gamma \leq f(\nu/2);\ \ \nu/2 \leq f(2\tau);\ \ 2\tau \leq f(\eta/2).
\end{align}
Then Lemma~\ref{HES} with $\mathcal{X}, \mathcal{P}, 6\rho,\gamma,\nu/2,2\tau,\eta/2$ playing the roles of $\mathcal{U},\mathcal{P},\rho,\gamma,\nu,\tau,\eta$ implies that there is a cycle $C$ in $G$ which contains every vertex in $\bigcup_{X \in \mathcal{X}}X$.
So
\begin{eqnarray*}
|V(C)| &\geq& \sum_{U \in \mathcal{U}}|U| - 2\rho \ell n \stackrel{(\ref{largesize})}{\geq} \min\left\lbrace \frac{t}{r-1} + \frac{\ell}{r^2} -2\rho \ell, 1 - 2\rho \ell \right\rbrace n\\
&\stackrel{(\ref{rho})}{\geq}& \min\left\lbrace \frac{t}{r-1}, 1- \eps \right\rbrace n,
\end{eqnarray*}
as required.
\hfill$\square$

\medskip
\noindent\emph{Proof of Theorem~\ref{biptconnected} (Sketch).} The proof is almost the same as that of Theorem~\ref{tconnected}.
We proceed similarly as we did there to obtain a robust partition $\mathcal{V}$ with parameters $\rho,\nu,\tau',k',\ell'$.
Using that $G$ is bipartite, it is easy to check that $k'=0$.
Thus $\ell'\le \lfloor (r-1)/2 \rfloor = (r-2)/2$ by Proposition~\ref{fewstructs}. Instead of the claim in the proof of Theorem~\ref{tconnected}, we now show that there
exists an integer $\ell\le t$ such that $G$ has a weak robust subpartition $\mathcal{U}$ with parameters $\rho,\nu,\tau,\eta,0,\ell$ where
$
\sum_{U \in \mathcal{U}}|U| \geq \min\left\lbrace 2tn/(r-2), n\right\rbrace .
$
(Using that $\ell'\le (r-2)/2$, this follows as in the claim.) The remainder of the proof is now similar to that of Theorem~\ref{tconnected}.
\hfill$\square$


\medskip

{\footnotesize \obeylines \parindent=0pt

Daniela K\"uhn, Allan Lo, Deryk Osthus, Katherine Staden
School of Mathematics
University of Birmingham
Edgbaston
Birmingham
B15 2TT
UK
}
\begin{flushleft}
{\it{E-mail addresses}:
\tt{\{d.kuhn, s.a.lo, d.osthus, kls103\}@bham.ac.uk}}
\end{flushleft}

\end{document}